\documentclass[reqno]{amsart}
\usepackage[top=1.1in, bottom=1in, left=1in, right=1in]{geometry}
\usepackage{amsfonts}
\usepackage{amssymb}
\usepackage{amsthm}
\usepackage{amsmath}
\usepackage{mathrsfs}
\usepackage{color}
\usepackage{enumerate}
\usepackage[numbers,sort&compress]{natbib}
\usepackage{pdfsync}
\usepackage{esint}
\usepackage{graphicx}
\usepackage{float}
\usepackage{caption}
\usepackage{subfigure}
\usepackage{stmaryrd}
\usepackage[titletoc]{appendix}

\allowdisplaybreaks

\usepackage{tikz} 
\usetikzlibrary{calc}
\usetikzlibrary{intersections}
\usetikzlibrary{patterns}
\usetikzlibrary{positioning, shapes.geometric}

\usepackage[colorlinks,
linkcolor=black,       
anchorcolor=blue,      
citecolor=blue,        
]{hyperref}

\makeatother

\numberwithin{equation}{section}

\newtheorem{theorem}{Theorem}[section]
\newtheorem{definition}[theorem]{Definition}

\newtheorem{lemma}[theorem]{Lemma}
\newtheorem{remark}[theorem]{Remark}

\newtheorem{proposition}[theorem]{Proposition}
\newtheorem{corollary}[theorem]{Corollary}
\numberwithin{equation}{section}

\newcommand*{\Id}{\ensuremath{\mathrm{Id}}}
\newcommand*{\supp}{\ensuremath{\mathrm{supp\,}}}

\newcommand{\T}{{\mathbb{T}}}

\newcommand{\rr}{\mathring{R}}
\newcommand{\ru}{\mathring{R}_q}

\newcommand{\p}{\partial}
\renewcommand{\P}{\mathbb{P}}
\renewcommand{\div}{{\mathrm{div}}}
\newcommand{\curl}{{\mathrm{curl}}}
\renewcommand{\u}{{u_q}}

\renewcommand{\d}{{\rm d}}

\newcommand{\norm}[1]{\lVert#1\rVert}

\newcommand{\la}{\lambda_q}
\newcommand{\laq}{\lambda_{q+1}}

\newcommand{\rs}{r_{\perp}}
\newcommand{\rp}{r_{\parallel}}

\newcommand{\zq}{z_{q}}

\newcommand{\R}{{\mathbb R}}
\newcommand{\lbb}{\overline{\lambda}}

\def\a{{\alpha}}

\def\lbb{\lambda}
\def\wt{\widetilde}
\def\9{{\infty}}
\def\ve{{\varepsilon}}
\def\na{{\nabla}}

\def\({\left(}
\def\){\right)}
\def\tq{{\theta_q}}
\def\tq1{{\theta_{q+1}}}

\def\E{\mathbf{E}}
\def\j{J_q}

\def\cq{\chi_{q+1}}

\begin{document}
	
\title[] {Existence and non-uniqueness of weak solutions with continuous energy to the 3D deterministic and stochastic Navier-Stokes equations}

\author{Alexey Cheskidov}
\address{Institute for Theoretical Sciences, Westlake University, China.}
\email[Alexey Cheskidov]{cheskidov@westlake.edu.cn}
\thanks{}

\author{Zirong Zeng}
\address{School of Mathematics, Nanjing University of Aeronautics and Astronautics, China.}
\email[Zirong Zeng]{beckzzr@nuaa.edu.cn}
\thanks{}

\author{Deng Zhang}
\address{School of Mathematical Sciences, CMA-Shanghai, Shanghai Jiao Tong University, China.}
\email[Deng Zhang]{dzhang@sjtu.edu.cn}
\thanks{}



\begin{abstract}
The continuity of the kinetic energy is an important property of incompressible viscous fluid flows.
We show that for any prescribed finite energy divergence-free initial data there exist infinitely many global in time weak solutions with smooth energy profiles to both the 3D deterministic and stochastic incompressible Navier-Stokes equations. In the stochastic case the constructed solutions are probabilistically strong. 

Our proof introduces a new backward convex integration scheme with delicate selections of initial relaxed solutions, backward time intervals, and energy profiles. Our initial relaxed solutions satisfy a new time-dependent frequency truncated NSE, different from the usual approximations as it decreases the large Reynolds error near the initial time, which plays a key role in the construction.
\end{abstract}

\maketitle

{\small
	\tableofcontents
}
\section{Introduction}

\subsection{Deterministic Navier-Stokes equations}

We consider the 3D Navier-Stokes equations (NSE) on the torus $\T^3:=[0,1]^3$,
\begin{equation}\label{equa-NS-deter}
	\left\{\aligned
	& \partial_t v - \nu \Delta v+(v\cdot \nabla )v  + \nabla P = 0, \\
 & \div\, v = 0,\\
 & v(0,x)=v_0(x), \\
	\endaligned
	\right.
\end{equation}
where $v:[0,\9)\times\T^3\rightarrow \R^3$ is the velocity field,
$P:[0,\9)\times\T^3\rightarrow \R$ stands for the pressure of the fluid,
and $\nu$ is the positive viscosity coefficient.

The global in time existence of classical solution to the NSE is an outstanding open problem even for smooth initial data. On the other hand, weak solutions, introduced by Leray \cite{leray1934} in 1934, exist globally in time for any finite energy initial data, i.e., $v_0 \in L^2$. These solutions have bounded energy $v \in L^\infty(0,\infty;L^2)$. Moreover, they are continuous in $L^2$ at $t=0$ and satisfy the energy inequality starting from almost every initial time. In particular, solutions constructed by Leray have finite energy dissipation, i.e., $v \in L^2(0,\infty;H^1)$. In bounded domains such solutions were constructed by Hopf \cite{hopf1951}. These weak solutions, usually referred to as the Leray-Hopf solutions, enjoy certain nice properties, such as the weak-strong uniqueness, as they cannot blow up instantaneously. However, it is not known whether the Leray-Hopf solutions are continuous in $L^2$ for all time.

Starting with pioneering work of De Lellis and Sz\'ekelyhidi Jr. \cite{dls09}, numerous new constructions of weak solutions to fluid equations appeared employing the method of convex integration. Displaying some properties observed in turbulent fluid flows, the drawback of these solutions is the infinite energy dissipation.
In the ground breaking paper \cite{bv19b},
Buckmaster and Vicol used convex integration to construct weak solutions in the continuous energy class $C([0,T]; L^2)$ with arbitrary prescribed smooth energy profiles (but not initial data), hence proving the non-uniqueness of the trivial solution $u \equiv 0$. 
For arbitrary prescribed divergence-free initial data in $L^2$, non-unique weak solutions in $C((0,T]; L^2)$ were constructed 
by Buckmaster, Colombo and Vicol \cite{bcv21} for the 3D NSE, and by Burczak, Modena and Sz\'{e}kelyhidi \cite{bms21} for more general power-law flows. However, unless the initial data is smooth enough, all these solutions are not continuous in $L^2$  at the initial time $t=0$, but only weakly continuous, as the energy increases instantaneously. The existence of weak solutions in $C([0,\infty); L^2)$ with  arbitrary prescribed divergence-free initial data in $L^2$ remained an open problem, which we address in the present paper.

We also note that non-uniqueness results can be classified as (see \cite{cl20.2}):
\begin{enumerate}
  \item[$\bullet$]
  ``Weak non-nonuniqueness'': there exists a non-unique weak solution in the class $X$;

  \item[$\bullet$]
  ``Strong non-nonuniqueness'':
  any weak solutions in the class $X$ is non-unique.
\end{enumerate}
Here $X$ denotes a certain functional class of weak solutions.
The non-uniqueness results in \cite{bcv21,bms21}
are strong in the class $C_{weak}([0,T]; L^2) \cap C((0,T]; L^2)$. In this paper, for any divergence-free initial data in $L^2$, we prove the existence and strong non-uniqueness of solutions in the class $C([0,\infty); L^2)$. 

\begin{theorem}[Existence and strong non-uniqueness of solutions with continuous energy]  \label{thm-Nonuniq}
For any divergence-free initial data in $L^2$ there exist infinitely many weak solutions to the 3D NSE \eqref{equa-NS-deter} in the class $C([0,\infty); L^2)$.
\end{theorem}

Theorem \ref{thm-Nonuniq} actually follows from
the following more general result
which,
for any prescribed divergence-free initial data in $L^2$,
gives infinitely many global-in-time continuous energy solutions
with smooth energy profiles.

\begin{theorem} [Global energy solutions to NSE with prescribed initial data] \label{Thm-Non-SNS-Deter} 
For any $T>0$, given any divergence-free initial datum $v_0\in L^2$,
there exist infinitely many distinct smooth functions
$e: [0,T]\rightarrow [0,\9)$ such that $e(0)=\|v_0\|_{L^2_x}$ and
for each such function $e$ the following holds:

\begin{enumerate}
    \item[$\bullet$]
  Global energy solution:
  there exists a global continuous energy solution $v\in C([0,T]; L^2_x)$ to the 3D NSE \eqref{equa-NS-deter}.

  \item[$\bullet$]
  Energy profile:
  \begin{align}\label{energy-1}
 \|v(t)\|_{L^2_x}^2=e(t) \quad\text{for all} \quad t\in [0,T].
\end{align}
\end{enumerate}
\end{theorem}

The proof of Theorem \ref{Thm-Non-SNS-Deter} is based on a new backward integration scheme,
which includes suitably selections of
backward time sequence and a new approximate NSE system, which we call $\Lambda$-Navier-Stokes equations.
We will present some remarks in 
Subsection \ref{subsec-rem} below
to show the flexibility of this method.

\subsection{Stochastic Navier-Stokes equations}

Our strategy also allows to consider more general stochastic Navier-Stokes equations
and construct global-in-time probabilistically strong finite energy solutions with prescribed initial data,
which solves one of the open problems in the field of SPDEs (see \cite{F08, HZZ22.2}).

More precisely,
we consider the 3D stochastic Navier-Stokes equations
driven by additive noise on the torus $\T^3$,
\begin{equation}\label{equa-NS}
	\left\{\aligned
	& \d v+(- \nu \Delta v+(v\cdot \nabla )v  + \nabla P)\d t= \text{d} W_t ,\\
 & \div\, v = 0,\\
 & v(0,x)=v_0(x).\\
	\endaligned
	\right.
\end{equation}
Here, the flow is driven by $GG^*$-Wiener processes $W$
on a stochastic basis $(\Omega,\mathcal{F}, \{\mathcal{F}_t\}, \mathbb{P})$
with the normal filtration $\{\mathcal{F}_t\}_{t\geq 0}$ generated by
$W$, $G$ is a Hilbert-Schmidt operator from a Hilbert space $U$ to $L^2_\sigma(\T^3)$, which is the subspace of $L^2(\T^3)$ containing mean-free and divergence-free functions.
In particular,
when the noise is absent,
\eqref{equa-NS} reduces to
the classical NSE \eqref{equa-NS-deter}.
Initiated by Bensoussan and Temam \cite{BT73}, the stochastic NSE has been extensively studied, but the global existence of probabilistically strong solutions remained a major problem, see \cite{F08}. 

Even though constructing probabilistically strong Leray-Hopf solutions is still out of reach, recent advances in convex integration technique brought to the field of mathematical fluid dynamics by De Lellis and Sz\'{e}kelyhidi \cite{dls09} resulted in 
constructions of non-Leray-Hopf (wild) solutions. 
In \cite{HZZ22.2},  Hofmanov\'a, Zhu, and Zhu  constructed global-in-time probabilistically strong and analytically weak solutions to stochastic NSE for any divergence-free initial data in $L^2$ almost surely. These solutions lie in the unbounded energy class $L^q([0,\infty);L^2)$ almost surely, $1\leq q <\infty$. Based on convex integration schemes, these constructions also show that there is no uniqueness for such solutions.  In a consequent paper \cite{HZZ22.3}, the authors constructed probabilistically strong and analytically weak solutions in the continuous energy class $C([0,\infty); L^2_x)$ almost surely. However,  the initial data can not be prescribed for these solutions. See Subsection~\ref{sec:Previous_Work} for a more detailed overview.

To summarise, even for arbitrary prescribed smooth initial data, the existence of finite energy solutions was not known for the  3D stochastic NSE. We introduce a new backward convex integration scheme to solve this problem. Moreover, using our new $\Lambda$-NSE approximation as the first step in the scheme, we can prove the existence of infinitely many solutions in the continuous energy class $C([0,\infty); L^2)$ for arbitrary prescribed initial data in $L^2_\sigma$ almost surely, see Theorem \ref{Thm-Non-SNS} below. 

\begin{theorem} [Global energy solutions to stochastic NSE with prescribed $L^2_\sigma$ initial data] \label{Thm-Non-SNS}
Let $r\in [1,\infty)$ be fixed. 
Then, for any $T>0$, 
any $\mathcal{F}_0$-measurable and divergence-free initial datum $v_0\in L^2_\sigma$,  $\P$-a.s.,
satisfying $ \|v_0\|_{L^{2}_x} \leq M$ for some 
deteministic constant $M>0$, 
there exist infinitely many distinct smooth functions $e: [0,T]\rightarrow [0,\9)$,
and for each such function $e$, the following holds:
\begin{itemize}
    \item Global energy solution: there exists a global $\{\mathcal{F}_t\}$-adapted solution $v\in L^r(\Omega;C([0,T];L^2_x))$ to stochastic NSE \eqref{equa-NS}.
    \item Energy profile: \begin{align}\label{energy}
 \E \|v(t)\|_{L^2_x}^2= e(t) \quad\text{for all} \quad t\in [0,T].
\end{align}
\end{itemize}
\end{theorem}

As a consequence, 
we obtain non-uniqueness
in the continuous energy class $  C([0,\infty); L^2_x)$ for the 3D stochastic NSE for every divergence-free initial data in $L^2_\sigma$. 

\begin{corollary}[Strong non-uniqueness of continuous energy solutions]  \label{thm-Nonuniq-Stoch}
	For any divergence-free initial datum $v_0\in L_\sigma^2$, $\P$-a.s., there exist infinitely many global-in-time
    probabilistically strong solutions in $C([0,\infty);L^2_\sigma)$, $\P$-a.s.
	to \eqref{equa-NS} with the same
	initial datum $v_0$. 
\end{corollary}

\subsection{Comparison with previous works} \label{sec:Previous_Work}
In recent years, there have been a significant progress 
towards
the understanding of non-uniqueness for 3D Navier-Stokes equations.

\medskip
\paragraph{\it Deterministic NSE}

In the seminal paper \cite{bv19b},
Buckmaster and Vicol used convex integration method, introduced by De Lellis and Sz\'ekelyhidi Jr. \cite{dls09} for the Euler equations, to construct weak solutions in the continuous energy class $C([0,T]; L^2_\sigma)$ with arbitrary prescribed smooth energy profile, hence proving the non-uniqueness of the trivial solution $u \equiv 0$. The space-intermittent convex integration introduced in \cite{bv19b}
has been successfully applied to other viscous fluid models as well.
In \cite{bcv21}, Buckmaster, Colombo and Vicol constructed non-unique weak solutions of the 3D NSE in $C((0,T]; L^2_\sigma)$ for arbitrary prescribed divergence-free initial data in $L^2$. However, unless the initial data was smooth enough, all these solutions were not continuous in $L^2$  at the initial time $t=0$, but only weakly continuous. Remarkably, this result holds not only for the classical 3D NSE, but also for the 3D hyper-dissipative NSE with any power of the Laplacian less than the critical $5/4$ Lions exponent, see also \cite{lt20}.
Later, Burczak, Modena and Sz\'{e}kelyhidi \cite{bms21} considered more general power-law flows that contain NSE as a typical model. For arbitrary prescribed $L^2_\sigma$ initial data they also constructed non-unique weak solutions in $C((0,T]; L^2_\sigma)$, but again only weakly continuous at the initial time $t=0$. As in \cite{bcv21}, the Reynolds stress associated to the initial relaxed solution remained bounded, but did not vanish at the initial time, which resulted in the discontinuity in $L^2$ of the final constructed weak solution at $t=0$ unless the initial data was smooth enough.

Moreover,
the sharp non-uniqueness for NSE near endpoints of Lady\v{z}enskaja-Prodi-Serrin (LPS) criteria was proved in \cite{cl20.2,cl23}.
Afterwards,
this sharpness regularity
with respect to the LPS criteria
was extended to the 3D hyper-dissipative NSE in  \cite{lqzz22}.
See also \cite{lzz21,lzz21.2} for relevant results of MHD equations.

We also would like to mention the recent works \cite{ABC21,ABC22} by
Albritton, Bru\'{e} and Colombo,
where the non-uniqueness of Leray-Hopf solutions was proved for the 3D NSE with external forces, using the unstable vortex construction by Vishik \cite{vis18a,vis18b}. In the unforced case, a method for proving non-uniqueness of self-similar solutions based on instability was 
developed by Jia and \v{S}ver\'{a}k \cite{JS14,JS15}.

For the convenience of readers,
the following table shows the weak/strong non-uniqueness in the existing works.
$$
\begin{array}{lll}
	\hline \\ \text { Results } \qquad& \text { Category } \qquad&  \text { space } \\[2mm]
	\hline \\ \text { \cite{leray1934},\cite{hopf1951} } & \text { Existence } &  C_{weak}([0,T]; L^2)\cap L^2(0,T;\dot{H}^1)  \\[2mm]
	\text{ \cite{ABC21} } & \text { Weak Nonuniqueness with force} &  C([0,T]; L^2)\cap L^2(0,T;H^1) \\[2mm]
 	\text { \cite{bv19r}  } & \text { Weak  Nonuniqueness}  & C([0,T]; H^{\beta
}), \ 0<\beta \ll 1 \\[2mm]
    \text { \cite{bcv21} } & \text { Strong Nonuniqueness } &  C([0,T]; H^\beta), \ 0<\beta \ll 1  \\[2mm]
    \text { \cite{bcv21}, \cite{bms21} } & \text { Strong Nonuniqueness } &  C_{weak}([0,T]; L^2) \cap C((0,T]; L^2)  \\[2mm]
	\text { \cite{cl20.2}} \quad& \text { Strong Nonuniqueness } &  L^\gamma(0,T;L^\9),\ 1\leq \gamma<2 \\[2mm]
	\text { Theorem~\ref{Thm-Non-SNS-Deter} }\qquad & \text { Strong Nonuniqueness }  & C([0,\infty); L^2) \\[2mm]
	\hline
\end{array}
$$

\medskip
\paragraph{\it Stochastic NSE}

The mathematical study of stochastic NSE was initiated by Bensoussan and Temam \cite{BT73}.
Since then there have been considerable efforts to show the existence and uniqueness of weak solutions to 3D stochastic NSE in various settings.
Flandoli-Gatarek \cite{FG95} proved the global existence of Leray-Hopf martingale (probabilistically weak) solutions to stochastic NSE in bounded domains.
The case of unbounded domains was solved by Mikulevicius-Rozovskii \cite{MR05} and Brze\v{z}niak-Motyl \cite{BM13}.
See also \cite{HM06,KNS20} and references therein
for the long-term ergodic behavior of 2D stochastic NSE.
The existing literature highlights the need for a deeper understanding of the probabilistic aspects of weak solutions in stochastic settings.

It has been pointed out by Flandoli
\cite[Page 83]{F08} that ``As in the deterministic case,
strong continuity of trajectories in $H$ is an open problem'', where $H$ is the space $L^2_\sigma$.
Moreover,
it is also mentioned that
(cf. \cite[Page 84]{F08})
``there is no result of existence of strong solutions (with Leray-Hopf regularity, \cite[Definition 4.2]{F08})
for the 3D stochastic Navier-Stokes equation".
Here {\it probabilistically strong solutions}
mean that the solutions correspond to the given Brownian stochastic basis
and have the Leray-Hopf regularity.
This concept is delicately different from that of {\it probabilistically weak/martingale solutions},
which correspond to flexible (not fixed)
Brownian stochastic basis.
We refer to Definitions 4.1 and 4.2 of \cite{F08} for the precise definitions.

Using the method of convex integrations,
Hofmanov\'a, Zhu, and Zhu \cite{HZZ19} constructed local in time (up to a stopping time) probabilistically strong analytically weak solutions and proved
the non-uniqueness in law of weak solutions.
Afterwards, in \cite{HZZ22.2}, for any divergence-free initial data in $L^2$ almost surely, the authors proved the non-uniqueness of global-in-time
probabilistically strong solutions
in the unbounded energy class $L^p_TL^2_x$,
$1\leq p<\infty$. 
Like in \cite{bms21}, 
a Stokes type stochastic solution was also used in \cite{HZZ22.2} as the initial relaxed solution,
and the input of noise gave rise to an explosive Reynolds stress at the initial time.
This leads to stochastic solutions that are not only discontinuous in $L^2$ at the initial time, as in the deterministic case \cite{bms21,bcv21}, but even unbounded in $L^2$.

Hofmanov\'a, Zhu, and Zhu also proved the existence of infinitely many stationary (with shift invariant law) and ergodic solutions \cite{HZZ22.3}. In particular, in \cite{HZZ22.3}, the authors constructed probabilistically strong and analytically weak solutions in the continuous energy class $C([0,\infty); L^2_x)$ almost surely. However,  the initial data can not be prescribed for these solutions. See also \cite{HZZ23-arma} for the non-uniqueness for the stochastic NSE with white noise. Also, the sharp non-uniqueness result near one of the LPS endpoints $L^2_tL^\infty$ \cite{cl20.2} has been extended to the stochastic setting in \cite{cdz22}.

Finally based on the method in \cite{ABC21,ABC22},
the non-uniqueness in law of Leray-Hopf martingale solutions
was recently proved in \cite{BJLZ23,HZZ23} for the case of external stochastic forces.

\medskip
\paragraph{\it Current work}
In the present work, we introduce a new {\it backward convex integration scheme}
for 3D NSE, where one of the main novelties is a new NSE approximation \eqref{equa-nse-2}, which we call the $\Lambda$-NSE.
In this system only low modes below the frequency $\Lambda(t)$ are advected by the velocity, and the dissipative term regularizes solutions above that frequency. Contrary to usual approximations, the 
truncation wavenumber $\Lambda(t)$ is not constant, but a decreasing function that blows up at the origin: $\Lambda(t) \to \infty$ as $t \to \infty$. So the approximation becomes more precise as time approaches zero, which is an essential ingredient of the scheme. Instead of choosing Stokes type solutions
as in \cite{bms21,HZZ22.2} or solutions to the fractional Stokes equation as in \cite{cl23},
we choose solutions to the $\Lambda$-NSE as the initial relaxed solutions
to start our backward convex integration scheme. These solutions are globally smooth and the corresponding nonlinearity vanishes on high modes at the initial time. Hence the $\Lambda$-NSE allows to decrease the problematic Reynolds error near the initial time,
and to obtain the strong continuity of solutions for all time (including $t=0$). Combining this with a delicate selection of a sequence of backward time intervals, energy profiles, and frequency parameters, we design a scheme that is suitable for the rough regularity of $L^2$ initial data. A room between lower and upper bounds on energy errors (see \eqref{asp-0-energy} and \eqref{asp-e-q}) allows to construct solutions with infinitely many distinct smooth energy profiles.

\subsection{Further comments and remarks} \label{subsec-rem}

Let us present some further comments and remarks 
below. 

\medskip 
$(i)$ {\it Regular initial data.}
For initial data with better regularity than $L^2$,
our backward convex integration scheme is flexible to gain regularity for solutions 
on the whole time interval starting from the initial time. 
For instance, when $v_0\in H^3$,
the energy solutions in Theorems 
\ref{thm-Nonuniq} and \ref{Thm-Non-SNS} can be
constructed in the regular class 
$C([0,T]; H^{\beta'})
\cap C^{\beta'}([0,T]; L^2)$
for some $\beta'>0$, 
as in the ususal convex integrations. 
Actually, 
the backward scheme 
further reveals 
the deep relationship between 
the regularity of initial data 
and 
the decay rate of Reynolds stress 
near the initial time, 
as well as the parameters in 
intermittent flows.  

To be more precise,  
when $v_0\in H^3$, 
an improved decay rate for the $L^1$-norm of
the initial Reynolds stress $\rr_0$ 
can be derived, such as
\begin{align*}
    \|\mathring{R}_{0}\|_{L^r_{\Omega}C_{[0,T_*]}L^{1}_x} \lesssim T_*^{k}, 
\end{align*}
for some $k>0$.
Note that
the above decay rate is better than the vanishing property
in \eqref{pro-r0} below for the rough $L^2$ case.
It allows 
to link the frequency of the approximate solution to its amplitude as in the usual case 
(see \eqref{la} below), 
to select the backward time 
in an explicit way as in \eqref{Tq-H3}, 
and thus to construct solutions with 
gained regularity on the whole time interval. 
More detailed proofs are contained in Remarks 
\ref{Rem-decay}, \ref{Rem-parameter} and 
\ref{Rem-regular} below.

\medskip 
$(ii)$
{\it Strong non-uniqueness for hyperdissipative NSE up to the Lions exponent.} 
Let us consider a more general model of 
incompressible hyperdissipative NSE
\begin{equation}  \label{equa-hyperNSE}
	  \partial_t v + \nu (-\Delta)^\alpha v+(v\cdot \nabla )v  + \nabla P = 0, \\
\end{equation}
where $\alpha\geq 1$, ${\text div}\, v =0$. 

A classical result by Lions \cite{lions69} states that, 
if the viscous exponent $\alpha \geq 5/4$, 
equation \eqref{equa-hyperNSE} is 
well-posed in the space $C([0,T]; L^2_\sigma) \cap L^2(0,T; H^1)$ 
for any $L^2_\sigma$ initial data. 
In contrast,
if $\alpha\in [1,5/4)$,
it was proved in \cite{lt20,bcv21} that there exist different weak solutions
in $C([0,T]; L^2_\sigma)$,
and thus, the Lions exponent $5/4$ serves as the threshold for the 
well-posedness in  $C([0,T]; L^2_\sigma)$.

Our scheme also applies to this model 
and reveals that
the non-uniqueness phenomenon indeed occurs in the strong sense
for {\it any} $L^2_x$ initial data
when $\alpha \in [1,5/4)$.
This is in stark contrast to 
the high viscous case $\alpha\geq 5/4$
where the well-posedness holds  
for any $L^2_x$ initial data.

In order to prove this, in \eqref{larsrp}
we modify the parameters for the velocity perturbations as follows
\begin{align*}
    \rp:= \lambda_{q+1}^{-1+2\ve},\ \rs := \lambda_{q+1}^{-1+4\ve},\
	 \lambda := \lambda_{q+1},\ \mu:=\lambda_{q+1}^{\frac32-6\ve},\ \sigma=\lambda_{q+1}^{2\ve},
\end{align*}
where $\ve\in \mathbb{Q}_+$ is chosen sufficiently small such that \begin{align}\label{b-beta-ve-3}
0<\ve<\frac{1}{10} \min\{1, \frac54-\a \},\ \ b> \frac{10^4}{\varepsilon}, \ \
	0<\beta<\frac{1}{100b^{3}}.
\end{align}
Then analogous arguments to the ones in Sections~\ref{Sec-Lambda-NSE}-\ref{Sec-prf-thm}
lead to similar results in Theorem~\ref{Thm-Non-SNS-Deter} for the hyper-dissipative NSE \eqref{equa-hyperNSE} 
up to the Lions exponent.

\medskip 
$(iii)$ {\it Sharpness of the Lions exponent for stochastic NSE.} 
It is also worthwhile to consider the stochastic hyper-dissipative NSE 
\begin{equation}  \label{equa-hyperNSE-stocha}
	 \d v + (\nu (-\Delta)^\alpha v  +(v\cdot \nabla )v   + \nabla P) \d t = \d W_t, 
\end{equation}
under the incompressibility condition. 
On the one hand, 
it was proved in \cite{BJLZ23} 
that equation \eqref{equa-hyperNSE-stocha} is well-posed 
in $L^2(\Omega; C([0,T]; L^2_\sigma))$ when $\alpha \geq 5/4$. 
On the other hand, 
our strategy  also yields 
the strong non-uniqueness of continuous energy solutions
to \eqref{equa-hyperNSE-stocha} 
when $\alpha \in [1,5/4)$.  
Thus, the Lions exponent $5/4$ still serves as the sharp threshold for the
well-posedness of probabilistically strong solutions in the continuous energy class.

Let us also mention that,
for probabilistically weak/martingale solutions in the regular Leray-Hopf class,
the sharpness of Lions exponent for the corresponding well-posedness has been recently proved in \cite{BJLZ23}
in the case of external forces, 
based on the self-similarity method in \cite{ABC21,ABC22}.

\medskip 
$(iv)$ {\it Other viscous models.} The new backward convex integration scheme employing the $\Lambda$-NSE approximation 
applies to other viscous models as well.
For instance,
this method also applies to the classical viscous and resistive magnetodrodynamic (MHD) equations,
and one can derive the strong non-uniqueness
for velocity and magnetic solutions in the continuous energy class. 
It is worth noting that 
MHD has a strong coupling between velocity and magnetic fields,
the nonlinearity is symmetric in the NS component, 
but is anti-symmetric in the Maxwell component. 
Thus, 
the geometry of MHD equations is quite different from that of NSE, 
and requires new constructions of convex integration scheme.
This will be done in a forthcoming work.

\bigskip
{\bf Notations.} For $p\in [1,\infty]$ and $s\in \R$, we set
\begin{align*}
	L^p_x:=L^p(\T^3),\quad W^{s,p}_x:=W^{s,p}(\T^3),\quad H^s_x:=H^s(\T^3).
\end{align*}
We denote by $C([0,T];\mathbb{X})$ the space of continuous functions from $[0,T]$ to $\mathbb{X}$, where $\mathbb{X}$ is a Banach space,
equipped with the norm $\|u\|_{C_T\mathbb{X}}:=\sup_{t\in [0,T]}\|u(t)\|_\mathbb{X}$.
More generally,
for any $I\subseteq [0,T]$,
we write
$\|u\|_{C_I\mathbb{X}} := \sup_{t\in I}
\|u(t)\|_{\mathbb{X}}$.
For $N\in \mathbb{N}_+$ we set
\begin{align*}
	\norm{u}_{C_{T,x}^N}:=\sum_{0\leq m+|\zeta|\leq N}
	\norm{\p_t^m \na^{\zeta} u}_{C([0,T]\times \mathbb{T}^3)},
\end{align*}
where $\zeta=(\zeta_1,\zeta_2,\zeta_3)$ denotes the multi-index
and $\na^\zeta:= \partial_{x_1}^{\zeta_1} \partial_{x_2}^{\zeta_2} \partial_{x_3}^{\zeta_3}$.
Similarly, $C^{\eta}_{[0,T]}:=C^{\eta}([0,T])$, $\eta>0$, denote the temporal H\"older continuous functions.
For $1\leq p<\9$, $1\leq q\leq \9$,  $\eta>0$, $s>0$ we set
\begin{align*}
	& \| u\|_{L^p_{\Omega}C_{T,x}}:= \E(\|u\|_{C_{T,x}}^p )^{\frac1p},\quad
 \| u \|_{L^p_{\Omega}C_T \mathbb{X} }:= \E(  \|u\|_{C_T \mathbb{X}}^p )^{\frac1p},\quad \| u\|_{L^p_{\Omega}C^\eta_{[0,T]} \mathbb{X}}:= \E(  \|u\|_{C^\eta([0,T]; \mathbb{X})}^p )^{\frac1p}.
\end{align*}

Throughout this paper the notation $a\lesssim b$ means that $a\leq C b$ for some absolute constant $C>0$.

\section{Strategy of proof}   \label{Sec-Strategy}

In this section we explain in detail the strategy of the
backward convex integration scheme.
We focus on the more general stochastic NSE \eqref{equa-NS},
as the arguments apply directly to the deterministic NSE \eqref{equa-NS-deter}
by removing the stochastic convolution
and probability moments in the absence of noise. The construction of weak solutions to \eqref{equa-NS}  
is mainly performed on the  finite interval $[0, T]$, 
where $T\in (0,\infty)$, 
the global solutions on $[0,\infty)$ can then be constructed  via gluing arguments.

Since  solutions to the NSE conserve the mean, without loss of generality, we present the proof of Theorems \ref{Thm-Non-SNS-Deter} and \ref{Thm-Non-SNS} in the case where the initial velocity field $v_0$ is of zero mean.

\subsection{$\Lambda$-Navier-Stokes equations} \label{Subsec-Lambda-NSE}

In order to maintain the initial data in the convex integration scheme,
Stokes type solutions were used in \cite{bms21,HZZ22.2}, and fractional Stokes solution were used in \cite{cl23} to generate the initial approximation. Then velocity perturbations were added away from the initial time.
In the stochastic setting, due to the presence of noise
in the relaxed system (see \eqref{equa-nsr} below), the above strategy
gives rise to large Reynolds stress errors near the initial time,
which accumulate to infinity through the iterations.
The resulting solutions thus do not lie in the bounded energy class.

In order to  circumvent this problem,
the key idea in this paper is to use, in the initial step,
solutions to a new approximation of the NSE, which we call the $\Lambda$-NSE, specifically designed to prevent the accumulation of large errors.
More precisely,
we introduce a new type of approximate NSE where only low modes below the frequency $\Lambda$ are advected by the velocity:
\begin{align}\label{equa-nse-2}
	\begin{cases}
		\p_t   u-\nu \Delta  u + \P_{<\Lambda(t)}  ( \P_{<\Lambda(t)}u\cdot \nabla \P_{<\Lambda(t)}u)+\nabla p=0,\\
		\div   u=0,\\
		u(0)=v_0 \in L^2_\sigma.
	\end{cases}
\end{align}
Here $\Lambda: (0,T]\rightarrow \R^+$ is a  strictly decreasing smooth function such that $\Lambda(t) \rightarrow +\infty$ as $t\rightarrow 0^+$, $\Lambda(t) \geq c$ for some constant $c>0$, and
\begin{align}\label{decay-lam}
    \Lambda(t)\leq t^{-\frac18} \quad \text{for} \quad t\in (0, T].
\end{align}
For any $t\in (0,T]$,
the frequency truncated operator $\P_{<\Lambda(t)}$ is defined by
\begin{align}
	\widehat{\P_{<\Lambda(t)}u}(\xi):= \hat u(\xi)\varphi_{\Lambda(t)}(\xi), \ \ \xi\in \mathbb{Z}^3,
\end{align}
where $\varphi_{\Lambda(t)}(\xi):= \varphi(\xi/\Lambda(t))$ for some smooth radial function $\varphi: \R^3 \rightarrow [0,1]$ such that
\begin{align*}
	\varphi(x)=\begin{cases}
		1,\quad |x|\in [0,1),\\
		0,\quad |x|\in [2,+\infty).\\
	\end{cases}
\end{align*}

We note that the $\Lambda$-NSE is different from the usual Galerkin-NSE
as the frequency truncation $\Lambda$ changes in time.
Solutions to the $\Lambda$-NSE
enjoy several useful properties including the energy balance,
smoothness for positive times, and the vanishing of nonlinearity
on high modes at the initial time. The latter one is due to the fact that the trancation frequency $\Lambda(t)$ blows up at the origin, making the approximation more and more precise.
These properties are summarized in Theorem~\ref{thm-fns} below.

\begin{theorem} [Global well-posedness of $\Lambda$-NSE]
\label{thm-fns}
Given any divergence-free and {mean-free} initial datum $v_0\in L^2_x$, $\P$-a.s. such that 
 \begin{align}
     \|v_0\|_{L_x^2}\leq M, 
 \end{align}
  for some deterministic constant $M\geq 1$, the following hold:
 \begin{itemize}
    \item[(i)] 
    $\mathbb{P}$-a.s. there exists a global solution $\wt u$ to \eqref{equa-nse-2} which is adapted and smooth for positive times, continuous at the initial time:
 \begin{align}\label{s-con-v0}
\wt u(t) \rightarrow v_0 \quad \text { strongly in}\quad  L^2_x\quad \text{as}\quad t\to 0^+,
\end{align}
and satisfies the energy balance
	\begin{align}\label{eq-e-2}
		\frac12 \|\wt u(t)\|_{L^2_x}^2 + \nu \int_{0}^{t}\|\nabla \wt u(s)\|_{L^2_x}^2\d s = \frac12 \|v_0\|_{L^2_x}^2, \qquad t\in [0,T].
	\end{align} 
Moreover, for any $1\leq \rho <\infty$, it holds
\begin{align}\label{urho-con}
  \|\wt u -v_0 \|_{ L^\rho({\Omega}; C([0,T_*];L^2_x))} \rightarrow 0\quad \text{as}\quad T_*\rightarrow 0^+.
\end{align}

\item[(ii)] Any 
probabilistically strong and analytically weak solution $v$ to \eqref{equa-nse-2} satisfying the energy inequality
\begin{align}\label{eq-ine-2}
		\frac12 \|v(t)\|_{L^2_x}^2 + \nu \int_{0}^{t}\|\nabla v(s)\|_{L^2_x}^2\d s \leq  \frac12 \|v_0\|_{L^2_x}^2, \quad t\in [0,T],\ \ \mathbb{P}-a.s., 
	\end{align}
coincides with $\wt u$. In particular, the solution $\wt u$ is unique in the Leray-Hopf class.
\item[(iii)] The nonlinearity on high modes  vanishes
at the initial time, i.e.,
for any $1\leq \rho <\infty$, it holds
\begin{align}\label{deacy-non}
\|\P_{\geq \Lambda}(\P_{< \Lambda}\wt u\mathring \otimes \P_{< \Lambda}\wt u)\|_{L^{\rho}({\Omega}; C([0,T_*]; L^1_x))}   \to 0,
\ \ \text{as}\ \ T_*\to 0^+.
\end{align}
 \end{itemize}
\end{theorem}

The proof of Theorem \ref{thm-fns} will be given in Section \ref{Sec-Lambda-NSE} later.

Let us mention that
the decay property \eqref{deacy-non}
is essential to reduce the initial Reynolds stress at the initial time.
More precisely,
if we choose the $\Lambda$-NSE
to initialize the convex integration construction, then
the corresponding Reynolds stress $\rr_0$ mainly has two components:
the high modes of the nonlinearity, and the noise part.
Since the Wiener process starts from zero, the noise part vanishes at the initial time. Then the decay property \eqref{deacy-non} ensures that the rest of the Reynolds stress vanishes at the initial time as well
(see \eqref{pro-r0} below).
Then it  is crucial to select suitable backward time sequence later in 
the convex integration scheme, 
in order to achieve the desired main iteration and, simultaneously,
preserve the initial data in the iterative construction.

\subsection{Relaxed random Navier-Stokes-Reynolds system}
Solutions to stochastic NSE
are taken in the probabilistically strong
and analytically weak sense below.

\begin{definition} \label{weaksolu}
Given any weakly divergence-free and
{mean-free} initial datum $v_0 \in L^2_\sigma$, {$\P$-a.s.},
we say that $v$ is a probabilistically strong and analytically weak solution
to the stochastic NSE \eqref{equa-NS},
if $\P$-a.s.
\begin{itemize}	
	\item For all $t\in[0,T]$, $ v(t,\cdot) $ is divergence free in the sense of distributions and has zero spatial mean.
	\item $v \in C([0,T];  L^2(\T^3))$
	is a continuous $\{\mathcal{F}_t\}_{t\geq 0}$-adapted process in $L^2(\T^3)$.
	\item $v$ solves \eqref{equa-NS} in the sense of distributions, i.e., for any divergence-free function $\varphi \in C_0^{\infty}(\mathbb{T}^3)$, it holds
		\begin{align*}
		\int_{\mathbb{T}^3} v(t)\cdot\varphi\, \d x= \int_{\mathbb{T}^3} v_0\cdot\varphi \,\d x+\int_0^t\int_{\mathbb{T}^3} v\cdot ( \nu \Delta \varphi+ (v\cdot \nabla )\varphi ) \, \d x\d s  +  \int_{\mathbb{T}^3}  W_t \cdot\varphi\, \d x
		\end{align*}
     for each $0\leq t\leq T$.
	\end{itemize}
\end{definition}

We use the random shift
\begin{align}  \label{u-vz-shift}
    u:=v-z
\end{align}
to reformulate equation \eqref{equa-NS}
in the new form
\begin{equation}   \label{equ-sns}
	\left\{\aligned
	&\p_t u-\nu \Delta u + \div((u+z)\otimes(u+z))+\nabla P=0, \\
	&\div\, u = 0,\\
	& u(0)=v_0.
	\endaligned
	\right.
\end{equation}
Here $z$ is the stochastic convolution
$z(t):=\int_{0}^{t}e^{\nu(t-s)\Delta } \d W_s$,
that is, the solution to the linear stochastic equation
with zero initial condition
\begin{equation}\label{spde2}
	\left\{\aligned
	&\d z-\nu\Delta z\d t=\d W_t,\\
	&\div\, z=0,\\
	&z(0)=0.
	\endaligned
	\right.
\end{equation}

Note that in the deterministic case,
the stochastic convolution $z$ can be removed,
and thus the shifted equation \eqref{equ-sns}
is exactly NSE \eqref{equa-NS-deter}.

Let us recall the following regularity result for the stochastic convolution term.

\begin{proposition} (\cite[Proposition 3.1]{HZZ22.2})\label{Prop-noise}
	Suppose that $tr(GG^*)<\infty$. Then, for any $\delta\in (0,1/2)$ and $p\geq 2$, we have
	\begin{align}\label{bdd-ew}
	 \mathbf{E} (\|z(t)\|_{C_TH_x^{1-\delta}}^p
          +\|z\|_{C_T^{\frac 12-\delta} L_x^2}^p )\leq (p-1)^{\frac p2}L^p,
	\end{align}
	where $L(\geq 1)$ depends only on $tr(GG^*)$, $T$ and $\delta$. Moreover,
    we have
		\begin{align}\label{bdd-z-l2}
		\mathbf{E}
   \|z(t)\|_{C_tL^2_x}^p
      \leq C_*T^{\frac {(1-\delta)p}{2}},
	\end{align}
where $C_*$ is a universal constant depending only on $\delta$, $p$ and the semigroup $\{e^{\nu t\Delta}\}$.
\end{proposition}

Let $M\geq 1$ be the given deterministic constant in Theorem \ref{Thm-Non-SNS} such that 
\begin{align}\label{ini-bdd}
    \|v_0\|_{L^2_x}\leq M,\ \ \mathbb{P}-a.s. 
\end{align}

Consider the following relaxed random Navier-Stokes-Reynolds system
for each integer $q\in \mathbb{N}$,
\begin{equation}\label{equa-nsr}
	\left\{\aligned
	&\p_t \u-\nu \Delta \u+ \div((\u+\zq)\otimes(\u+\zq))+\nabla P_q=\div \ru,  \\
	&\div \u = 0,\\
	& \u(0)=v_0.
	\endaligned
	\right.
\end{equation}
Here, $\ru$ represents the Reynolds stress which is a symmetric traceless $3\times 3$ matrix,
and $z_q$ is the frequency truncated version of noise
defined by
\begin{align}\label{def-ziq}
 z_{q}:=\mathbb{P}_{\leq \lambda_{q+1}^{\frac{\ve}{8}}}z.
\end{align}
where $\lambda_{q+1}$ is the frequency parameter given by \eqref{def-la0}-\eqref{def-laq} and $\ve$ is a small parameter given by \eqref{b-beta-ve}.

Note that, by the Sobolev embedding, one has
\begin{align}\label{est-zq}
 \| z_{ q} \|_{L^{p}_{\Omega} C_TL^{\infty}_x} \leq(p-1)^{\frac 12} L\lambda_{q+1}^{\frac{\ve}{8}},
\quad\| z_{ q} \|_{L^{p}_{\Omega} C_T^{\frac12-\delta}  L^{\infty}_x}
 \leq (p-1)^{\frac 12} L \lambda_{q+1}^{\frac{\ve}{4}} .
\end{align}
For every $q\geq  3$, define the amplitude parameter $\delta_{q}$ by 
\begin{equation}\label{delta}
	\delta_{q} :=10^{-3q}.
\end{equation}
We also will choose $\delta_0$, $\delta_1$ and $\delta_2$ later.

\subsection{Backward convex integration scheme}   \label{Subsec-back-CI}
The backward convex integration scheme
includes careful selections of initial relaxed solution,
backward time intervals and energy profiles.

\subsubsection{\bf Selection of initial relaxed solutions} 
As mentioned above,  
we choose the solution $\wt u$ to $\Lambda$-NSE \eqref{equa-nse-2}
as the initial solution to the relaxed random Navier-Stokes-Reynolds system \eqref{equa-nsr}.
Namely, we let 
\begin{align}\label{def-inistep}
    u_0 := \wt u
\end{align}
and derive the corresponding Reynolds stress
\begin{align}
	&\mathring{R}_0 := \P_{\geq \Lambda }(\P_{< \Lambda }\wt u\mathring \otimes \P_{< \Lambda }\wt u)+ (\P_{< \Lambda }\wt u\mathring \otimes \P_{\geq \Lambda}\wt u )+ (\P_{\geq\Lambda }\wt u\mathring \otimes \wt u)+\wt u\mathring \otimes z_0+ z_0\mathring \otimes \wt u+  z_0\mathring\otimes z_0 ,   \label{r0b}
\end{align}
and the pressure term
\begin{align}
    P_0 = p-\frac{1}{3} ( \P_{\geq \Lambda }(|\P_{<\Lambda }\wt u|^2)-|\wt u|^2 +2\wt u\cdot z_0 +| z_0|^2),
\end{align}
where $z_0$ is the frequency truncated stochastic convolution given by \eqref{def-ziq}
with $q=0$.

The important fact is that the corresponding Reynolds stress $\mathring{R}_0$ tends to zero as $t\rightarrow 0^+$.
To be precise, 
by the structure of the initial Reynolds stress in \eqref{r0b},
for any $T_*\in (0,T)$ and $\delta\in (0,1/6)$,
\begin{align}\label{r0-est-1}
	\|\mathring{R}_{0}\|_{L^r_{\Omega}C_{[0,T_*]} L^{1}_x} &\leq \|\P_{\geq \Lambda}(\P_{< \Lambda}\wt u\mathring \otimes \P_{< \Lambda}\wt u) \|_{L^r_{\Omega}C_{[0,T_*]}L^{1}_x} +\|(\P_{< \Lambda}\wt u\mathring \otimes \P_{\geq \Lambda}\wt u )\|_{L^r_{\Omega}C_{[0,T_*]}L^{1}_x} \notag\\
	&\quad + \|(\P_{\geq\Lambda}\wt u\mathring \otimes \wt u) \|_{L^r_{\Omega}C_{[0,T_*]}L^{1}_x} +  \|\wt u \|_{L^{2r}_{\Omega}C_{[0,T_*]}L^{2}_x} \| z \|_{L^{2r}_{\Omega} C_{[0,T_*]}L^{2}_x}+ \| z\|_{L^{2r}_{\Omega}C_{[0,T_*]}L^{2}_x}^{2}   \\
	&\lesssim \|\P_{\geq \Lambda}(\P_{< \Lambda}\wt u\mathring \otimes \P_{< \Lambda}\wt u) \|_{L^r_{\Omega}C_{[0,T_*]}L^{1}_x}+\|\wt u \|_{L^{2r}_{\Omega}C_{[0,T_*]}L^{2}_x}\|\P_{ \geq \Lambda} \wt u \|_{L^{2r}_{\Omega}C_{[0,T_*]}L^2_x}\notag\\
 &\quad+ T_*^\frac{1-\delta}{2} 
 \|\wt u\|_{L_\Omega^{2r}C_{[0,T_*]}L^{2}_x} + T_*^{ 1-\delta},  \notag
\end{align}
where we used the moment bound \eqref{bdd-z-l2} 
in the last step.
Note that, for any $t\in (0,T_*]$,
\begin{align*}
    \|\P_{ \geq \Lambda(t)} \wt u(t)\|_{L^2_x}
    \leq& \|\wt u(t) - v_0 \|_{L^2_x} +  \|\P_{ \geq \Lambda(t)} v_0\|_{L^2_x} \notag  \\
     =&  \|\wt u(t) - v_0 \|_{L^2_x}
           +  \|(1 - \varphi_{\Lambda(t)}) \widehat{v_0} \|_{L^2_\xi} \notag \\
     \leq& \|\wt u(t) - v_0 \|_{L^2_x}
           +  \|(1 - \varphi_{\Lambda(T_*)}) \widehat{v_0} \|_{L^2_\xi},
\end{align*}
where the last step is due to the fact that
$\Lambda$ and $\varphi$ decrease on $(0,T_*]$,
and thus $0\leq 1 - \varphi_{\Lambda(t)} \leq 1 - \varphi_{\Lambda(T_*)}$
for $t\in (0,T_*]$.
Hence, using \eqref{urho-con}
and the fact $1 - \varphi_{\Lambda(T_*)} \to 0$ as $T_*\to 0^+$
we obtain for any $\rho\in[1,\infty)$ fixed, 
\begin{align}\label{u-l2-decay}
    \|\P_{ \geq \Lambda} \wt u\|_{L^{\rho}_{\Omega}C_{[0,T_*]}L^2_x}
    \leq  \|\wt u - v_0 \|_{L^{\rho}_{\Omega}C([0,T_*];L^2_x)}
           +  \|(1 - \varphi_{\Lambda(T_*)}) \widehat{v_0} \|_{L^{\rho}_{\Omega}L^2_\xi}
    \to 0, \ \ \text{as}\ \ T_* \to 0^+.
\end{align} 
Thus, in view of the vanishing of nonlinearity on high modes \eqref{deacy-non} 
and 
the boundedness 
$ \|\wt u\|_{L_\Omega^{\infty}C_{[0,T]}L^{2}_x} 
\leq M$ 
due to the energy balance \eqref{eq-e-2}, 
we derive
the key decay property of Reynolds stress at the initial time
\begin{align}\label{pro-r0}
 \|\mathring{R}_{0}\|_{L^r_{\Omega}C_{[0,T_*]}L^{1}_x} \rightarrow 0\ \ \text{as}\ \ T_*\rightarrow 0^+.
\end{align}

The decay property \eqref{pro-r0} enables us to control the Reynolds stress 
near the initial time in the first iterative step, 
which is different from the previous works \cite{bcv21,bms21, HZZ22.2} 
where the Reynolds stress remains bounded in \cite{bcv21,bms21} 
and explodes at the initial time in \cite{HZZ22.2}.
It allows to select a suitable backward time sequence as well 
so that convex integration constructions can be performed  in a backward manner.

\begin{remark} 
(Regular initial data: improved decay rate of Reynolds stress) 
\label{Rem-decay}
For regular initial data, say, $v_0\in H^3_x$,  $\P$-a.s., 
one has $\wt u\in C_TH^3_x$, $\P$-a.s.
Slightly modifying $\Lambda(t)$ such that 
$\Lambda(t)\geq t^{-\frac{1}{10}}$ in \eqref{decay-lam} and assuming $\|v_0\|_{H^3_x}\leq M$ in \eqref{ini-bdd} for simplicity, 
one can obtain the improved upper bound of the decay estimate
\begin{align} \label{decay-R0-refined}
\|\mathring{R}_{0}\|_{L^r_{\Omega}C_{[0,T_*]}L^{1}_x} \lesssim T_*^{\frac{3}{10}}.
\end{align} 
This fact indeed reveals the deep 
relationship 
between the initial regularity 
and the decay rate of Reynolds stress 
near the initial time. 

To this end, let us assume  $T_*<1$ 
without loss of generality. 
Similarly to \eqref{r0-est-1}, since $\delta\in (0,1/6)$, 
\begin{align} \label{r0-est-1-2}
\|\mathring{R}_{0}\|_{L^r_{\Omega}C_{[0,T_*]} L^{1}_x} 
	&\lesssim \|\P_{\geq \Lambda}(\P_{< \Lambda}\wt u\mathring \otimes \P_{< \Lambda}\wt u) \|_{L^{r}_{\Omega}C_{[0,T_*]}L^{1}_x}+\|\wt u \|_{L^{2r}_{\Omega}C_{[0,T_*]}L^{2}_x}\|\P_{ \geq \Lambda} \wt u \|_{L^{2r}_{\Omega}C_{[0,T_*]}L^2_x}\notag\\
 &\quad+ T_*^\frac{5}{12}\|\wt u\|_{L^{2r}_{\Omega}C_{[0,T_*]}L^{2}_x} + T_*^{\frac56}.  
\end{align} 
Note that, since $\Lambda(t)\geq t^{-\frac{1}{10}}$, 
\begin{align}  \label{r0-est-1-2.1}
    \|\P_{ \geq \Lambda} \wt u \|_{L^{2r}_{\Omega}C_{[0,T_*]}L^2_x}\lesssim \| |\xi|^{-3}(1-\varphi_\Lambda)\|_{C_{[0,T_*]}L^\infty_\xi} \| |\xi|^3 \widehat{\wt u} \|_{L^{2r}_{\Omega}C_{[0,T_*]}L^2_\xi}\lesssim  T_*^{\frac{3}{10}} \|\wt u\|_{L^{2r}_{\Omega}C_{[0,T_*]}H^3_x }.
\end{align}
Concerning the first term on the right-hand side of \eqref{r0-est-1-2}, arguing as in the proof of \eqref{est-non-decay-2} below, 
we derive 
\begin{align}   \label{r0-est-1-2.2}
\|\P_{\geq \Lambda}(\P_{< \Lambda}\wt u\mathring \otimes \P_{< \Lambda}\wt u)\|_{L^{r}_{\Omega}C_{[0,T_*]}L^1_x}
&\lesssim  ( \|\wt u \|_{L^{2r}_{\Omega}C_{[0,T_*]}L^2_x}  \| \P_{[\frac{\Lambda}{6},2\Lambda ]}\wt u\|_{L^{2r}_{\Omega}C_{[0,T_*]}L^2_x}  
+ \| \P_{[\frac{\Lambda}{6},2\Lambda]}\wt u\|_{L^{2r}_{\Omega}C_{[0,T_*]}L^2_x}^2 )\notag\\
&\lesssim T_*^{\frac{3}{10}} \|\wt u \|_{L^{2r}_{\Omega}C_{[0,T_*]}L^2_x} \|\wt u \|_{L^{2r}_{\Omega}C_{[0,T_*]} H^3_x}+ T_*^{\frac{6}{10}} \|\wt u \|_{L^{2r}_{\Omega}C_{[0,T_*]} H^3_x}^2.
\end{align}  
Moreover, taking $\dot H^3$ inner product of \eqref{equa-nse-2} with $\wt u$ and using Kato’s inequality and Poincar\'e's inequality 
we get
\begin{align*}
    \frac12\frac{\d}{\d t}\|\wt u(t)\|_{\dot H^3_x}^2+\nu\| \wt u(t)\|^2_{\dot H^4_x}&\leq \|\nabla \P_{< \Lambda(t)} \wt u(t)\|_{L^\infty_x}\|\P_{< \Lambda(t)} \wt u(t)\|_{H^3_x}^2\lesssim \| \P_{< \Lambda(t)} \wt u(t)\|_{ \dot H^3_x}^3,
\end{align*}
which, via Gronwall’s inequality, yields
\begin{align*}
    \|\wt u(t)\|_{C_T \dot H^3_x}^2\lesssim \|v_0 \|_{\dot H^3_x}^2 \exp\left\{ \int_{0}^T \|\P_{< \Lambda(t)} \wt u(t)\|_{\dot H^3_x} \d t  \right\}.
\end{align*}
Since $\Lambda(t)\leq t^{-\frac18}$ 
for $t\in (0,T]$,  
\begin{align*}  
   \int_0^T \|\P_{<\Lambda(t)}\wt u(t)\|_{\dot H^3_x}  \d t= \int_0^T \| |\xi|^3 \widehat  \varphi_{\Lambda(t)}(\xi) \widehat {\wt u}(t)  \|_{L^2_\xi}  \d t \leq \int_0^T 8\Lambda^3(t) \d t \|u\|_{C_T L^2_x} \lesssim M.
\end{align*}
It follows that 
\begin{align}  \label{r0-est-1-2.3}
\|\wt u(t)\|_{C_T\dot  H^3_x}^2\lesssim e^M \|v_0 \|_{\dot H^3_x}^2\lesssim M^2 e^M,
\quad \P\text{-a.s.},
\end{align}
where the implicit constants are deterministic. 
Thus, plugging \eqref{r0-est-1-2.1}-\eqref{r0-est-1-2.3} into \eqref{r0-est-1-2} we obtain 
the improved estimate \eqref{decay-R0-refined}. 
\end{remark}

\subsubsection{\bf Selection of backward time intervals}\label{subsubsec-tq}

Our next step is to divide the time interval $[0,T]$ into infinitely many sub-intervals
$\{[T_{q+1}, T_q]\}_q$,
where $\{T_q\}$ is a strictly decreasing sequence
satisfying $T_0=T$
and $T_q \to 0^+$ as $q\to \infty$. 
The purpose is to control the Reynolds stress on the whole time interval
for all iterative steps and, simultaneously,
maintain the prescribed initial data. 

In the first step, we add the velocity perturbation on $[T_1,T]$
and use convex integration constructions to decrease the Reynolds error on $[T_1,T]$.
Then, we move to a larger time interval $[T_2,T]$
and decrease the associated Reynolds error.
Continuing this iterative procedure,
we progressively control the
velocity fields and Reynolds errors on $[T_q, T]$, $q\geq 1$, to reach the desired iterative aim.
We note that the initial datum is unchanged in each step,
and thus, $u_q(0)=v_0$.

More delicately,
the backward time sequence $\{T_q\}$ is chosen 
in an iterative way as follows.  

Recall the choice of $\{\delta_q\}_{q\geq 3}$ 
in \eqref{delta}. 
Let $\delta_2$ be a large parameter such that 
\begin{align}  \label{R0-CtL1}
	\|\mathring{R}_{0}\|_{ L^r_{\Omega} C_{[0,T]} L^{1}_x}\leq \frac18c_*\delta_2,
\end{align}
where $c_*>0$ be a small parameter such that 
\begin{align}\label{def-c8}
	0< c_*< \min\left\{1,\frac{\ve_u}{500}\right\},
\end{align}
with $\ve_u$ being the geometric constant given by Lemma \ref{geometric lem 2}. 

Taking advantage of the decay property \eqref{pro-r0} of $\mathring{R}_{0}$
near the initial time, we choose $0<T'_1<T$ and $T_2'\in (0,T'_1/2)$ such that
\begin{align*}
	\|\mathring{R}_{0}\|_{ L^r_{\Omega} C_{[0,T'_1]} L^{1}_x}\leq \frac18c_*\delta_3\quad 
	\|\mathring{R}_{0}\|_{ L^r_{\Omega} C_{[0,T'_2]} L^{1}_x}\leq \frac18c_*\delta_4.
\end{align*}
Then, take $T_1$ and $\ell'_0>0$ such that 
\begin{align*}
	T'_2< T_1:=T'_1-\ell'_0<T'_1,\quad  \ell'_0\leq \frac{T'_1-\ell'_0-T'_2}{20}. 
\end{align*}
We have 
\begin{align} 
	\|\mathring{R}_{0}\|_{ L^r_{\Omega} C_{[0,T_1+\ell'_0]} L^{1}_x}\leq \frac18c_*\delta_3.
\end{align}
Next, we choose $T'_3\in (0,T'_2/2)$ such that 
\begin{align*}
	\|\mathring{R}_{0}\|_{ L^r_{\Omega} C_{[0,T'_3]} L^{1}_x}\leq \frac18c_*\delta_5,
\end{align*}
and $T_2$ and $\ell'_1>0$ such that
\begin{align*}
	T'_3< T_2:=T'_2-\ell'_1<T'_2,\quad  \ell'_1\leq \frac{T'_2-\ell'_1-T'_3}{20}.
\end{align*}
Note that
\begin{align*}
	\ell'_0\leq \frac{T_1-T_2}{20} 
\end{align*}
and
\begin{align}
	\|\mathring{R}_{0}\|_{ L^r_{\Omega} C_{[0,T_2+\ell'_1]} L^{1}_x}\leq \frac18c_*\delta_4.
\end{align}

Proceeding iteratively, for any $q\geq 3$, we choose $T'_{q+1} \in (0,T'_{q}/2)$ such that
\begin{align*}
	\|\mathring{R}_{0}\|_{ L^r_{\Omega} C_{[0,T'_{q+1}]} L^{1}_x}\leq \frac18c_*\delta_{q+3},
\end{align*}
and $T_q$ and $\ell'_{q-1}>0$ such that
\begin{align*}
	T'_{q+1}< T_q:=T'_q-\ell'_{q-1}<T'_q,\quad  \ell'_{q-1}\leq \frac{T'_q-\ell'_{q-1}-T'_{q+1}}{20}.
\end{align*}
We have 
\begin{align}  \label{ell'}
	\ell'_{q-2}\leq \frac{T_{q-1}-T_{q}}{20}, 
\end{align}
and
\begin{align}
	\|\mathring{R}_{0}\|_{ L^r_{\Omega} C_{[0,T_{q}+\ell'_{q-1}]} L^{1}_x}\leq \frac18c_*\delta_{q+2}.
\end{align}
As a consequence, 
we obtain a backward time sequence $\{T_q\}$ and a sequence of small parameters $\{\ell'_q\}$.

Now, let $a$ be a large integer depending on $r$,
$b\in 14\mathbb{N}$ a large integer multiple of $14$,
and $\varepsilon\in \mathbb{Q}_+$ a small parameter,
such that 
\begin{align}\label{b-beta-ve} 
	\varepsilon<10^{-3}\ \text{and}\ b> 100 \varepsilon^{-1}.
\end{align}
Then, for each $q\in \mathbb{N}$, we choose the frequency parameter $\lbb_q$ and mollification parameter $\ell_q$ as follows \begin{align}
	&\lambda_0\geq\max \left\{ a,\, T_2^{-2},\,  (\ell'_0)^{-\frac{1}{30}}  \right\}, \label{def-la0}\\
	&\lambda_q\geq \max \left\{ \lambda_{q-1}^b,\, T_{q+2}^{-2}, \, (\ell'_q)^{-\frac{1}{30}} \right\},\label{def-laq}\\
	&\lambda_q^{\frac17}\in \mathbb{N},\  \ell_q:= \lambda_q^{-30}. \label{def-ell}
\end{align}  
Note that, by the above choice, 
\begin{align}\label{a-big-lambda}
	\ell_{q}\leq \frac{T_{q+1}-T_{q+2}}{20}\ \ \text{and}\ \ 1+T_{q+2}^{-2}\leq \lambda_q,\ \ \forall q\geq 0, 
\end{align}

Let us mention that 
the requirement $\lambda_q \geq T_{q+2}^{-2}$ is to control the growth of the $C_{T,x}^1$-norm of the approximate solutions $u_q$,  
which will be used, 
for instance, 
to prove the iterative estimate \eqref{uc}.  
Moreover, 
the small parameter $\ell'_q$ is incorporated into the definition of $\lambda_q$ to ensure that the mollification and cut-off procedures do not impact the main iteration estimates. We also choose $a$ to be  sufficiently large such that $\lambda_{q}^{-1}\leq \delta_{q+3}$ for every $q\in \mathbb{N}$.
 
Note that by the above construction,  
$\ell_q< \ell'_q$ for all $q\in \mathbb{N}$, 
and so 
\begin{align}
	&\|\mathring{R}_{0}\|_{ L^r_{\Omega} C_{[0,T]} L^{1}_x}\leq \frac18c_*\delta_2, \label{asp-0} \\
	&\|\mathring{R}_{0}\|_{ L^r_{\Omega}C_{[0,T_q+\ell_{q-1}]} L^{1}_x}\leq \frac18c_*\delta_{q+2}. \label{asp-q}
\end{align}

\begin{remark}
(Regular initial data: standard parameters)
\label{Rem-parameter}
	For more regular $H^3_x$ initial data, 
	in view of the improved decay estimate \eqref{decay-R0-refined},
	one can choose the parameters $\lbb_q$,  $\delta_{q}$ and $\ell_q$ as in the usual convex integrations: 
	\begin{equation}\label{la}
		\la=a ^{(b^q)}, \ \
		\delta_{q+3}=\frac12 \lambda_{q+3}^{-2\beta},\ \ \ell_q:=\la^{-30},
	\end{equation}
	where $a$ is a large integer,
	$b\in 2\mathbb{N}$ is a large integer of multiple $2$, $\beta>0$ 
	and $\varepsilon \in \mathbb{Q}_+$ are small parameters
	such that $ \lambda_q^{\frac17}\in \mathbb{N}$
	and
	\begin{align}\label{b-beta-ve-2} 
		b> 10^4 \varepsilon^{-1}, \ \
		0<\beta<\frac{3}{1000b^4}.
	\end{align}
	Then, the backward time can be chosen by 
\begin{align}  \label{Tq-H3}
T_q =\frac12 (\frac18 c_*\delta_{q+ 2})^\frac{10}{3}.
\end{align}

Thanks to the above  explicit choice of parameters, 
we see that for $a$ sufficiently large, but independent of $\{T_q\}$, 
\eqref{a-big-lambda}  
and the inequalities  
$\lbb_0^{-1} \leq (\frac{1}{8}c_*)^{\frac{10}{3}}/ 80$, $ (\frac18 c_*\delta_{q+ 2})^{\frac{10}{3}} /2> \ell_{q-1} $ 
hold. 
Moreover,  by the improved decay estimate \eqref{decay-R0-refined},
	\begin{align*}
		\|\mathring{R}_{0}\|_{L^r_{\Omega}C_{[0,T_q+\ell_{q-1}]}L^{1}_x} \leq  (T_q+\ell_{q-1})^{\frac{3}{10}} \leq \frac18 c_*\delta_{q+ 2}, 
	\end{align*} 
 which verifies \eqref{asp-q}. 
\end{remark}

\subsubsection{\bf Selection of energy profiles}
Regarding the energy of approximate solutions,
we choose a smooth energy profile $e$ and a large enough constant $\delta_1$
such that $e(0)=\E\|v_0\|_{L^2_x}^2$ and
\begin{align}
	\frac34\delta_{2}\leq e(t)-\mathbf{E} \| u_0(t)+z (t)\|_{L^2_x}^2 \leq \delta_{1},\ \
	 & \text{for }\ t\in [T_1-\ell_0,T], \label{asp-0-energy} \\
	 \frac34\delta_{q+2}\leq e(t)-\E\| u_0(t)+z (t)\|_{L^2_x}^2 \leq \delta_{q+1},\ \
      &\text{for }\ t\in [T_{q+1}-\ell_{q},T_q+\ell_{q-1}].   \label{asp-e-q}
\end{align}

The delicate point here is to choose the order $\mathcal{O}(\delta_{q+2})$,
different from the upper bound order $\mathcal{O}(\delta_{q+1})$,
for the lower bound of the energy difference in \eqref{asp-0-energy}-\eqref{asp-e-q}.

As a matter of fact, on the one hand,
the order $\mathcal{O}(\delta_{q+1})$ seems to be the first candidate 
for the lower order.
But this forces the energy difference to be of order
$\mathcal{O}(\delta_{q+1})$ on $[T_{q+1}, T_q]$,
while of order $\mathcal{O}(\delta_{q+2})$ on $[T_{q+2}, T_{q+1}]$.
Since $\delta_{q+2} \ll \delta_{q+1}$,
the resulting energy profile losses continuity and
behaves like step functions on each subintervals
$[T_{q+1}, T_q]$, $q\geq 1$.

On the other hand, in order to apply the Geometric Lemma to derive \eqref{velcancel},
one needs
\begin{align}\label{f-est-e-w}
e(t) -\E(\|u_{q}(t)+z_{q}(t)\|_{L_x^2}^2+\|w_{q+1}(t)\|_{L_x^2}^2) \geq 0,\quad  \forall\ t\in [0,T],\ q\geq0,
\end{align}
where $w_{q+1}$ is the perturbation added on the approximate solution $u_q$ at level $q$.
By the main iterative estimates, the old Reynolds stress
is of order $\mathcal{O}(\delta_{q+2})$ on $[T_{q+1}, T_q]$ and so the $L^2_x$-norm of the perturbation $w_{q+1}$ is at least of order $\mathcal{O}(\delta_{q+2})$ on $[T_{q+1}, T_q]$.
See, e.g., \eqref{asp-q} and \eqref{prin-end}. 
This suggests unreasonable to choose orders strictly lower than $\delta_{q+2}$ for the energy difference on this subinterval. 

It is also worth noting that
the open room between the lower order $\delta_{q+2}$ and
the upper order $\delta_{q+1}$
allows the existence of infinitely many smooth energy profiles
on the whole time interval.

\subsubsection{\bf Main iterative estimates} 
We are now ready to state the main iterative estimates 
in the backward convex integration scheme. 

Fix any $\delta \in (0,1/6)$. 
Assume the following inductive estimates at level $q\in \mathbb{N}$ for every $m\geq 1$: \
\begin{align}
	&  \| u_q \|_{L^{2r}_{\Omega} C_{T}L^2_x}\lesssim  \sum\limits_{n=0}^q \delta_{n}^{\frac12} ,\label{ul2}\\
	&  \| u_q  \|_{ L^{m}_{\Omega}C_{[T_{q+2},T],x}^{1}} \lesssim \lambda_{q}^{4}(12^{q-1}\cdot 8\cdot 6mL^2)^{6\cdot 12^{q-1}},\label{uc} \\
	&  \| \mathring{R}_{q}\|_{L^{r}_{\Omega}C_{T} L^1_x}\leq  c_*\delta_{q+ 2} ,  \label{rl1s}\\
	&  \| \mathring{R}_{q}  \|_{L^{m}_{\Omega} C_{T} L^{1}_x} \lesssim (12^q\cdot 8mL^2)^{12^q},  \label{rl1b}\\
	&  \| \mathring{R}_{q} \|_{L^{m}_{\Omega}C^{\frac12-\delta}_{[T_{q+2},T]}L^1_x}+ \| \mathring{R}_{q}  \|_{L^{m}_{\Omega}C_{[T_{q+2},T]}W^{1-\delta,1}_x} \lesssim \lambda_{q}^5 (12^q\cdot 8mL^2)^{12^q},  \label{rl1b-s}\\
	& \frac12\delta_{q+ 2}\leq e(t)-\E\|(u_q+z_{q})(t)\|_{L^2_x}^2 \leq 2\delta_{q+1},\ \ t\in [T_{q+1},T].  \label{energy-est}
 \end{align}
The above implicit constants are independent of $q$, $m$, $r$ and $L$.

Proposition \ref{Prop-Iterat} below contains the main iterative estimates of
relaxed solutions $(u_q, \mathring{R}_{q})$
which is the heart of the proof of Theorem \ref{Thm-Non-SNS}.

\begin{proposition} [Main iterative estimates]\label{Prop-Iterat}
	There exists a sequence of integers $\{a_q\}$, such that the following holds:
	
	Suppose that for some $q\in \mathbb{N}$, $(\u, \ru)$ solves  \eqref{equa-nsr}
	and satisfies  \eqref{ul2}-\eqref{energy-est}.
	Then, there exists $(u_{q+1}, \mathring{R}_{q+1})$
	which solves \eqref{equa-nsr} and satisfies \eqref{ul2}-\eqref{energy-est} with $q+1$ replacing $q$.
	In addition,  we have
 
	\begin{align}\label{u-B-L2tx-conv}
		\|u_{q+1}-u_{q}\|_{L^{2r}_{\Omega}C_TL^{2}_{x}}\lesssim	 \delta_{q+ 1}^{\frac12}.
	\end{align}
	The implicit constants are independent of $q$, $m$, $r$ and $L$.
\end{proposition}

The proof of the main iterative estimates in  Proposition \ref{Prop-Iterat} occupies Sections 
\ref{Sec-Interm-Flow}-\ref{Sec-prf-thm} below.

\section{Well-posedness of the $\Lambda$-Navier-Stokes equations}  \label{Sec-Lambda-NSE}

This section is devoted to the proof of Theorem~\ref{thm-fns}.

\medskip
\paragraph{\bf $(i)$ Global existence.} The proof mainly proceeds in three steps.
For simplicity, the dependence on the probability argument $\omega$ is omitted.

\medskip
\noindent{\bf $\bullet$ Galerkin approximation.}
Let us first apply the classical Galerkin method to prove the existence of a global solution $u$ to \eqref{equa-nse-2}.
Let $\{e_j\}_{j=1}^\infty$ be an orthogonal basis in $L^2_\sigma$ satisfying $\Delta e_j =\lambda_j e_j$, $j \geq 1$, and $\mathbb{P}_n: L^2_x \rightarrow L^2_\sigma$ be the projection operator defined by
\begin{align}
\P_n u = \sum_{i=1}^{n} (u, e_j)e_j,
\end{align}
where $(\cdot, \cdot)$ denotes the inner product in $L^2_x$.

Consider the approximate system
\begin{align}\label{equa-nse-g}
\begin{cases}
\p_t u_n-\nu  \Delta  u_n + \P_n\P_{<\Lambda(t)}  (\P_{<\Lambda(t)} u_n\cdot\nabla \P_{<\Lambda(t)}u_n)=0,\\
\div   u_n=0,\\
u_n(0)=\P_n v_0,
\end{cases}
\end{align}
where $\Lambda$ is the smooth decreasing temporal function as in
\eqref{equa-nse-2}.

We first show that there exists a global solution $u_n$
to \eqref{equa-nse-g} of the form
\begin{align}\label{def-un}
u_n(t,x)=\sum_{j=1}^{n} \theta_j^n(t) e_j(x),
\end{align}
where $\{\theta_j^n \}$ are scalar functions.

To this end, taking $L^2_x$-inner product of \eqref{equa-nse-g} with $e_k$
we get
\begin{align*}
\sum_{j=1}^n( \frac{\d\theta_j^n }{\d t}e_j, e_k)-\sum_{j=1}^n ( \nu \Delta e_j, e_k)\theta_j^n +\sum_{i, j=1}^n ( \P_{<\Lambda(t)} (\P_{<\Lambda(t)}e_i \cdot \nabla  \P_{<\Lambda(t)}e_j) , e_k) \theta_i^n\theta_j^n=0,
\end{align*}
which leads to a system of ODEs,
\begin{align}\label{ode-1}
\frac{\d\theta_k^n }{\d t}=-\lambda_k \theta_k^n - \sum_{i, j=1}^n ( \P_{<\Lambda(t)} (\P_{<\Lambda(t)}e_i \cdot \nabla  \P_{<\Lambda(t)}e_j) , e_k) \theta_i^n\theta_j^n.
\end{align}
Note that the right-hand side of \eqref{ode-1} is local Lipschitz with respect to $\{\theta_j^n\}$. Hence, from the classical ODE theory
we know that there exists a unique local solution $\{\theta_j^n\}$.

To show that $\{\theta_j^n\}_{j=1}^n$ do not blow up at time $T$, we take the $L^2_x$ inner product of equation \eqref{equa-nse-g} with $u_n$ to get
\begin{align}  \label{dun-un}
	(\frac{\d}{\d t} u_n, u_n) + (-\nu \Delta u_n, u_n) +(\P_n\P_{<\Lambda(t)}  ( \P_{<\Lambda(t)} u_n\cdot\nabla \P_{<\Lambda(t)}u_n),u_n) =0.
\end{align}
Note that by \eqref{def-un},
$\P_n u_n = u_n$,
we have
\begin{align*}
	(\P_n\P_{<\Lambda(t)}  (\P_{<\Lambda(t)}  u_n\cdot\nabla \P_{<\Lambda(t)}u_n),u_n)= (  (\P_{<\Lambda(t)}  u_n\cdot\nabla \P_{<\Lambda(t)}u_n),\P_{<\Lambda(t)} u_n)=0.
\end{align*}
Hence, integrating \eqref{dun-un} in time we obtain the energy balance
\begin{align}\label{eq-e}
	\frac12 \|u_n(t)\|_{L^2_x}^2 + \nu \int_{0}^{t}\|\nabla u_n(s)\|_{L^2_x}^2\d s = \frac12 \|v_0\|_{L^2_x}^2.
\end{align}
In particular,
$u_n$ is uniformly bounded in $L^2_x$.
Thus, taking into account
\begin{align*}
\sum_{j=1}^n\left|\theta_j^n(t)\right|^2=\left\|u_n(t)\right\|_{L^2_x}^2,
\end{align*}
we infer that $\{\theta_j^n\}$ are bounded on the existing time,
which yields the global existence of smooth solution $u_n$ to \eqref{equa-nse-g} on $[0,T]$. 
Moreover, 
$u_n$ is $\mathcal{F}_0$ measurable 
and thus is $\{\mathcal{F}_t\}$-adapted. 

\medskip
\noindent{\bf $\bullet$ Passing to the limit.}
Now, in order to pass to the limit as $n \to \infty$, we let
$V:= L^2_\sigma \cap H^1_x$
and take the $L^2_x$ inner product of \eqref{equa-nse-g} with $\varphi\in V$ to get
\begin{align}
(\partial_t u_n, \varphi)= (\nu  \Delta u_n, \varphi)-(\P_n\P_{<\Lambda(t)}  ( \P_{<\Lambda(t)} u_n\cdot\nabla \P_{<\Lambda(t)}u_n),\varphi).
\end{align}
By the H\"{o}lder inequality, interpolation and the Sobolev embedding $H^1_x\hookrightarrow L^6_x$, we derive
\begin{align*}
 |(\nu  \Delta u_n, \varphi)|
 \leq \nu|(\nabla u_n, \nabla\varphi)|
 \leq \nu \| \nabla u_n\|_{L^2_x} \|\varphi\|_{V},
\end{align*}
and
\begin{align*}
 |(\P_n\P_{<\Lambda}  (\P_{<\Lambda}  u_n\cdot\nabla \P_{<\Lambda}u_n),\varphi)|
&\leq \|\P_{<\Lambda} u_n\|_{L^3_x}\|\nabla \P_{<\Lambda} u_n\|_{L^2_x} \|\P_{<\Lambda}\P_n\varphi\|_{L^6_x}\notag\\
&\lesssim  \left\|\P_{<\Lambda} u_n\right\|_{L^2_x}^{\frac12}\left\|u_n\right\|_{L^6_x}^{\frac12}\left\|\nabla u_n\right\|_{L^2_x} \left\| \varphi\right\|_{V} \\
& \lesssim \left\|u_n\right\|_{L^2_x}^{\frac12}\left\|\nabla u_n\right\|^{\frac32}_{L^2_x }\|\varphi\|_{V} .
\end{align*}
It follows that
\begin{align}
\int_{0}^{T}	\|\p_t u_n\|_{V^*}^{\frac43}\d t &\lesssim \int_{0}^{T}\| \nabla u_n\|_{L^2_x}^{\frac43}\d t+  \int_{0}^{T}\|u_n \|_{L^2_x}^{\frac23} \|\nabla u_n\|^{2}_{L^2_x }\d t\notag\\
& \lesssim T^{\frac13}\left\|u_n\right\|_{L^2(0, T; V)}^{\frac43}+ \|u_n \|_{L^{\infty}(0, T ; L^2_x)}^{\frac23} \|u_n \|_{L^2(0, T ; V)}^2\notag \\
& \lesssim T^{\frac13} \|v_0 \|_{L^2_x}^{\frac43}+  \|v_0 \|_{L^2_x}^{\frac83},
\end{align}
where the last step was due to the energy balance \eqref{eq-e},
and the implicit constants are independent of $n$ and $T$.

Thus, by the Aubin-Lions Lemma~\ref{ab-lem} in the Appendix and a standard diagonal argument, 
we can choose a subsequence (still denoted by $\{n\}$), 
which may depend on $\omega$, such that
\begin{align}
&u_n \rightarrow u \quad \text { strongly in}\quad L^2(0, T ; L^2_\sigma), \notag\\
&u_n \stackrel{w^*}{\rightharpoonup} u \quad\text { weakly-}* \text {in}\quad L^{\infty}(0, T ; L^2_\sigma), \label{conv} \\
&\nabla u_n \rightharpoonup \nabla u \quad\text { weakly in }\quad L^2\left(0, T ; L^2_x \right),\notag
\end{align}
for some weakly continuous $L^2_\sigma$-valued function $u$. 
We note that the weak continuity of $u$ can be obtained via a standard embedding argument as in \cite{bv2022}.

We claim the following limit holds
\begin{align}\label{con-non-un}
    \int_0^T (\P_{<\Lambda(t)}  (\P_{<\Lambda(t)}  u_n\cdot\nabla \P_{<\Lambda(t)}u_n),\varphi) \d t \rightarrow  \int_0^T (\P_{<\Lambda(t)}  (\P_{<\Lambda(t)}  u \cdot\nabla \P_{<\Lambda(t)}u ),\varphi) \d t
\end{align}
for any $\varphi\in C_0^\infty([0,T]\times \T^3 )$.

To this end, we note that
\begin{align*}
    &\P_{<\Lambda}  (\P_{<\Lambda} u_n\cdot\nabla \P_{<\Lambda}u_n)-\P_{<\Lambda}  ( \P_{<\Lambda} u \cdot\nabla \P_{<\Lambda}u )\\
    = &\, \P_{<\Lambda}(\P_{<\Lambda}  (u_n-u)\cdot\nabla \P_{<\Lambda}u_n)+ \P_{<\Lambda}  (\P_{<\Lambda}  u \cdot\nabla \P_{<\Lambda}(u_n-u )).
\end{align*}
From \eqref{eq-e} and \eqref{conv} it follows that
\begin{align*}
\bigg|\int_0^T (\P_{<\Lambda(t)}(\P_{<\Lambda(t)}  (u_n-u)\cdot\nabla \P_{<\Lambda(t)}u_n),\varphi )\d t \bigg|
& \lesssim  \int_0^T \|\P_{<\Lambda(t)} ( u_n-u)\|_{L^2_x} \| \nabla \P_{<\Lambda(t)}u_n \|_{L^2_x} \d t\notag\\
& \lesssim (\int_0^T \| u_n-u\|_{L^2_x}^2 \d t)^{\frac12}  (\int_0^T \| \nabla u_n \|_{L^2_x}^2 \d t)^{\frac12}  \rightarrow 0.
\end{align*}
Moreover, due to the fact that $\P_{<\Lambda(t)} (\P_{<\Lambda(t)} u_i\P_{<\Lambda(t)}\varphi_j) \in L^2(0,T;L^2_x)$, $1\leq i,j\leq 3$,
using \eqref{conv} again we have
\begin{align*}
\int_0^T (\P_{<\Lambda(t)}  ( \P_{<\Lambda(t)} u_i \p_i \P_{<\Lambda(t)}(u_n-u )_j,\varphi_j )\d t
& =  \int_0^T ( \p_i (u_n-u )_j,  \P_{<\Lambda(t)} (\P_{<\Lambda(t)} u_i\P_{<\Lambda(t)}\varphi_j) )\d t  \rightarrow 0
\end{align*}
as $n\to \infty$,
which yields \eqref{con-non-un}, as claimed.

Thus, we can use \eqref{conv} and \eqref{con-non-un}
to pass to the limit $n\to \infty$ in \eqref{equa-nse-g}
and obtain that $u$ satisfies $\Lambda$-NSE \eqref{equa-nse-2} in the sense of distributions.

\medskip
\noindent{\bf $\bullet$ Spatial regularity.}
Concerning the spatial regularity of $u$, for every $\xi\in \mathbb{Z}^3$, we denote by $\widehat{u}(\cdot,\xi)$ the Fourier coefficient of $u$, i.e.
\begin{align*}
  \widehat{u}(\cdot,\xi):=\int_{\T^3} u(\cdot,x) e^{ix\cdot \xi} \d x.
\end{align*}
We note that the Fourier coefficient of $u$ satisfies
for every $\xi\in \mathbb{Z}^3$,
\begin{align}\label{u-xi}
\widehat{u}(t,\xi)- \widehat{u}(0,\xi)=\int_{0}^t \sum_{ l+h=\xi}
 i \( \frac{(\xi\cdot\widehat{u}(l))\widehat{u}(h)\cdot \xi}{|\xi|^2} \xi -(\xi\cdot\widehat{u}(l))\widehat{u}(h) \)\varphi_{\Lambda}(\xi)\varphi_{\Lambda}(l)\varphi_{\Lambda}(h)-\nu |\xi|^2 \widehat{u}(\xi) \d r.
\end{align}
For any given $t\in (0,T]$, let $k_0= \Lambda(t)$. Then, for any $|\xi|>2k_0$ and $\tau\geq t$ we have
\begin{align*}
\widehat{u}(\tau,\xi)- \widehat{u}(t,\xi)=-\int_{t}^{\tau} \nu |\xi|^2 \widehat{u}(r,\xi) \d r.
\end{align*}
Hence,
\begin{align}\label{u-xi-k0}
\widehat{ u}(\tau,\xi) =\widehat{ u}(t,\xi)e^{-\nu |\xi|^2 (\tau-t)}, \ \ \forall\ |\xi|>2k_0,\ \tau\geq t.
\end{align}
For any $s\geq 0$ and $\tau\geq \frac{3}{2}t$, by \eqref{decay-lam} and \eqref{u-xi-k0} we derive
\begin{align}\label{est-u-hs}
\|u(\tau)\|_{\dot{H}^s_x}^2
&\leq  \sum_{|\xi|\leq  2k_0} \left||\xi|^{s} \widehat{u} (\tau,\xi)\right|^2+ \sum_{|\xi|>2k_0 } \left| |\xi|^{s} \widehat{u} (\tau,\xi)\right|^2\notag \\
&\leq \sum_{|\xi|\leq  2k_0} \left||\xi|^{s} \widehat{u} (\tau,\xi)\right|^2+ \sum_{|\xi|>2k_0 } | |\xi|^{s} e^{-\nu |\xi|^2 (\tau-t)} \widehat{ u} (t,\xi)|^2\notag \\
&\leq \sum_{|\xi|\leq  2k_0} | |\xi|^{s} \widehat{u} (\tau,\xi)|^2 + \|   e^{- \nu \Delta (\tau-t)} u(t) \|_{\dot{H}^s_x}^2\notag\\
&\lesssim ( \Lambda^{2s}(t)+ t^{-s})\|v_0\|_{L^2_x}^2,
\end{align}
where we also use the heat semigroup estimate
\begin{align*}
    \| e^{-\nu \Delta (\tau-t)} u(t) \|_{\dot{H}^s_x}^2\lesssim (\tau-t)^{-s} \|u\|_{C_TL^2_x}^2\lesssim t^{-s} \|u\|_{C_TL^2_x}^2\ \text{for any}\ \tau\geq \frac32 t.
\end{align*}
In particular, by \eqref{decay-lam}
\begin{align}\label{spa-regu-pointwise}
    \|u(\tau)\|_{\dot{H}^s_x}^2 \lesssim (\Lambda^{2s}(\frac23 \tau)+(\frac23 \tau)^{-s})\|v_0\|_{L^2}^2\lesssim (1+(\frac23 \tau)^{-\frac{s}{4}}+(\frac23\tau)^{-s}) \|v_0\|_{L^2}^2\lesssim ( 1+ \tau^{-s}) \|v_0\|_{L^2}^2\ \ \forall \tau> 0.
\end{align}
This yields that
\begin{equation} \label{spa-regu}
\|u\|_{C([t,T];\dot{H}^s_x)) } \lesssim ( 1+ t^{-\frac{s}{2}})\|v_0\|_{L^2_x}\lesssim ( 1+ t^{-\frac{s}{2}})M  < \infty,
\end{equation}
where the implicit constants are deterministic and universal.

\medskip
\noindent{\bf $\bullet$ Temporal regularity.}
In order to obtain the temporal regularity of $u$, we use the inductive
arguments. By equation \eqref{equa-nse-2}, \eqref{decay-lam},
\eqref{spa-regu-pointwise} and the embedding $W^{s+3,1}_x\hookrightarrow{H}^{s+1}_x$, we get
\begin{align}\label{tem-regu-1}
\|\p_t u\|_{C([t,T];\dot{H}^s_x)}
& \leq \nu\|\Delta u \|_{C([t,T];\dot{H}^s_x)}
+ \|\P_H \P_{<\Lambda}\div( (\P_{<\Lambda}u )\otimes (\P_{<\Lambda}u ))\|_{C([t,T];\dot{H}^s_x)}\notag\\
&\lesssim \| u \|_{C([t,T];\dot{H}^{s+2}_x)}+ \| (\P_{<\Lambda}u )\otimes (\P_{<\Lambda}u )\|_{C([t,T];{H}^{s+1}_x)}\notag\\
&\lesssim \| u \|_{C([t,T];\dot{H}^{s+2}_x)}+ \| \P_{<\Lambda} u \|_{C([t,T];{H}^{s+3}_x)}^2\notag\\
&\lesssim (1+t^{-\frac{s+2}{2}})\|v_0\|_{L^2_x}+(1+\Lambda^{2s+6}(t))\|v_0\|_{L^2_x}^2  \notag \\
&\lesssim (1+ t^{-\frac{s+2}{2}}) (\|v_0\|_{L^2_x}+\|v_0\|_{L^2_x}^2),
\end{align}
where the implicit constants are deterministic and universal.

Assume that $\|\p_t^{m} u\|_{C([t,T];\dot{H}^s)}$ is finite for all $s>0$ and $1\leq m\leq N-1$. By \eqref{u-xi}, we derive
\begin{align}\label{tem-regu-2}
\|\p_t^N u\|_{C([t,T];\dot{H}^s_x)}&=\|\sum_{\xi \in \mathbb{Z}^3}|\xi|^s \widehat{\p_t^N u}(\xi)\|_{C([t,T];L^2_\xi)}\notag\\
& \leq  \| \sum_{\xi \in \mathbb{Z}^3}\sum_{ l+h=\xi} \sum_{N_1+N_2+N_3=N-1} |\xi|^s (\xi\cdot\widehat{\p_t^{N_1}u}(l))\widehat{\p_t^{N_2}u}(h) \p_t^{N_3}(\varphi_{\Lambda}(\xi)\varphi_{\Lambda}(l)\varphi_{\Lambda}(h))\|_{C([t,T];L^2_\xi)}\notag\\
&\quad + \|\sum_{\xi \in \mathbb{Z}^3}|\xi|^{s+2 }\widehat{\p_t^{N-1} u}(\xi) \|_{C([t,T];L^2_\xi)}\notag\\
&:= J_1+J_2.
\end{align}
By the inductive assumption, we have $J_2<+ \infty$.
Moreover, the first term on the right-hand side above can be estimated by
\begin{align}
J_1& \leq \sum_{N_1+N_2+N_3=N-1} \left\| \sum_{\xi \in \mathbb{Z}^3}\sum_{ l+h=\xi}  |\xi|^s (\xi\cdot\widehat{\p_t^{N_1}u}(l))\widehat{\p_t^{N_2}u}(h) \right\|_{C([t,T];L^2_\xi)}\sup_{l+h=\xi,\,\xi \in \mathbb{Z}^3}\|\p_t^{N_3}(\varphi_{\Lambda}(\xi)\varphi_{\Lambda}(l)\varphi_
{\Lambda}(h))\|_{C([t,T])}\notag\\
&\leq  \sum_{N_1+N_2+N_3=N-1} \left\| (\p_t^{N_1}u\cdot \nabla )\p_t^{N_2}u  \right\|_{C([t,T];\dot{H}^s_x)}\sup_{l+h=\xi,\, \xi \in \mathbb{Z}^3}\|\p_t^{N_3}(\varphi_{\Lambda}(\xi)\varphi_{\Lambda}(l)\varphi_{\Lambda}(h))\|_{C([t,T])}\notag\\
&\leq  \sum_{N_1+N_2+N_3=N-1}  \|  \p_t^{N_1}u \|_{C([t,T];H^{s+2}_x)} \|  \p_t^{N_2}u \|_{C([t,T];H ^{s+3}_x)}\sup_{l+h=\xi,\,\xi \in \mathbb{Z}^3}\|\p_t^{N_3}(\varphi_{\Lambda}(\xi)\varphi_{\Lambda}(l)\varphi_{\Lambda}(h))\|_{C([t,T])}\notag\\
&<+\infty.
\end{align}

Thus, using inductive arguments we obtain that
$\|\p_t^{m} u\|_{C([t,T];\dot{H}^s)}<\infty $ for all $s>0$ and $m\in \mathbb{N}$.
In particular,
$u$ is smooth for positive times.
Moreover, due to the skew-symmetry of the nonlinear term, the energy balance \eqref{eq-e-2} holds.

\medskip
\paragraph{\bf Continuity at initial time} Fix $\omega\in \Omega$, from the energy balance \eqref{eq-e-2}, we get
\begin{align*}
\limsup_{t\to 0^+} \|u(t)\|_{L^2_x}\leq \|v_0\|_{L^2_x}.
\end{align*}
Moreover, since $u(t)\rightharpoonup v_0$ in $L^2_x$ as $t\to 0^+$, 
\begin{align*}
 \|v_0\|_{L^2_x}\leq \liminf_{t\to 0^+}\|u(t)\|_{L^2_x}.
\end{align*}
Hence,
\begin{align*}
\|u(t)\|_{L^2_x}\to \|v_0\|_{L^2_x}\quad \text{as}\quad t\to 0^+,
\end{align*}
which along with the fact that $v_0$ is the weak limit of $u(t)$ in $L^2_x$ shows that
\begin{align*}
u(t)\to v_0  \quad \text { strongly in}\quad  L^2_x\quad \text{as}\quad t\to 0^+.
\end{align*}
Hence, we derive
\begin{align}
    \|u(\cdot)-v_0 \|_{ C([0,T_*];L^2_x)}\rightarrow 0\quad \text{as}\quad T_*\rightarrow 0^+.
\end{align}
Moreover, by \eqref{spa-regu}, we get
\begin{align}
    \|u -v_0 \|_{ L^\infty_{\Omega}C([0,T_*];L^2_x)} \leq  \|u \|_{ L^\infty_{\Omega}C([0,T_*];L^2_x)}+ \|v_0 \|_{ L^\infty_{\Omega}C([0,T_*];L^2_x)}\lesssim M.
\end{align}
Therefore, by the Lebesgue dominated convergence theorem, fix any $\rho\in (1,\infty)$, we have
\begin{align}
  \|u -v_0 \|_{ L^\rho_{\Omega}C([0,T_*];L^2_x)} \rightarrow 0\quad \text{as}\quad T_*\rightarrow 0^+.
\end{align}

\medskip
\paragraph{\bf $(ii)$ Uniqueness.}
We shall prove that any 
probabilisticall strong and analytically weak solution $v$ to \eqref{equa-nse-2} satisfying the energy inequality \eqref{eq-ine-2} coincides with the above constructed solution $u$.

To this end, we first note that by \eqref{conv}, $u\in L^2(0,T;H^1_x)$.
We claim that $\P_{<\Lambda}  (\P_{<\Lambda}u\cdot \nabla \P_{<\Lambda}u)\in L^2(0,T;H^{-1}_x)\cap L^1(0,T;L^2_x)$. To this end, for $\phi\in L^2(0,T;H^1_x)$, by \eqref{decay-lam} and Sobolev's embedding $H^1_x\hookrightarrow L^4_x$, we have
\begin{align}
\int_0^T (\P_{<\Lambda}  (\P_{<\Lambda}u\cdot \nabla \P_{<\Lambda}u),\phi)\d t
&\lesssim (\int_0^T \|\P_{<\Lambda}u\|^4_{H^1_x}\d t)^\frac12 \|\nabla \phi\|_{L^2_TL^2_x} \notag\\
&\lesssim (\int_0^T (1+ \Lambda(s))^4  \d s)^\frac12\|u\|_{C_T L^2_x}^2\|\nabla \phi\|_{L^2_TL^2_x}
<+\infty.
\end{align}
Next, we show that $\P_{<\Lambda}  (\P_{<\Lambda}u\cdot \nabla \P_{<\Lambda}u) \in L^1(0,T;L^2_x)$. Indeed, using \eqref{decay-lam} and Sobolev's embedding $H^2_x\hookrightarrow L^\infty_x$ we get
\begin{align*}
\int_0^T \| \P_{<\Lambda}  (\P_{<\Lambda}u\cdot \nabla \P_{<\Lambda}u)\|_{L^2_x} \d t
&\leq \int_0^T \|  \nabla \P_{<\Lambda}u \|_{L^2_x}\| \P_{<\Lambda}u\|_{L^\infty_x} \d t\notag\\
&\lesssim \int_0^T  \Lambda(t) \| u \|_{L^2_x} \| \P_{<\Lambda}u\|_{H^2_x} \d t\notag\\
&\lesssim \int_0^T  \Lambda(t)(1+ \Lambda^2(t))  \d t\| u \|_{C_TL^2_x}^2 <+\infty.
\end{align*}
Thus, we get $\P_{<\Lambda}  (\P_{<\Lambda}u\cdot \nabla \P_{<\Lambda}u)\in L^2(0,T;H^{-1}_x)\cap L^1(0,T;L^2_x)$, as claimed,
which along with the fact that $u\in L^2(0,T;H^1_x)$ give
\begin{align}
\p_t u \in L^2(0,T;H^{-1}_x).
\end{align}
Since $v\in C_TL^2_x\cap L^2_TH^1_x$, we have
\begin{align}\label{e-u-v}
\int_0^t (\partial_t u, v)\d s
 +\nu\int_0^t (\nabla u, \nabla v)\d s+\int_0^t
 (\P_{<\Lambda(s)}  (\P_{<\Lambda(s)}u\cdot \nabla \P_{<\Lambda(s)}u), v)\d s=0.
\end{align}
Similarly, $\P_{<\Lambda}  (\P_{<\Lambda}v\cdot \nabla \P_{<\Lambda}v) \in L^2(0,T;H^{-1}_x)\cap L^1(0,T;L^2_x)$, 
and we have
\begin{align}\label{e-v-u}
-\int_0^t ( v, \partial_t u )\d s
+\nu\int_0^t  (\nabla v, \nabla u)\d s+\int_0^t  (\P_{<\Lambda(s)}  ( \P_{<\Lambda(s)}v\cdot \nabla \P_{<\Lambda(s)}v), u ) \d s
= ( v(0), u(0) )-  ( v(t), u(t)) .
\end{align}
Combining \eqref{e-u-v} and \eqref{e-v-u} together and using the fact that $u(0)=v(0)=v_0$ we obtain
\begin{align}\label{ineq-u+v}
&\quad 2 \nu\int_0^t  (\nabla v, \nabla u)\d s+\int_0^t ( \P_{<\Lambda(s)}  (\P_{<\Lambda(s)} u\cdot \nabla \P_{<\Lambda(s)}u), v)\d s+\int_0^t (\P_{<\Lambda(s)}  ( \P_{<\Lambda(s)}v\cdot \nabla \P_{<\Lambda(s)}v), u)\d s\notag\\
&=\left\|v_0\right\|^2_{L^2_x}- ( v(t), u(t)).
\end{align}
Let us denote $w:=u-v$ the difference between $u$ and $v$. Then, by straightforward computations,
\begin{align}
 & 2 (\nabla u, \nabla v)=\|\nabla u\|_{L^2_x}^2+\|\nabla v\|_{L^2_x}^2-\|\nabla w\|_{L^2_x}^2, \label{left-1} \\
 &  (v(t), u(t))=\frac{1}{2}\|u(t)\|_{L^2_x}^2+\frac{1}{2}\|v(t)\|_{L^2_x}^2-\frac{1}{2}\|w(t)\|_{L^2_x}^2\label{right-2}
\end{align}
and
\begin{align}\label{left-23}
	&\int_0^t  (\P_{<\Lambda(s)}  (\P_{<\Lambda(s)} u\cdot \nabla \P_{<\Lambda(s)}u), v)\d s
     + \int_0^t (\P_{<\Lambda(s)}  (\P_{<\Lambda(s)} v\cdot \nabla \P_{<\Lambda(s)}v), u)\d s\notag \\
 = &\,\int_0^t  (\P_{<\Lambda(s)}  ( \P_{<\Lambda(s)}w\cdot \nabla \P_{<\Lambda(s)}w), u) \d s.
\end{align}
Plugging \eqref{left-1}-\eqref{left-23} into \eqref{ineq-u+v} we obtain
\begin{align}
\frac{1}{2}\|w(t)\|_{L^2_x}^2+\nu\int_0^t\|\nabla w\|_{L^2_x}^2\d s-\int_0^t(\P_{<\Lambda(s)}  (\P_{<\Lambda(s)} w\cdot \nabla \P_{<\Lambda(s)}w), u)\d s=E_1+E_2,
\end{align}
where, by the energy balance \eqref{eq-e-2},
\begin{align*}
E_1 :=\frac{1}{2}\|u(t)\|_{L^2_x}^2+\nu\int_0^t\|\nabla u\|_{L^2_x}^2 \d s-\frac{1}{2}\|v_0\|_{L^2_x}^2=0,
\end{align*}
and by the energy inequalitiy \eqref{eq-ine-2},
\begin{align*}
E_2 :=\frac{1}{2}\|v(t)\|_{L^2_x}^2+\nu\int_0^t\|\nabla v(s)\|_{L^2_x}^2 \d s-\frac{1}{2}\|v_0\|_{L^2_x}^2 \leq 0.
\end{align*}
Thus, it follows that
\begin{align}
\frac{1}{2}\|w(t)\|_{L^2_x}^2+\nu\int_0^t\|\nabla w\|_{L^2_x}^2\d s &\leq \int_0^t ( \P_{<\Lambda(s)}  ( \P_{<\Lambda(s)}w\cdot \nabla \P_{<\Lambda(s)}w), u)\d s\notag\\
&\leq \int_0^t \|\P_{<\Lambda(s)}u\|_{L^\infty_x} \|\P_{<\Lambda(s)} w\|_{L^2_x } \|\nabla \P_{<\Lambda(s)}w\|_{L^2_x } \d s\notag\\
&\leq \frac{2}{\nu} \int_0^t \|\P_{<\Lambda(s)}u\|_{H^2_x}^2 \| w\|_{L^2_x }^2\d s +\frac\nu2 \int_0^t  \|\nabla w\|_{L^2_x}^2 \d s,
\end{align}
which yields that
\begin{align}\label{est-w}
\|w(t)\|_{L^2_x}^2+\nu\int_0^t\|\nabla w\|_{L^2_x}^2\d s&\leq \frac{4}{\nu}\int_0^t \|\P_{<\Lambda(s)}u\|_{H^2_x}^2 \| w\|_{L^2_x }^2\d s.
\end{align}
Since $\Lambda(t)\leq t^{-\frac18}$ as $t\in (0,T]$, we have 
\begin{align}   \label{Lambda-uH2-esti}
   \int_0^T \|\P_{<\Lambda(s)}u\|_{H^2_x}^2 \d s= \int_0^T \|(1+|\xi|^2)\widehat  \varphi_{\Lambda(s)}(\xi) \widehat u  \|_{L^2_\xi}^2 \d s \leq \int_0^T (1+4\Lambda(s)^2)^2  \d s\|u\|_{C_T L^2_x}^2<\infty.
\end{align}
Consequently,
applying Gronwall's lemma to \eqref{est-w}
we obtain that $w(t)=0$ for all $t\in [0,T]$. 

We note that by the 
above uniqueness result, 
the convergences in  
\eqref{conv} hold for the whole sequence $\{u_n\}$ 
which is independent of $\omega$. Since $\{u_n\}$ are $\{\mathcal{F}_t\}$-adapted, 
so is the limit $u$.

\medskip
\paragraph{\bf $(iii)$ Vanishing of nonlinearity on high modes at the initial time.}
First, we note that for the Fourier support
\begin{align}
\supp_{\xi}  \(\P_{< \frac{\Lambda}{6}}\wt u \mathring \otimes \P_{< \frac{\Lambda}{6}}\wt u\) \subseteq [0, \frac{2\Lambda}{3}],
\end{align}
and so
\begin{align}
    \P_{\geq \Lambda}\(\P_{< \frac{\Lambda}{6}}\wt u \mathring \otimes \P_{< \frac{\Lambda}{6}}\wt u\)=0.
\end{align}
It follows that
\begin{align}\label{decom-I}
\P_{\geq \Lambda}(\P_{< \Lambda(t)}\wt u\mathring \otimes \P_{< \Lambda}\wt u)
&=  \P_{\geq \Lambda}\((\P_{< \frac{\Lambda}{6}}\wt u + \P_{[\frac{\Lambda}{6},2\Lambda]}\wt u)\mathring \otimes (\P_{< \frac{\Lambda}{6}}\wt u + \P_{[\frac{\Lambda}{6},2\Lambda ]}\wt u)\)
\notag \\
&=   \P_{\geq \Lambda}\(\P_{< \frac{\Lambda}{6}}\wt u \mathring \otimes \P_{[\frac{\Lambda}{6},2\Lambda]}\wt u\)+   \P_{\geq \Lambda}\(\P_{[\frac{\Lambda}{6},2\Lambda]}\wt u \mathring \otimes\P_{< \frac{\Lambda}{6}} \wt u\) \notag\\
&\quad + \P_{\geq \Lambda}\(\P_{[\frac{\Lambda}{6},2\Lambda]}\wt u \mathring \otimes \P_{[\frac{\Lambda}{6},2\Lambda]}\wt u\),
\end{align}
where $\P_{[\frac{\Lambda}{6},2\Lambda]}\wt u :=
\P_{<\Lambda}\wt u - \P_{ <\frac{\Lambda}{6}}\wt u
(=\P_{\geq \frac{\Lambda}{6}}\wt u - \P_{\geq \Lambda}\wt u)$,
which has Fourier support on $[\Lambda/6, 2\Lambda]$.
Then, fix any $\rho\in [1,\infty)$, we derive from \eqref{decom-I},
the uniform boundedness of $\P_{\geq \Lambda}$ in $L^1_x$
and the inequality $\|\P_{< \frac{\Lambda}{6}}\wt u\|_{L^2_x} \leq \|\wt u\|_{L^2_x}$
that
\begin{align}\label{est-non-decay-2}
&\quad \|\P_{\geq \Lambda}(\P_{< \Lambda}\wt u\mathring \otimes \P_{< \Lambda}\wt u)\|_{L^{\rho}_{\Omega}C([0,T_*]; L^1_x)} \notag\\
&\leq   \|\P_{\geq \Lambda}
\(\P_{< \frac{\Lambda}{6}}\wt u \mathring \otimes \P_{[\frac{\Lambda}{6},2\Lambda ]}\wt u\)\|_{L^{\rho}_{\Omega}C([0,T_*]; L^1_x)}
+ \| \P_{\geq \Lambda}\(\P_{[\frac{\Lambda}{6},2\Lambda]}\wt u \mathring \otimes\P_{< \frac{\Lambda}{6}} \wt u\) \|_{L^{\rho}_{\Omega}C([0,T_*]; L^1_x)}  \notag\\
&\quad + \|\P_{\geq \Lambda}\(\P_{[\frac{\Lambda}{6},2\Lambda]}\wt u \mathring \otimes \P_{[\frac{\Lambda}{6},2\Lambda]}\wt u\)\|_{L^{\rho}_{\Omega}C([0,T_*]; L^1_x)}  \notag\\
&\leq C ( \|\wt u \|_{L^{2\rho}_{\Omega}C([0,T_*]; L^2_x)}  \| \P_{[\frac{\Lambda}{6},2\Lambda ]}\wt u\|_{L^{2\rho}_{\Omega}C([0,T_*]; L^2_x)}  + \| \P_{[\frac{\Lambda}{6},2\Lambda]}\wt u\|_{L^{2\rho}_{\Omega}C([0,T_*]; L^2_x)} ^2 ).
\end{align}
Taking into account that, by \eqref{u-l2-decay},
\begin{align*}
   \|\P_{[\frac{\Lambda}{6},2\Lambda]}\wt u \|_{L^{2\rho}_{\Omega}C([0,T_*]; L^2_x)}
  \leq  \| \P_{\geq \frac{\Lambda}{6}}\wt u\|_{L^{2\rho}_{\Omega}C([0,T_*]; L^2_x)}
        + \|\P_{\geq \Lambda}\wt u\|_{L^{2\rho}_{\Omega}C([0,T_*]; L^2_x)}
  \to 0, \ \ \text{as}\ T_*\to 0^+,
\end{align*}
we thus obtain \eqref{deacy-non} and finish the proof of Theorem~\ref{thm-fns}.
\hfill $\square$

\section{Intermittent velocity flow} \label{Sec-Interm-Flow}

From this section to Section \ref{Sec-Rey-stress}
we aim to prove the main iterative estimates in  Proposition~\ref{Prop-Iterat} at level $q+1$.

\subsection{Intermittent jets}

Let us begin with the construction of velocity perturbations at level $q+1$, which will be indexed by the following five parameters
\begin{equation}\label{larsrp}
\rp:= \lambda_{q+1}^{-\frac47},\ \rs := \lambda_{q+1}^{-\frac67},\
	 \lambda := \lambda_{q+1},\ \mu:=\lambda_{q+1}^{\frac97},\ \sigma=\lambda_{q+1}^{\frac17}.
\end{equation}

We use the intermittent jets introduced in \cite{bcv21}
as the basic building blocks. More precisely,
let $\Phi : \mathbb{R}^2 \to \mathbb{R}$ be a smooth function with support in a ball of radius $\frac18$, centered at $(\frac12,\frac12)$. Suppose that $\Phi$ is normalized such that $\phi := - \Delta\Phi$ satisfies
\begin{equation}\label{e4.91}
\int_{\mathbb{R}^2} \phi^2(x)\d x = 1.
\end{equation}
 Define $\psi: \mathbb{R} \rightarrow \mathbb{R}$ to be a smooth and mean zero function, satisfying
\begin{equation}\label{e4.92}
\int_{\mathbb{R}} \psi^{2}\left(x\right) \d x=1, \quad \supp\psi\subseteq [\frac38,\frac58].
\end{equation}
We define the rescaled cut-off functions by
\begin{equation*}
	\phi_{\rs}(x) := {\rs^{-1}}\phi\left(\frac{x}{\rs}\right), \quad
	\Phi_{\rs}(x):=   {\rs^{-1}} \Phi\left(\frac{x}{\rs}\right),\quad
	\psi_{\rp}\left(x\right) := {r_{\|}^{- \frac 12}} \psi\left(\frac{x}{r_{\|}}\right).
\end{equation*}
By an abuse of notation, let us periodize $\phi_{\rs}$, $\Phi_{\rs}$ and $\psi_{\rp}$,
so that
$\phi_{\rs}$, $\Phi_{\rs}$ are treated as periodic functions on $\mathbb{T}^2$,
and $\psi_{\rp}$ as a periodic function on $\mathbb{T}$.

\vspace{1ex}

Let $(k,k_1,k_2)$ be the orthonormal bases for every $k\in \Lambda$, where $\Lambda \subset \mathbb{S}^2 \cap \mathbb{Q}^3$ is the wavevector set as in the Geometric Lemma ~\ref{geometric lem 2}.
The \textit{intermittent jets} are defined by
\begin{equation*}
	W_{(k)} :=  \psi_{\rp}(\sigma N_{\Lambda}(k_1\cdot x+\mu t))\phi_{\rs}( \sigma N_{\Lambda}k\cdot (x-\alpha_k), \sigma N_{\Lambda}k_2\cdot (x-\alpha_k))k_1,\ \  k \in \Lambda,
\end{equation*}
where $N_{\Lambda}$ is a constant given by \eqref{NLambda}, the shifts $\a_k\in \R^3$, $k\in \Lambda$, are chosen suitably  such that the functions $\{W_{(k)}\}$ have mutually
disjoint supports. The existence of such $\{\a_k\}$ can be guaranteed by taking $\rs$ sufficiently small.

For brevity of notations, we set
\begin{equation}\label{snp}
	\begin{array}{ll}
		&\psi_{(k_1)}(x) :=\psi_{\rp}(\sigma N_{\Lambda}(k_1\cdot x+\mu t)), \\
		&\phi_{(k)}(x) := \phi_{\rs}( \sigma N_{\Lambda}k\cdot (x-\a_k),\sigma N_{\Lambda}k_2\cdot (x-\a_k)), \\
		&\Phi_{(k)}(x) := \Phi_{\rs}(\sigma N_{\Lambda}k\cdot (x-\a_k),\sigma N_{\Lambda}k_2\cdot (x-\a_k)),
	\end{array}
\end{equation}
and thus
\begin{equation}\label{snwd}
	W_{(k)} = \psi_{(k_1)}\phi_{(k)} k_1,\quad  k\in \Lambda.
\end{equation}

Furthermore, in order to define the incompressible corrector of perturbations, we set
	\begin{align}\label{corrector vector}
		\wt W_{(k)}^c := \frac{1}{\lambda^2N_{ \Lambda }^2} \nabla\psi_{(k_1)}\times\curl(\Phi_{(k)} k_1),\quad 	W^c_{(k)} := \frac{1}{\lambda^2N_{\Lambda}^2 } \psi_{(k_1)}\curl\Phi_{(k)} k_1.
	\end{align}
It holds that
\begin{equation}\label{wcwc}
	W_{(k)} + \wt W_{(k)}^c
	=\curl \left(\frac{1}{\lambda^2N_{\Lambda}^2 } \psi_{(k_1)}\curl\Phi_{(k)} k_1\right)
	=\curl W^c_{(k)}.
\end{equation}
In particular,
it follows the incompressiblity
$\div (W_{(k)}+ \wt W^c_{(k)}) =0$.

The key estimates of the intermittent jets are summerized in Lemma \ref{buildingblockestlemma} below.

\begin{lemma} [Estimates of intermittent jets, \cite{bv19r}] \label{buildingblockestlemma}
	For $p \in [1,\infty]$, $k\in\Lambda$,
 and $N,\,M \in \mathbb{N}$, we have
	\begin{align}
		&\left\|\nabla^{N} \partial_{t}^{M} \psi_{(k_1)}\right\|_{C_T L^{p}_{x}}
		\lesssim r_{\|}^{\frac 1p- \frac 12} (\sigma r_{\|}^{-1} )^{N}
		 (\sigma \mu r_{\|}^{-1} )^{M}, \label{intermittent estimates} \\
		&\left\|\nabla^{N} \phi_{(k)}\right\|_{L^{p}_{x}}+\left\|\nabla^{N} \Phi_{(k)}\right\|_{L^{p}_{x}}
		\lesssim r_{\perp}^{\frac 2p- 1}  \lambda^{N}. \label{intermittent estimates2}
	\end{align}
 Moreover, it holds that
	\begin{align}
		&\displaystyle\left\|\nabla^{N} \partial_{t}^{M} W_{(k)}\right\|_{C_T  L^{p}_{x}}
		+\frac{r_{\|}}{r_{\perp}}\left\|\nabla^{N} \partial_{t}^{M} \wt W_{(k)}^{c}\right\|_{C_T L^{p}_{x}}
		+\lambda \left\|\nabla^{N} \partial_{t}^{M} W_{(k)}^c\right\|_{C_T L^{p}_{x}} \lesssim r_{\perp}^{\frac 2p- 1} r_{\|}^{\frac 1p- \frac 12} \lambda^{N}
		 (\sigma \mu r_{\|}^{-1})^{M}.   \label{ew}
	\end{align}
	The implicit constants above are independent of $\rs,\,\rp,\,\lambda$, $\sigma$ and $\mu$.
\end{lemma}

\subsection{Velocity flows}   \label{Sec-Pert}
Below we construct the key velocity  perturbations
and verify the inductive estimates for the corresponding velocity flows.

\subsubsection{\bf Velocity perturbations}
As in \cite{bcv21,bv19b},
the crucial velocity perturbations in the convex integration scheme
consist of three components: the principal part, the incompressibility corrector and the temporal corrector.

\medskip
\paragraph{\bf Principal part.}

Let $\phi_{\varepsilon}$ be  a family of standard  mollifiers on $\T^3$ supported on a ball of radius $\ve (>0)$ centered at $0$ and $\varphi_{\varepsilon}$ be a family of  standard mollifiers on $\T$ supported on $(0,\varepsilon)$. Note that, the one-sided temporal mollifier allows to preserve the adaptedness. The mollification of $\ru$ in space and time is defined by
\begin{equation}\label{mol}
	\begin{array}{ll}
		\mathring{R}_{\ell_q}:= (\ru *_{x} \phi_{\ell_q}) *_{t} \varphi_{\ell_q},
	\end{array}
\end{equation}
where $\ell_q$ is given by \eqref{def-ell}.

Note that, for any $N\in \mathbb{N}_+$,
\begin{align}
&  \| \mathring{R}_{\ell_q}(t)\|_{L^{1}_{ x}} \lesssim   \sup_{s\in [t-\ell_q,t]}\|\mathring R_{q}(s)\|_{L^{1}_{x}}, \label{rll1-s}\\
&  \| \mathring{R}_{\ell_q}\|_{C_{T,x}^N} \lesssim  \ell_q^{-4-N}\|\mathring{R}_q \|_{C_TL^1_x}.  \label{e3.92}
\end{align}

Set
\begin{align}
&	\rho (t,x) :=   2 \varepsilon_u^{-1} (\ell_q^2+|\mathring{R}_{\ell_q}(t,x) |^2)^{\frac12}
 ,\label{rhob}\\
 &	\gamma_{q}(t):= \frac13\Big (e(t) -f_q  -\E\|u_q+z_{ q}\|_{L^2_x}^2 \Big)*_{t}  \varphi_{\ell_q}(t), \label{def-gamma}
\end{align}
where $\varepsilon_{u}$ is the small radius in Geometric Lemma \ref{geometric lem 2} and
\begin{align}\label{def-fq}
	f_q(t) :=1_{[T_{q+1},T]}\frac{1}{2}\delta_{q+ 2}+ 1_{[T_{q+2}+\ell_{q},T_{q+1}]}\frac34\delta_{q+ 3}.
\end{align}
Note that
\begin{equation}\label{rhor}
	\left|  \frac{\mathring{R}_{{\ell_q}} }{\rho+\gamma_{q} } \right|\leq \varepsilon_u,
\end{equation}
and
\begin{align}
	\label{rhoblowbound}
	&\rho \gtrsim \ell_q,\quad 0\leq \gamma_{q}\leq 2\delta_{q+1},\\
	\label{rhoblp}
	&\norm{ \rho (t) }_{ L^p_{x}} \leq  2\ve_u^{-1}( \ell_q + \norm{\mathring{R}_{q} }_{C_{[t-\ell_q,t]}L^p_{x}}),\quad p\in [1, +\infty).
\end{align}
Moreover, using the iterative estimate \eqref{rhoblowbound}, we have for $N\geq 1$ and $M\geq 0$,
\begin{align}\label{rhoB-Ctx.1}
 \norm{ \rho  }_{C_{T,x}} \lesssim   \ell_q^{-4}(1+\|\rr_q \|_{C_{T}L^1_x}), \quad   \norm{ \rho  }_{C_{T,x}^N}  \lesssim   \ell_q^{-6N+1}(1+\|\rr_q \|_{C_{T}L^1_x})^{N}, \quad\norm{ \p_t^M \gamma_q}_{C_{T}} \lesssim \ell_q^{-M},
\end{align}
where we also choose $a$ sufficiently large such that the implicit constants are deterministic and universal.

The amplitudes of velocity perturbations are defined by
\begin{equation}\label{akb}
	a_{(k)}(t,x):= (\rho (t,x)+\gamma_{q} (t))^{\frac{1}{2} }\gamma_{(k)}
       \left(\Id-\frac{\mathring{R}_{\ell_q}(t,x)}{\rho(t,x)+\gamma_{q} (t)}\right), \quad k \in \Lambda,
\end{equation}
where $\gamma_{(k)}$ and $\Lambda$ are as in Geometric Lemma~\ref{geometric lem 2}.

In view of Geometric Lemma~\ref{geometric lem 2}
and the expression \eqref{akb},
the following algebraic identity holds
\begin{align}\label{velcancel}
	\sum\limits_{ k \in  \Lambda } a_{(k)}^2
	W_{(k)} \otimes W_{(k)}
	& = (\rho+\gamma_{q})\Id - \mathring{R}_{\ell_q}
	+  \sum\limits_{ k \in \Lambda  }  a_{(k)}^2 \P_{\neq 0}(  W_{(k)} \otimes  W_{(k)} ),
\end{align}
where $\P_{\neq 0}$ denotes the spatial projection onto nonzero Fourier modes.

By virtue of \eqref{rhoblowbound}-\eqref{rhoB-Ctx.1}, we have the analytic estimates of amplitudes in Lemma~\ref{mae} below.
The proof follows in an analogous way as in \cite{lzz21},
and thus it is omitted here.

\begin{lemma} [Estimates of amplitudes] \label{mae}
For every $N\geq 1$, $k\in \Lambda$ and $t\in [0,T]$, we have
	\begin{align}
		&\norm{a_{(k)}}_{C_TL^2_{x}} \lesssim (\ell_q + J_q)^{\frac12}+\delta_{q+1}^{\frac12},\label{est-akbl2}\\
			&\norm{ a_{(k)} }_{C_{T,x} } \lesssim  \ell_q^{-2} (1+J_q)^{\frac12}, \label{est-akbc}\\
		& \norm{ a_{(k)} }_{C_{T,x}^N} \lesssim  \ell_q^{-6N-7}(1+J_q)^{N+2}, \label{est-akbn}
	\end{align}
where
\begin{align}\label{def-jq}
		J_q:= \norm{\mathring{R}_{q} }_{C_{T}L^1_{x}},
\end{align}
and the implicit constants are deterministic and universal.
\end{lemma}

Now, the principal part of the velocity perturbations is defined by
	\begin{align}\label{pv}
		w_{q+1}^{(p)} &:= \sum_{k \in \Lambda} a_{(k)}W_{(k)}.
	\end{align}
By the symmetric nonlinearity  in the velocity equation
and identity \eqref{velcancel},
we have
\begin{align}  \label{vel oscillation cancellation calculation}
	 w_{q+ 1}^{(p)} \otimes  w_{q+ 1}^{(p)}+ \mathring{R}_{\ell_q} =& (\rho+\gamma_{q}) Id   +\sum_{k \in \Lambda} a_{(k)}^2\P_{\neq0}( W_{(k)}\otimes W_{(k)}) .
\end{align}

\paragraph{\bf Incompressibility corrector}
Define the incompressibility corrector by
\begin{align}  \label{wqc}
	 w_{q+1}^{(c)}
		&:=   \sum_{k\in \Lambda  }  \( \nabla a_{(k)} \times  W^c_{(k)}   +a_{(k)}  \wt W_{(k)}^c  \),
\end{align}
where $W^c_{(k)}$ and  $\wt W_k^c $  are given by \eqref{corrector vector}.
It follows from \eqref{wcwc} that
\begin{align} \label{div free velocity}
		&  w_{q+1}^{(p)}+ w_{q+1}^{(c)}
		=   \curl (  \sum_{k \in \Lambda} a_{(k)}  W_{(k)}^c ).
	\end{align}
In particular,
$ w_{q+1}^{(c)}$ permits to correct the pricipal part to be
divergence free:
\begin{align} \label{div-wpc-dpc-0}
	\div (w_{q+1}^{(p)} + w_{q +1}^{(c)}) =  0.
\end{align}

\paragraph{\bf Temporal corrector}
In order to balance the high spatial oscillations in \eqref{vel oscillation cancellation calculation}
caused by the spatial concentration function $\psi_{(k_1)}$,
one also needs the temporal corrector $w_{q+1}^{(t)}$
defined by
	\begin{align}
		&w_{q+1}^{(t)} := -{\mu}^{-1}  \sum_{k\in \Lambda} \P_{H}\P_{\neq 0}(a_{(k)}^2 \psi_{(k_1)}^2 \phi_{(k)}^2  k_1),\label{veltemcor}
	\end{align}
where $\P_{H}$ denotes the Helmholtz-Leray projector, i.e.,
$\P_{H}=\Id-\nabla\Delta^{-1}\div. $

Then, via the Leibniz rule, one has
\begin{align} \label{utem}
	&\partial_{t} w_{q+1}^{(t)}+    \sum_{k \in \Lambda  }  \P_{\neq 0}
	\(a_{(k)}^{2} \div (W_{(k)} \otimes W_{(k)})\) \notag \\
	=&- {\mu}^{-1}   \sum_{k \in \Lambda}  \P_{H} \P_{\neq 0} \partial_{t}
	\left(a_{(k)}^{2} \psi_{(k_1)}^{2} \phi_{(k)}^{2}   k_1\right)
	  + {\mu}^{-1}     \sum_{k \in \Lambda}  \P_{\neq 0}
	\(a_{(k)}^{2} \partial_{t} (\psi_{(k_1)}^{2} \phi_{(k)}^{2}) k_1\)  \nonumber \\
	=&(\nabla\Delta^{-1}\div)  {\mu}^{-1}  \sum_{k \in \Lambda  }  \P_{\neq 0} \partial_{t}
	\(a_{(k)}^{2} \psi_{(k_1)}^{2} \phi_{(k)}^{2} k_1\)
	  - {\mu}^{-1}    \sum_{k \in \Lambda  }  \P_{\neq 0}
\(\partial_{t}( a_{(k)}^{2})
 (\psi_{(k_1)}^{2} \phi_{(k)}^{2}) k_1\).
\end{align}

\paragraph{\bf Velocity fields}
We are now in position to define the velocity perturbations at level $q+1$. We note that by assumptions \eqref{asp-0}-\eqref{asp-q}, the Reynolds error at level $q$ already satisfies the decay estimate \eqref{rl1s} on the interval $[T_{q+2},T_{q+1}]$. Hence, the perturbations are only needed to be added on the interval $[T_{q+1},T]$ to decrease the old Reynolds error. In order to also maintain the initial value of the approximate solution $u_{q+1}$ at level $q+1$, we use the deterministic smooth cut-off function $\chi_{q+1}$ which satisfies
\begin{align*}
	\chi_{q+1}(t)=\begin{cases}
		1,\quad t\in [T_{q+1},T],\\
		0,\quad t\in [0,\frac{T_{q+2}+T_{q+1}}{2}],\\
	\end{cases}
\end{align*}
and for $0\leq N\leq 2$
\begin{align}\label{est-chiq}
\|	\chi_{q+1}\|_{C^N_t}\lesssim \ell_{q}^{-N}.
\end{align}
The velocity perturbation  $w_{q+1}$ at level $q+1$ is now defined by	\begin{align}
	w_{q+1} &:= \chi_{q+1}w_{q+1}^{(p)}+\chi_{q+1} w_{q+1}^{(c)}+ \chi_{q+1}^2  w_{q+1}^{(t)}.
		\label{velocity perturbation}
	\end{align}
To ease notations,
set
\begin{align}
	\wt w_{q+1}^{(p)}:= \chi_{q+1}w_{q+1}^{(p)}, \quad \wt w_{q+1}^{(c)} :=\chi_{q+1} w_{q+1}^{(c)}\quad \wt w_{q+1}^{(t)}:=\chi_{q+1}^2  w_{q+1}^{(t)}.
\end{align}
Then,
\begin{align}
w_{q+1} = \wt w_{q+1}^{(p)} +\wt  w_{q+1}^{(c)} +\wt w_{q+1}^{(t)}.
\end{align}
The new velocity fields at level $q+1$ is thus defined by
\begin{align}  \label{q+1 iterate}
		  u_{q+1}:= u_{q} + w_{q+1}.
\end{align}

The key estimates of velocity perturbations are summarized in Lemma \ref{totalest}.

\begin{lemma}  [Estimates of velocity perturbations] \label{totalest}
	For any $\rho \in(1,\9)$, $0\leq N\leq 1$,
the following estimates hold:
	\begin{align}
		&\norm{ \na^N w_{q+1}^{(p)}}_{C_TL^\rho_x } \lesssim \ell_q^{-2}(1+\j)^{ N+2}  \lambda^N(\rp \rs^2 )^{\frac{1}{\rho}-\frac12} ,\label{uprinlp}\\
		&\norm{ \na^N w_{q+1}^{(c)} }_{C_TL^\rho_x  } \lesssim \ell_q^{-2}(1+\j)^{ N+3}  \lambda^N\rs \rp ^{-1}(\rp \rs^2  )^{\frac{1}{\rho}-\frac12}, \label{divcorlp}\\
		&\norm{ \na^Nw_{q+1}^{(t)}}_{C_TL^\rho_x  }\lesssim  \ell_q^{-4}(1+\j)^{ N+4}\mu^{-1}  \lambda^N(\rp \rs^2 )^{\frac{1}{\rho}-1},\label{temcorlp}
	\end{align}
where $\j$ is defined in \eqref{def-jq}. Moreover, we have
\begin{align}
& \norm{ w_{q+1}^{(p)} }_{C_{T,x}^1 }  + \norm{ w_{q+1}^{(c)} }_{C_{T,x}^1 }+\norm{ w_{q+1}^{(t)} }_{C_{T,x}^1 }
\lesssim \lambda^{4}(1+\j)^{ 6}.\label{principal c1 est}
\end{align}
\end{lemma}

\begin{proof}
By Lemmas \ref{buildingblockestlemma} and \ref{mae}, for any $\rho \in (1,+\infty)$,
\begin{align}\label{uplp}
\norm{\nabla^N w_{q+1}^{(p)} }_{C_TL^\rho_x }  \lesssim& \, \sum_{k \in \Lambda }
\sum\limits_{N_1+N_2 = N}
\|a_{(k)}\|_{C_TC^{N_1}_{x}}
\norm{ \nabla^{N_2}  W_{(k)} }_{C_TL^\rho_x } \notag \\
\lesssim&\,	\ell_q^{-2}(1+\j)^{  N+2}  \lambda^N(\rp \rs^2)^{\frac{1}{\rho}-\frac12},
\end{align}
which verifies \eqref{uprinlp}.

Moreover, using \eqref{b-beta-ve}, \eqref{ew}, \eqref{wqc} and Lemma \ref{mae}
we have
\begin{align} \label{lp vel corrector estimate}
 \norm{\na^N w_{q+1}^{(c)}}_{C_TL^\rho_x}
	 \lesssim &  \sum_{k\in \Lambda }
	\left\|\na^N \( \nabla a_{(k)} \times W^c_{(k)}  +a_{(k)}  \wt W_{(k)}^c  \) \right\|_{C_TL^\rho_x}  \notag\\
	\lesssim&
	\sum\limits_{k\in \Lambda } \sum_{N_1+N_2=N}
        \( \norm{ \nabla^{N_1+1} a_{(k)} }_{C_{T,x}} \norm{\na^{N_2} W^c_{(k)}}_{C_TL^{\rho}_x  }
	 +  \norm{ a_{(k)} }_{C_{T,x}^{N_1}} \norm{ \na^{N_2}\wt W^c_{(k)}}_{C_TL^\rho_x}  \)  \nonumber   \\
	 \lesssim&  \ell_q^{-2}(1+\j)^{  N+3}  \lambda^N\rs \rp ^{-1}(\rp \rs^2)^{\frac{1}{\rho}-\frac12}.
\end{align}
	
Regarding the temporal correctors,
by \eqref{veltemcor}, Lemmas \ref{buildingblockestlemma}, \ref{mae}
and the boundedness of operators $P_{\not =0}$ and $\mathbb{P}_H$ in $L^\rho$ spaces, $\rho\in (1,\9)$, we have
\begin{align}
\label{lp vel time estimate}
\norm{\na^N w_{q+1}^{(t)} }_{C_TL^\rho_x}
\lesssim & \,\mu^{-1}    \sum_{k \in \Lambda }
\sum_{N_1+N_2+N_3 =N}
\|\nabla^{N_1}(a_{(k)}^2)\|_{C_{T,x} }\norm{  \na^{N_2} (\psi_{(k_1)}^2) }_{C_TL^{\rho}_x }\norm{  \na^{N_3}(\phi_{(k)}^2) }_{L^{\rho}_x }   \notag  \\
\lesssim & \, \ell_q^{-4}(1+\j)^{ N+4}\mu^{-1}  \lambda^N(\rp \rs^2)^{\frac{1}{\rho}-1},
\end{align}
which yields  \eqref{temcorlp}.

It remains to prove the $C^1$-estimate \eqref{principal c1 est} of the perturbations. By Lemmas \ref{buildingblockestlemma} and \ref{mae},
\begin{align} \label{wprincipal c1 est}
	\norm{ w_{q+1}^{(p)} }_{C_{T,x}^1}
	\lesssim&  \, \sum_{k \in \Lambda }\sum_{N_1+N_2\leq 1}
	\|a_{(k)}\|_{C_{T,x}^{N_1} }\norm{   W_{(k)}  }_{C_{T,x}^{N_2}}  \notag \\
	\lesssim&\, \ell_q^{-13}(1+\j)^{3}( \sigma\rp^{-1}\mu )(\rp\rs^2)^{-\frac 12}  \lesssim \lambda^{4}(1+\j)^{3},
\end{align}
where we also used \eqref{b-beta-ve} and \eqref{larsrp} in the last step.
Similarly, we have
\begin{align} \label{uc c1 est}
	\norm{ w_{q+1}^{(c)} }_{C_{T,x}^1 }
	& \lesssim   \sum_{k\in \Lambda }
	\left\|  \nabla a_{(k)} \times W^c_{(k)}   +a_{(k)}  \wt W_{(k)}^c  \right\|_{C_{T,x}^1 }  \notag \\
	& \lesssim   \sum_{k \in \Lambda } (\| \nabla a_{(k)}\|_{C_{T,x}^1}\norm{ W^c_{(k)}  }_{C_{T,x}^1}+
\|  a_{(k)}\|_{C_{T,x}^1}\norm{ \wt W_{(k)}^c  }_{C_{T,x}^1})  \notag \\
	& \lesssim \lambda^{4}(1+\j)^{4}.
\end{align}
Moreover, by Sobolev's embedding $W^{1,6}(\mathbb{T}^3) \hookrightarrow L^\9(\mathbb{T}^3)$
and the boundedness of  $\mathbb{P}_H \mathbb{P}_{\not =0}$ in the space $L^6_x$,
\begin{align} \label{ut c1 est}
\norm{ w_{q+1}^{(t)} }_{C_{T,x}^1 }
& \lesssim \mu^{-1} \sum_{k \in \Lambda}
\sum_{0\leq N_1+N_2 \leq 1}\|a_{(k)}^2 \psi^2_{(k_1)} \phi^2_{(k)}\|_{C^{N_1}_t W^{N_2+1,6}_x} \notag \\
&\lesssim \ell_q^{-14}(1+\j)^{6} \mu^{-1} \lambda ( \sigma\rp^{-1}\mu ) (\rp\rs^2)^{-\frac56}  \notag   \\
& \lesssim \lambda^{4}(1+\j)^{6},
\end{align}
where the last step was due to Lemmas \ref{buildingblockestlemma} and \ref{mae}. 

Thus, combining \eqref{wprincipal c1 est}-\eqref{ut c1 est} altogether and using \eqref{larsrp}
we obtain \eqref{principal c1 est} and finish the proof.
\end{proof}

\subsubsection{\bf Verification of inductive estimates
for velocity flow} \label{Subsec-induc-vel-mag}

We aim to verify the inductive estimates \eqref{ul2}, \eqref{uc}, \eqref{energy-est} and \eqref{u-B-L2tx-conv}
for the new velocity flow at level $q+1$.
To ease notations, let us set
\begin{align*}
	& w_{q+1}^{(c)+(t)}:=  w_{q+1}^{(c) }+ w_{q+1}^{ (t) }, \quad \wt w_{q+1}^{(c)+(t)}:= \wt w_{q+1}^{(c) }+ \wt w_{q+1}^{ (t) }.
\end{align*}

Applying the refined H\"older estimate in Lemma~\ref{Decorrelation1} with $f= a_{(k)}$, $g = \psi_{(k_1)}\phi_{(k)}$ and using Lemma~\ref{buildingblockestlemma} we get
\begin{align}
	\label{Lp decorr vel}
	\norm{ w^{(p)}_{q+1}}_{C_T L^2_{x}}
	&\lesssim \sum\limits_{k\in \Lambda }
       \Big(\|a_{(k)} \|_{C_TL^2_{x}}\norm{ \psi_{(k_1)}\phi_{(k)}}_{C_TL^2} +\lambda^{-\frac{1}{14}}\|a_{(k)}\|_{C^1_{T,x}} \norm{  \psi_{(k_1)}\phi_{(k)} }_{C_TL^2}\Big).
\end{align}
Then, by \eqref{la}, \eqref{b-beta-ve} and Lemma \ref{mae},
we get
\begin{align}
	\label{Lp-wdp-1}
	\norm{w^{(p)}_{q+1} }_{C_TL^2_{x}}
	&\lesssim  J_q^{\frac{1}{2}}+ \delta_{q+ 1}^{\frac12}+\ell_q^{-13}(1+\j)^{3}\lambda^{-\frac{1}{14}}_{q+1}.
\end{align}
For the velocity perturbation $ w_{q+1}$,
using \eqref{est-chiq}, \eqref{velocity perturbation}, \eqref{Lp decorr vel} and Lemma \ref{totalest}  we obtain
\begin{align}  \label{e3.41}
	\norm{ w_{q+1} }_{C_TL^2_{x}}&\leq \|\chi_{q+1}\|_{C_T} \norm{ w_{q+1}^{(p)} }_{C_TL^2_{x}} +\|\chi_{q+1}\|_{C_T}\norm{ w_{q+1}^{(c)} }_{C_TL^2_{x}}+\|\chi_{q+1}^2\|_{C_T}\norm{  w_{q+1}^{(t)} }_{C_TL^2_{x}}\notag\\
     &\lesssim  J_q^{\frac{1}{2}}+ \delta_{q+ 1}^{\frac12}+\ell_q^{-13}(1+\j)^{3}\lambda^{-\frac{1}{14}}_{q+1}+   \ell_q^{-2} (1+\j)^{3}r_{\perp} r_{\parallel}^{-1}+  \ell_q^{-4}(1+\j)^4 \mu^{-1} (\rp\rs^2)^{-\frac12}\notag\\
     & \lesssim J_q^{\frac{1}{2}}+ \delta_{q+1}^{\frac{1}{2}} +\ell_q^{-13}(1+\j)^{4}\lambda^{-\frac{1}{14}}_{q+1}.
\end{align}
Then, taking into account \eqref{uc},  \eqref{principal c1 est},
we get
\begin{align}\label{ver-u-l2-1}
	\norm{u_{q+1} -u_{q} }_{ C_TL^2_{x}}
	& =  \norm{ w_{q+1}}_{C_TL^2_{x}}  \lesssim J_q^{\frac{1}{2}}+ \delta_{q+1}^{\frac{1}{2}} +\ell_q^{-13}(1+\j)^{4}\lambda^{-\frac{1}{14}}_{q+1},
\end{align}
which along with \eqref{rl1s} and \eqref{rl1b} yields
\begin{align}\label{ver-u-l2-1-e}
\|  u_{q+1} -u_{q}  \|_{L^{2r}_{\Omega}C_TL^2_{x}} &\lesssim \norm{\mathring{R}_{q} }_{L^r_{\Omega}C_TL^1_{x}}^{\frac12} + \delta_{q+1}^{\frac{1}{2}}+\ell_q^{-13} \laq^{-\frac{1}{14}}\| \rr_q \|_{L^{8r}_{\Omega}C_TL^1_x}^{4}\notag\\
	&\lesssim  \delta_{q+1}^{\frac{1}{2}} + \ell_q^{-13} \laq^{-\frac{1}{14}}(12^q\cdot 8\cdot 8rL^2)^{4\cdot12^q}\notag\\
	&\lesssim \delta_{q+1}^{\frac{1}{2}}.
\end{align}
Here we also used $\lambda_{q}\geq (12^q\cdot 8\cdot 8rL^2)^{4\cdot12^q}$ in the last step.
This verifies \eqref{u-B-L2tx-conv}.

Moreover, we derive
\begin{align}\label{verubl2}
	&  \| u_{q+1} \|_{L^{2r}_{\Omega}C_TL^2_x}\leq \| u_q \|_{L^{2r}_{\Omega}C_TL^2_x}+\| u_{q+1} -u_{q}  \|_{L^{2r}_{\Omega}C_T L^2_{x}} \lesssim \sum_{n=0}^{q}\delta_{n+1}^{\frac{1}{2}},
\end{align}
and thus,
\eqref{ul2} is verified at level $q+1$.

Regarding the $C^1$-estimate \eqref{uc} of the velocity, by  \eqref{spa-regu-pointwise}, \eqref{spa-regu}, \eqref{tem-regu-1}, Sobolev's embedding $H^2_x\hookrightarrow L^\infty_x$ and Lemma~\ref{totalest} we derive
\begin{align}\label{verifyuc1}
     \norm{u_{q+1}}_{C^1_{[T_{q+3},T],x}}&  \lesssim  \norm{u_{q}}_{C^1_{[T_{q+3},T],x}}+\norm{ w_{q+1}}_{C^1_{T,x}}\notag\\
 &\lesssim \norm{u_{0}}_{C^1_{[T_{q+3},T],x}}+ \sum_{n=1}^{q}\norm{w_{n}}_{C^1_{T,x}} + \norm{w_{q+1}}_{C^1_{T,x}} \notag\\
 &\lesssim \norm{u_{0}}_{C^1_{[T_{q+3},T],x}}+ \sum_{n=1}^{q} \lambda_{n}^{4}(1+J_{n-1})^{ 6}+\lambda_{q+1}^{4}(1+\j)^{ 6}.
\end{align}
Note that, by \eqref{spa-regu} and \eqref{tem-regu-1}, taking $a$ large enough we have
\begin{align}\label{est-u0-c1}
    \| u_0 \|_{ L^{m}_{\Omega}C_{[T_{2},T],x}^{1}}& \lesssim \| \p_t u_0 \|_{L^{m}_{\Omega}C_{[T_{2},T],x} }+\| u_0 \|_{ L^{m}_{\Omega}C_{[T_{2},T]}C_x^{1}}\notag\\
    &\lesssim \| \p_t u_0 \|_{L^{m}_{\Omega}C_{[T_{2},T]}H^2_x }+\| u_0 \|_{L^{m}_{\Omega} C_{[T_{2},T]}H_x^{3}}\notag\\
     &\lesssim (1+T_2^{-2})(\|v_0 \|_{L^{m}_{\Omega}L^2_x }+\|v_0 \|_{L^{2m}_{\Omega}L^2_x }^2)+(1+T_2^{-\frac32})\| v_0 \|_{L^{m}_{\Omega}L_x^{2}}\notag\\
     &\lesssim (1+T_2^{-2})(M+M^2)\leq \lambda_0^4
\end{align}
and, similarly
\begin{align}\label{u0-est-tq2}
    \| u_0 \|_{ L^{m}_{\Omega}C_{[T_{q+3},T],x}^{1}}\lesssim (1+T_{q+3}^{-2})(\|v_0 \|_{L^{m}_{\Omega}L^2_x }+\|v_0 \|_{L^{2m}_{\Omega}L^2_x }^2)\lesssim (1+T_{q+2}^{-2})(M+M^2). 
\end{align}
Hence, taking the $m$-th moment and using \eqref{ini-bdd}, \eqref{a-big-lambda}, \eqref{rl1b} and \eqref{u0-est-tq2},  we get for $a$ sufficiently large that
\begin{align}\label{verifyuc1-m}
	 \| u_{q+1}\|_{L^{m}_{\Omega}C^1_{[T_{q+3},T],x}}
  & \lesssim \lambda_{q+1}(M+M^2)+ \sum_{n=1}^{q}\lambda_{n}^{4}(12^{n-1}\cdot 8\cdot 6mL^2)^{6\cdot 12^{n-1}}+  \lambda^{4}_{q+1}(12^q\cdot 8\cdot 6mL^2)^{6\cdot 12^q} \notag\\
 & \lesssim \lambda^{4}_{q+1}(12^q\cdot 8\cdot 6mL^2)^{6\cdot 12^q},
\end{align}
which verifies \eqref{uc} at level $q+1$.

\subsubsection{\bf Verification of inductive estimates
for energy} \label{Subsec-induc-energy}

Now let us treat the delicate energy iteration estimate \eqref{energy-est}. 
In the case where $t\in [T_{q+2},\frac{T_{q+2}+T_{q+1}}{2}]$, we have $u_{q+1}(t)=u_0(t)$. Hence, 
\begin{align}\label{est-e-110}
&\quad e(t) -\E\|u_{q+1}(t)+z_{q+1}(t)\|_{L_x^2}^2\notag\\
&= e(t) -\E\|u_{0}(t)+z(t)\|_{L_x^2}^2 
+2\E \int_{\T^3} u_0(t)\cdot ( z-z_{q+1})(t) \d x+ \E \int_{\T^3} ( z+ z_{q+1})\cdot ( z-z_{q+1})(t) \d x.
\end{align}
Using \eqref{def-ziq} and Proposition~\ref{Prop-noise}, we get
\begin{align}\label{est-e-111}
\left|2\E \int_{\T^3} u_0(t)\cdot ( z-z_{q+1})(t) \d x \right|
& \lesssim \|u_0\|_{L^2_{\Omega}C_TL^2} \| \mathbb{P}_{>\lambda_{q+2}^{\frac{\ve}{8}}}z \|_{L^2_{\Omega}C_TL^2_x} \notag\\
&\lesssim \lambda_{q+2}^{-\frac{\ve}{8}(1-\delta)} \|u_0\|_{L^2_{\Omega}C_TL^2_x}\| z \|_{L^2_{\Omega}C_TH^{1-\delta}_x} 
\notag \\ 
&\leq \frac{1}{8}\delta_{q+3}
\end{align}
and
\begin{align}\label{est-e-112}
\left|\E \int_{\T^3}  ( z+ z_{q+1})(t) \cdot ( z-z_{q+1})(t) \d x \right|& \lesssim \|z+ z_{q+1} \|_{L^2_{\Omega}C_TL^2} \| \mathbb{P}_{>\lambda_{q+2}^{\frac{\ve}{8}}}z \|_{L^2_{\Omega}C_TL^2_x}\notag\\
&\lesssim \lambda_{q+2}^{-\frac{\ve}{8}(1-\delta)} \|z\|_{L^2_{\Omega}C_TL^2_x}\| z \|_{L^2_{\Omega}C_TH^{1-\delta}_x}\notag\\
&\leq \frac{1}{8}\delta_{q+3},
\end{align}
for $a$ large enough. Therefore, by \eqref{asp-0-energy}, \eqref{asp-e-q}, \eqref{est-e-110}, \eqref{est-e-111} and \eqref{est-e-112} we have
\begin{align}\label{e-t-1}
\frac12\delta_{q+3}\leq  e(t) -\E\|u_{q+1}(t)+z_{q+1}(t)\|_{L_x^2}^2 \leq 2\delta_{q+ 2}.
\end{align}

Next, we treat the more delicate case where $t\in [\frac{T_{q+2}+T_{q+1}}{2},T]$.  By definition,
\begin{align}\label{ver-en-diff}
	& 	 e(t) -\E\|u_{q+1}(t)+z_{q+1}(t)\|_{L_x^2}^2 -f_q(t)\notag\\
   =& \Big( e(t) -\E\|u_{q}(t)+z_{q}(t)\|_{L_x^2}^2 -f_q(t)-\E\|\wt w_{q+1}^{(p)}(t)\|_{L^2_x}^2 \Big)   \notag \\
    &  + \Big( -2 \E \|  (\wt w^{(p)}_{q+1} \cdot \wt w_{q+1}^{(c)+(t)})(t)\|_{L^1_x}- \E\|( \wt w_{q+1}^{(c)+(t)})(t)\|_{ L^2_x}^2
	  - 2\E \int_{\T^3} (u_q(t)+z_{q+1}(t))\cdot w _{q+1}(t) \d x \notag\\
   &\qquad  -  \E \int_{\T^3} |z_{q+1}(t)|^2 -|z_{q}(t)|^2\d x+ 2\E \int_{\T^3}  u_q(t)\cdot z_{q}(t)- u_q(t) \cdot z_{q+1}(t)\d x\Big)\notag\\
	:=& I_1+I_2,
\end{align}
where $f_q$ is defined by \eqref{def-fq}.
Let us estimate $I_1$ and $I_2$ separately
in the following.

Concerning the estimate of $I_2$, in view of \eqref{e3.41}-\eqref{ver-u-l2-1-e} and Lemma~\ref{totalest}, we get
\begin{align}\label{ver-en-2}
	2\E \|\wt  w^{(p)}_{q+1}\wt  w_{q+1}^{(c)+(t)}\|_{C_TL^1_x}
	&\lesssim   \|  w^{(p)}_{q+1} \|_{L^2_{\Omega}C_TL^2_x}  \| w_{q+1}^{(c)+(t)} \|_{L^2_{\Omega}C_TL^2_x}\notag\\
 &\lesssim \delta_{q+1}^{\frac{1}{2}}( \ell_q^{-2} \laq^{-\frac27} (12^q\cdot 48 L^2)^{3\cdot 12^q} + \ell_q^{-4} \laq^{-\frac17} (12^q\cdot 64 L^2)^{4\cdot 12^q})\notag\\
 &\leq \frac{1}{96}\delta_{q+3}.
\end{align}
For the second term of $I_2$, similarly to \eqref{ver-en-2}, we have
\begin{align}\label{ver-en-3}
\E\|\wt w_{q+1}^{(c)+(t)}\|_{C_TL^2_x}^2
&\lesssim  \E\|  w^{(c)}_{q+1}\|_{C_TL^2_x}^2+\E\|  w^{(t)}_{q+1}\|_{C_TL^2_x}^2 \leq  \frac{1}{96}\delta_{q+3}.
\end{align}
Regarding the third term in $I_2$, by \eqref{uc}, Proposition~\ref{Prop-noise} and Lemma~\ref{totalest},
\begin{align}\label{ver-en-4}
	 \E\| (u_q+z_{q+1}) \cdot w _{q+1} \|_{C_{[\frac{T_{q+2}+T_{q+1}}{2},T]}L^1_x}
	&\lesssim \|  u_q+z_{q}\|_{L^2_{\Omega}C_{[\frac{T_{q+2}+T_{q+1}}{2},T],x}  }\| w_{q+1} \|_{L^2_{\Omega}C_TL^1_x }\notag\\
 &\quad+\| z_{q+1}-z_{q}\|_{L^2_{\Omega}C_TL_x^2 }\| w_{q+1} \|_{L^2_{\Omega}C_TL^2_x } \notag\\
	&\lesssim (\la^5+ \laq^{\frac{\ve}{8} }L) (\ell_q^{-2}\lambda_{q+1}^{-\frac87+\ve}(12^q\cdot 32 L^2)^{2\cdot12^q} )+  \laq^{-\frac{\ve}{8}(1-\delta)}L \delta_{q+1}^{\frac12} \notag\\
	& \leq  \frac{1}{96}\delta_{q+3},
\end{align}
where we also choose $a$ sufficiently large such that $\lambda_{q}\geq (12^{q}\cdot 8L^2)^{ 12^{q}}$.
Finally, using \eqref{ul2}, \eqref{uc} and Proposition~\ref{Prop-noise}, the
last two terms of $I_2$ can be bounded by:
\begin{align}\label{ver-en-5}
	&\quad \E \|(z_{q+1}+z_{q})\cdot(z_{q+1}-z_{q})(t)\|_{L^1_x} + 2\E\|(z_{q+1}-z_{q} )(t)\cdot u_q(t)\|_{L^1_x}\notag\\	
	&\lesssim  \| (z_{q+1}-z_{q})(t)\|_{L^2_{\Omega}L^2_x } (\| u_q(t)\|_{L^2_{\Omega}L^2_x }+\| z_{q}(t) \|_{L^2_{\Omega}L^2_x } +\| z_{q+1}(t)\|_{L^2_{\Omega}L^2_x }  ) \notag\\
	&\lesssim \laq^{-\frac{\ve}{8}(1-\delta)}L( \sum\limits_{n=1}^q \delta_{n}^{\frac12}+L)\leq \frac{1}{96}\delta_{q+ 3}.
\end{align}
Thus, combining \eqref{ver-en-2}-\eqref{ver-en-5} altogether, we arrive at
\begin{align}\label{ver-en-i2}
	|I_2|\leq \frac{1}{24}\delta_{q+ 3}.
\end{align}

It remains to estimate the first term $I_1$ on the right-hand side of \eqref{ver-en-diff}.
Using \eqref{vel oscillation cancellation calculation} and the fact that $\mathring{R}_{\ell_q}$ is trace free, we get
\begin{align*}
    |\wt w_{q+1}^{(p)}(t)|^2= 3\cq^2 (\rho+\gamma_q)+ \sum_{k \in \Lambda  } \cq^2a_{(k)}^2(t)P_{\neq 0}(\left|W_{(k)}(t)\right|^2).
\end{align*}
Hence, we have
\begin{align}\label{ver-en-diff-3}
	I_1 	& =  e(t) -\E\|u_{q}(t)+z_{q}(t)\|_{L_x^2}^2 -f_q(t)-\E\|\wt w_{q+1}^{(p)}(t)\|_{L^2_x}^2\notag\\
	& = \( e(t) -\E\|u_{q}(t)+z_{q}(t)\|_{L_x^2}^2 -f_q(t)-\chi_{q+1}^2(t)\big( e -f_q-\E\|u_{q}+z_{q}\|_{L_x^2}^2 \big )*_{t}  \varphi_{\ell_q} (t) \) \notag\\
	&\quad+\(\chi_{q+1}^2(t)\E\int_{\T^3}6\ve_u^{-1} (\ell_q^2+ |\mathring{R}_{\ell_q}(t) |^2)^{\frac12}+\sum_{k \in \Lambda  } a_{(k)}^2(t) P_{\neq 0}(\left|W_{(k)}(t)\right|^2) \d x\)\notag\\
	&:=I_{11}+I_{12}.
\end{align}

In order to estimate the right-hand side of \eqref{ver-en-diff-3}, we divide $[\frac{T_{q+2}+T_{q+1}}{2},T]$
into three regimes
$$[\frac{T_{q+2}+T_{q+1}}{2},T]= \bigcup_{i=1}^3 J_i,$$
where $J_1:= [T_{q+1}+\ell_q,T]$, $J_2:=[T_{q+1}, T_{q+1}+\ell_q]$ and $J_3:=[\frac{T_{q+2}+T_{q+1}}{2},T_{q+1}]$.

{\bf $\bullet$ Estimate of $J_1$ regime.} Let us first treat the case where $t\in J_1$. In this case, by definition, we have
\begin{align*}
    \chi_{q+1}(t)=1,\quad f_q(t)- (f_q*_{t}\varphi_{\ell_q}) (t)=0.
\end{align*}
Regarding the second term $I_{12}$ on the right-hand side of \eqref{ver-en-diff-3}, using Lemmas~\ref{mae} and \ref{lem-mean} we have
\begin{align}\label{est-en-high}
	\E\sum_{k \in \Lambda   } \int_{\T^3} a_{(k)}^2 P_{\neq 0}( \left|W_{(k)}\right|^2) \d x & \lesssim \E\sum_{k \in \Lambda   } \sigma^{-1}\|\nabla (a_{(k)}^2)\|_{C_{T,x}} \|P_{\neq 0}( \left|W_{(k)}\right|^2)\|_{L^1_x} \notag\\
    & \lesssim \laq^{-\frac17}\ell_q^{-20}( 12^q\cdot 40L^2)^{5\cdot 12^q}\leq \frac{1}{48} \delta_{q+ 3} .
\end{align}
By \eqref{rl1s}, we get
\begin{align}\label{prin-end}
	6\ve_u^{-1}( \ell_q+  \E\|\mathring{R}_{\ell_q}(t)\|_{L^1_x})\leq 6\ve_u^{-1}(  \ell_q+ c_*\delta_{q+2}) \leq \frac{1}{48} \delta_{q+ 2}.
\end{align}
Hence, combining \eqref{est-en-high} and \eqref{prin-end} altogether we obtain
\begin{align}\label{est-j2}
	|I_{12}| \leq \frac{1}{24} \delta_{q+2}.
\end{align}
For the first term $I_{11}$ on the right-hand side of \eqref{ver-en-diff-3}, by \eqref{uc}, Proposition~\ref{Prop-noise} and the standard mollification estimates, we get
\begin{align}\label{est-moll-uqzq}
	&\quad(\E\|u_q(t)+z_{q}(t)\|_{L^2}^2)*_{t} \varphi_{\ell_q}- \E\|u_q(t)+z_{q}(t)\|_{L^2}^2  \notag\\
	&\lesssim \ell_q^{\frac12-\delta} \E( \|u_q+z_{q}\|_{C^{\frac12-\delta}_{[t-\ell_q,t]}L^2_x}\|u_q+z_{q}\|_{C_{[t-\ell_q,t]}L^2_x})\notag\\
	&\lesssim \ell_q^{\frac12-\delta} (\la^4(12^{q-1}\cdot 8\cdot 12L^2)^{6\cdot 12^{q-1}} +L )(\sum\limits_{n=1}^q \delta_{n}^{\frac12}   +L)\leq \frac{1}{48} \delta_{q+3},
\end{align}
and 
\begin{align}\label{est-e}
	|    e(t) - (e)*_t \varphi_{\ell_q}(t)|\lesssim \ell_q \|e\|_{C_T^1}\lesssim \frac{1}{48} \delta_{q+3},
\end{align}
for  $a$ sufficiently large.
Thus \eqref{est-moll-uqzq} and \eqref{est-e} together yield
\begin{align}\label{est-j1}
    |I_{11}|\leq \frac{1}{24}\delta_{q+ 3}.
\end{align}
Putting \eqref{ver-en-diff}, \eqref{ver-en-i2}, \eqref{ver-en-diff-3}, \eqref{est-j2} and \eqref{est-j1} altogether we obtain
\begin{align*}
|  e(t) -\E\|u_{q+1}(t)+z_{q+1}(t)\|_{L_x^2}^2 -\frac12 \delta_{q+2}| \leq \frac{1}{24}\delta_{q+2}+ \frac{1}{12}\delta_{q+3},
\end{align*}
which yields that
\begin{align}\label{e-t-2}
    \frac14 \delta_{q+2}\leq  e(t) -\E\|u_{q+1}(t)+z_{q+1}(t)\|_{L_x^2}^2 \leq \delta_{q+2}.
\end{align}

{\bf $\bullet$ Estimate of $J_2$ regime.} In the case where $t\in J_2$, by definition,
\begin{align*}
    \chi_{q+1}(t)=1,\quad 0\leq f_q(t)- (f_q(t))*_{t}\varphi_{\ell_q}\leq \frac{1}{2}\delta_{q+2}-\frac34\delta_{q+3}.
\end{align*}
Since \eqref{est-j2}, \eqref{est-moll-uqzq} and \eqref{est-e} still hold in this case, we deduce that
\begin{align}\label{i11-t-3}
-\frac{1}{24} \delta_{q+3} \leq I_{11} \leq \frac{1}{2} \delta_{q+2},\quad |I_{12}| \leq \frac{1}{24}\delta_{q+2}.
\end{align}
Therefore, plugging \eqref{ver-en-i2}, \eqref{est-j2} and \eqref{i11-t-3} into \eqref{ver-en-diff} we obtain
\begin{align}\label{e-t-3}
    \frac14 \delta_{q+2}\leq e(t) -\E\|u_{q+1}(t)+z_{q+1}(t)\|_{L_x^2}^2 \leq \frac32\delta_{q+2}.
\end{align}

{\bf $\bullet$ Estimate of $J_3$ regime.} In the case where $t\in J_3$, we note that there is no perturbation added to the approximate solution $u_{q-1}$ when $t\in [0,T_{q+1}]$ at level $q$, hence for $t\in J_3$ we have $\rr_q(t)= \rr_0(t)$ and $u_q(t) = u_0(t)$.

Similarly to \eqref{est-en-high} and \eqref{prin-end}, using the choice of $T_{q+1}$ in \eqref{asp-0} and \eqref{asp-q} we get
\begin{align}\label{est-i12-3}
   |I_{12}| &\leq \chi_{q+1}^2\( \E\sum_{k \in \Lambda   } \int_{\T^3} a_{(k)}^2 P_{\neq 0}( \left|W_{(k)}\right|^2) \d x+ 6\ve_u^{-1}( \ell_q+  \E\|\mathring{R}_{\ell_q}(t)\|_{L^1_x})\)\notag\\
   &\leq \chi_{q+1}^2\(\frac{1}{48}\delta_{q+3} +6\ve_u^{-1}(  \ell_q+ c_*\delta_{q+3})\)\notag\\
   &\leq \frac{1}{24} \chi_{q+1}^2\delta_{q+3}.
\end{align}
Concerning $I_{11}$, similar to \eqref{est-e-110} we derive
\begin{align}\label{est-e-130}
&\quad  e(t) -\E\|u_{q}(t)+z_{q}(t)\|_{L_x^2}^2 \notag\\
&=  e(t) -\E\|u_{0}(t)+z(t)\|_{L_x^2}^2 
+2\E \int_{\T^3} u_0(t)\cdot ( z-z_{q})(t) \d x+ \E \int_{\T^3} ( z+ z_{q})\cdot ( z-z_{q})(t) \d x.
\end{align}
By \eqref{def-ziq} and Proposition~\ref{Prop-noise}, we get
\begin{align}\label{est-e-131}
\left|2\E \int_{\T^3} u_0(t)\cdot ( z-z_{q})(t) \d x \right|& \lesssim  \lambda_{q+1}^{-\frac{\ve}{8}(1-\delta)} \|u_0\|_{L^2_{\Omega}C_TL^2_x}\| z \|_{L^2_{\Omega}C_TH^{1-\delta}_x}\leq \frac{1}{48}\delta_{q+3},
\end{align}
and
\begin{align}\label{est-e-132}
\left|\E \int_{\T^3}  ( z+ z_{q})(t) \cdot ( z-z_{q})(t) \d x \right|
\lesssim \lambda_{q+1}^{-\frac{\ve}{8}(1-\delta)} \|z\|_{L^2_{\Omega}C_TL^2_x}\| z \|_{L^2_{\Omega}C_TH^{1-\delta}_x}
\leq \frac{1}{48}\delta_{q+3},
\end{align}
for $a$ large enough. Hence, by \eqref{asp-0-energy} and \eqref{asp-e-q} we get
\begin{align}\label{est-e-133}
    \frac{17}{24}\delta_{q+3}\leq e(t) -\E\|u_{q}(t)+z_{q}(t)\|_{L_x^2}^2 \leq \delta_{q+2}+\frac{1}{24}\delta_{q+3}.
\end{align}
Using \eqref{est-e-133}, \eqref{def-fq}, \eqref{est-moll-uqzq} and \eqref{est-e}  we get\begin{align}\label{j1-2-low}
	I_{11}& = (1-\chi_{q+1}^2)(e(t) -\E\|u_{q}(t)+z_{q}(t)\|_{L_x^2}^2)- \frac34\delta_{q+ 3} +\chi_{q+1}^2 f_q*_{t}  \varphi_{\ell_q}\notag\\
	&\quad+\chi_{q+1}^2\(  e(t)  - \E\|u_{q}(t)+z_{q}(t)\|_{L_x^2}^2  -\big(  (e - \E\|u_{q}+z_{q}\|_{L_x^2}^2)\big)*_{t}  \varphi_{\ell_q}\)\notag\\
	&\geq(1-\chi_{q+1}^2) \frac{17}{24}\delta_{q+3} - \frac34\delta_{q+ 3} +\chi_{q+1}^2 \frac34\delta_{q+ 3}- \chi_{q+1}^2\frac{1}{24} \delta_{q+3}\geq -\frac{1}{24} \delta_{q+3},
\end{align}
and
\begin{align}\label{j1-2-up}
	I_{11}
 \leq(1-\chi_{q+1}^2) (\delta_{q+2}+\frac{1}{24}\delta_{q+3})  - \frac34\delta_{q+ 3} +\chi_{q+1}^2 \frac34\delta_{q+ 3}+ \chi_{q+1}^2\frac{1}{24} \delta_{q+3}\leq \delta_{q+2}.
\end{align}
Plugging \eqref{ver-en-i2}, \eqref{est-i12-3}, \eqref{j1-2-low} and \eqref{j1-2-up} into \eqref{ver-en-diff}, we arrive at
\begin{align}\label{e-t-4}
    \frac12 \delta_{q+3}\leq  e(t) -\E\|u_{q+1}(t)+z_{q+1}(t)\|_{L_x^2}^2 \leq \frac32\delta_{q+2}.
\end{align}
Thus, \eqref{e-t-1}, \eqref{e-t-2}, \eqref{e-t-3} and \eqref{e-t-4} together verify \eqref{energy-est} at level $q+1$.

Finally, the inductive estimates \eqref{ul2}, \eqref{uc}, \eqref{energy-est}  and \eqref{u-B-L2tx-conv}  are verified.

\section{Reynolds stress}   \label{Sec-Rey-stress}
In this section, we treat the delicate Reynolds stresses
in the main iteration of Proposition \ref{Prop-Iterat}.
Let us first decompose the Reynolds stress by
\begin{align}\label{ru}
		&\displaystyle\div\mathring{R}_{q+1}  - \nabla P_{q+1}  \notag\\
		&\displaystyle = \underbrace{\chi_{q+1} \partial_t (w_{q+1}^{(p)}+w_{q+1}^{(c)} ) -\nu  \Delta w_{q+1}  +\div \big((u_{q}+\zq) \otimes w_{q+1} + w_{q+ 1} \otimes (u_{q}+\zq)\big) }_{ \div\mathring{R}_{lin} +\nabla P_{lin} }   \notag\\
		&\displaystyle\quad+ \underbrace{\div (\wt  w_{q+1}^{(p)} \otimes \wt  w_{q+1}^{(p)}  +  \chi_{q+1}^2\mathring{R}_{q} )+ \chi_{q+1}^2\partial_t w_{q+1}^{(t)} }_{\div\mathring{R}_{osc}  +\nabla P_{osc}}  \notag\\
		&\quad+\underbrace{\div\Big(\wt w_{q+1}^{(c)+(t) }\otimes w_{q+1}+ \wt  w_{q+1}^{(p) }\otimes   \wt w_{q+1}^{(c)+(t) } \Big)}_{\div\mathring{R}_{cor}  +\nabla P_{cor}}\notag\\
		&\quad +\underbrace{\div \( u_{q+1}\otimes (z_{q+1}-\zq)+(z_{q+1}-\zq) \otimes u_{q+1} +z_{q+1}\otimes z_{q+1}-\zq\otimes \zq\)}_{\div\mathring{R}_{com} +\nabla P_{com}} \notag\\
		&\quad +  \underbrace{  \partial_t\chi_{q+1}(w^{(p)}_{q+ 1}+w^{(c)}_{q+ 1})+\partial_t(\chi_{q+1}^2)w^{(t)}_{q+ 1} +  (1-\chi_{q+1}^2)\div \mathring{R}_{q} }_{\div\mathring{R}_{cut} +\nabla P_{cut}} .
\end{align}

Then, using the inverse divergence operator $\mathcal{R}$ in \eqref{operaru}
we choose the Reynolds stress at level $q+1$
\begin{align}\label{rucom}
	\mathring{R}_{q+1}  := \mathring{R}_{lin}  +   \mathring{R}_{osc} + \mathring{R}_{cor} +\mathring{R}_{com} +\mathring{R}_{cut} ,
\end{align}
where the linear error
\begin{align}\label{rbp}
\mathring{R}_{lin}:= &  \mathcal{R} \( \chi_{q+1}\partial_t (w^{(p)}_{q+ 1}+w^{(c)}_{q+ 1}) \)- \nu \mathcal{R} \Delta w_{q+1}  +  (u_{q}+\zq) \mathring\otimes w_{q+1} + w_{q + 1} \mathring\otimes (u_{q}+\zq),
\end{align}
the oscillation error
\begin{align} \label{rob}
	\mathring{R}_{osc}  &:=  \chi_{q+1}^2\sum_{k \in \Lambda}\mathcal{R}\P_{\neq 0}\left ( \P_{\neq 0}(W_{(k)}\otimes W_{(k)})\nabla (a_{(k)}^2)\right) -\mu^{-1} \chi_{q+1}^2\sum_{k \in \Lambda}\mathcal{R}\P_{\neq 0}(\p_t (a_{(k)}^2)\phi_{(k_1)}^2\phi_{(k)}^2k_1)\notag\\
	&\quad +\cq^2(\mathring{R}_{q}-\mathring{R}_{\ell_q}) ,
\end{align}
the corrector error
\begin{align}\label{rbp2}
	\mathring{R}_{cor}  := &\,  \wt w_{q+1}^{(p)} \mathring\otimes \wt w_{q+1}^{(c)+(t)} +\wt w_{q+1}^{(c)+(t)} \mathring\otimes w_{q+1},
\end{align}
the commutator error
\begin{align} \label{Rucom-def}
	\mathring{R} _{com}&:=  u_{q+1}\mathring\otimes (z_{q+1}-\zq)+(z_{q+1}-\zq)\mathring \otimes u_{q+1} +z_{q+1}\mathring\otimes z_{q+1}-\zq\mathring\otimes \zq,
\end{align}
and the cut-off error
\begin{align} \label{RBcut-def}
	\mathring{R}_{cut}
	& := \mathcal{R}   \(  \partial_t\chi_{q+1}(w^{(p)}_{q+ 1}+w^{(c)}_{q+ 1})+\partial_t(\chi_{q+1}^2)w^{(t)}_{q+ 1}\) +  (1-\chi_{q+1}^2)\mathring{R}_{q}.
\end{align}

\subsection{Verification of decay estimate} \label{sucsec-est-r-s}
We aim to estimate the above five components of Reynolds stress
and to verify the decay estimate \eqref{rl1s} at level $q+1$.

Since the Calder\'on-Zygmund operators are bounded in the spaces $L^\rho_x$, $1<\rho<\infty$, we estimate the Reynolds stress in $L^p_x$ instead of $L^1_x$ where
\begin{align}\label{def-p}
p: =\frac{16 }{16-7\varepsilon}\in (1,2),
\end{align}
with $\ve$ given by \eqref{b-beta-ve}. Note that, $(2+40\varepsilon)(1-\frac{1}{p})=\ve$,
and thus, by \eqref{larsrp} we have
\begin{align}  \label{rs-rp-p-ve}
	(\rp\rs^2)^{\frac 1p-1} = \lambda^{\varepsilon},
	\quad (\rp\rs^2)^{\frac 1p-\frac 12} = \lambda^{-\frac87+\varepsilon}.
\end{align}

Below we consider the estimates of Reynolds errors in two temporal regimes 
$[0,T]=[\frac{T_{q+2}+T_{q+1}}{2},T] \cup [0,\frac{T_{q+2}+T_{q+1}}{2}]$. 
Let us first treat the difficult regime away from the initial time.   

\subsubsection{The temporal regime 
away from the initial time $I_{1}:=[\frac{T_{q+2}+T_{q+1}}{2},T]$}

\medskip
\paragraph{\bf Linear error.}
Note that,
by Lemmas  \ref{buildingblockestlemma} and \ref{mae}, \eqref{larsrp}, \eqref{div free velocity} and \eqref{rs-rp-p-ve},
\begin{align}\label{Rlin-t}
	\| \cq \mathcal{R}\partial_t( w_{q+1}^{(p) }+w_{q+1}^{(c)})\|_{L_x^p}
	\lesssim&\  \sum_{k \in \Lambda }\| \mathcal{R} \curl\partial_t( a_{(k)} W^c_{(k)})\|_{L_x^p}  \nonumber \\
	\lesssim&\ \sum_{k \in \Lambda }\( \|  a_{(k)} \|_{C_{T,x}^1}\| W^c_{(k)} \|_{C_T L^{p} }
	+\|  a_{(k)} \|_{C_{T,x} }\| \p_t W^c_{(k)} \|_{C_T L^{p} }\) \notag\\
	\lesssim&\  \ell_q^{-13}(1+\j)^3\lambda_{q+1}^{-1}(\rp\rs^2)^{\frac 1p-\frac 12}
	+\ell_q^{-2}(1+\j)^\frac12\lambda_{q+1}^{-1}\sigma\rp^{-1}\mu(\rp\rs^2)^{\frac 1p-\frac 12}\notag\\
	\lesssim&\   \ell_q^{-2}(1+\j)^3\lambda_{q+1}^{-\frac17+\ve}.
\end{align}

Regarding the viscosity term, an application of Lemma \ref{totalest} gives
\begin{align}\label{viscosity}
	\norm{\nu\mathcal{R}(-\Delta) w_{q+1}}_{ L^p_x}
	& \lesssim\norm{ \mathcal{R}(-\Delta)\wt  w_{q+1}^{(p)}  }_{L^p_x} + \norm{ \mathcal{R}(-\Delta) \wt w_{q+1}^{(c)}  }_{ L^p_x}
     + \norm{ \mathcal{R}(-\Delta) \wt w_{q+1}^{(t) }}_{ L^p_x}\notag\\
	&\lesssim (1+\j)^{5}\Big( \ell_q^{-2}  \lambda_{q+1} (\rp \rs^2)^{\frac{1}{p}-\frac12} + \ell_q^{-2}  \lambda_{q+1} \rs \rp ^{-1}(\rp \rs^2)^{\frac{1}{p}-\frac12}+  \ell_q^{-4}\mu^{-1}  \lambda_{q+1} (\rp \rs^2)^{\frac{1}{p}-1}\Big)\notag\\
	&\lesssim (1+\j)^{5}(\ell_q^{-2}\lambda_{q+1}^{-\frac17+\ve}+\ell_q^{-2}\lambda_{q+1}^{-\frac37+\ve}+\ell_q^{-4}\lambda_{q+1}^{-\frac27+\ve})\notag\\
	&\lesssim\ell_q^{-2}(1+\j)^{5}\lambda_{q+1}^{-\frac17+\ve} .
\end{align}

Concerning the nonlinear terms in \eqref{rbp}, we have
\begin{align} \label{linear estimate1}
	&\norm{   (u_{q}+\zq) \mathring\otimes w_{q+1} + w_{q + 1} \mathring\otimes (u_{q}+\zq)}_{ L^p_x}  \nonumber \\	
	\lesssim\,& \norm{u_{q}+\zq}_{L^\infty_x} \norm{ w_{q+1}}_{ L^p_x}  \nonumber \\
	\lesssim\, &  (1+\j)^{4} (  \|u_q\|_{L^\infty_x} + \|z_{q}\|_{L^\infty_x}) (\ell_q^{-2}\lambda_{q+1}^{-\frac87+\ve}+\ell_q^{-2}\lambda_{q+1}^{-\frac{10}{7}+\ve}+\ell_q^{-4}\lambda_{q+1}^{-\frac97+\ve})\notag\\
	\lesssim\, &\ell_q^{-2} (1+\j)^{4}\lambda_{q+1}^{-\frac87+\ve} ( \|u_q\|_{L^\infty_x} + \|z_{q}\|_{L^\infty_x}).
\end{align}

Therefore, combining \eqref{Rlin-t}-\eqref{linear estimate1} together
and using  \eqref{b-beta-ve} we come to
\begin{align}   \label{linear estimate}
	\norm{\mathring{R}_{lin}}_{L^p_x}
	& \lesssim \ell_q^{-2}(1+\j)^{5}\lambda_{q+1}^{-\frac17+\ve}+\ell_q^{-2} (1+\j)^{4}\lambda_{q+1}^{-\frac87+\ve} ( \|u_q\|_{L^\infty_x} + \|z_{q}\|_{L^\infty_x}).
\end{align}

\paragraph{\bf Oscillation error.}
Now let us treat the delicate oscillation errors.
We decompose
\begin{align*}
	\mathring{R}_{osc}  = \mathring{R}_{osc.1}  +  \mathring{R}_{osc.2} +  \mathring{R}_{osc.3},
\end{align*}
where $\mathring{R}_{osc.1} $ contains the low-high spatial  oscillations
\begin{align*}
	\mathring{R}_{osc.1}
	&:=  \cq^2 \sum_{k \in \Lambda } \mathcal{R} \P_{\neq 0}\left(\P_{\neq 0}(W_{(k)}\otimes W_{(k)})\nabla (a_{(k)}^2)\right) ,
\end{align*}
$\mathring{R}_{osc.2} $ contains the high temporal oscillation
\begin{align*}
	\mathring{R}_{osc.2}
	&:= -\mu^{-1}\cq^2 \sum_{k \in \Lambda}\mathcal{R}   \P_{\neq 0}\(\p_t (a_{(k)}^2)\psi_{(k )}^2\phi_{(k)}^2  k_1\),
\end{align*}
and  $\mathring{R}_{osc.3} $ contains the mollification error
\begin{align*}
	\mathring{R}_{osc.3}
	&:= \cq^2 (\mathring{R}_{q}-\mathring{R}_{\ell_q}).
\end{align*}

For the  low-high spatial  oscillations  $\mathring{R}_{osc.1}  $,
we apply Lemma~\ref{commutator estimate1} in the Appendix
with $a = \nabla (a_{(k)}^2)$ and $f =  \psi_{(k_1)}^2\phi_{(k)}^2 $ and use Lemmas \ref{buildingblockestlemma} and \ref{mae} to get
\begin{align}  \label{I1-esti}
	\norm{\mathring{R}_{osc.1} }_{ L^p_x}
	& \lesssim  \sum_{ k \in \Lambda } \sigma^{-1} \norm{ \cq^2}_{C_{T}} \norm{ \na^3(a^2_{(k)})}_{C_{T,x}}
	\norm{\psi_{(k_1)}^2\phi_{(k)}^2  }_{C_TL^p_x }  \nonumber  \\
	& \lesssim  \ell_q^{-32}(1+\j)^7 \sigma^{-1} \norm{ \psi^2_{(k_1)}}_{C_TL^p_x } \norm{\phi^2_{(k)} }_{C_TL^p_x }  \nonumber  \\
	& \lesssim  \ell_q^{-32}(1+\j)^7\sigma^{-1}  (\rp\rs^2)^{\frac{1}{p}-1}\lesssim \ell_q^{-32}(1+\j)^7\lambda_{q+1}^{-\frac17+\ve}.
\end{align}
Concerning $\mathring{R}_{osc.2}  $, we use the large parameter $\mu$ to balance the
the high temporal oscillations arising from $\psi_{(k_1)}$, that is,
\begin{align}  \label{I2-esti}
	\norm{\mathring{R}_{osc.2}}_{ L_x^p}
	&\lesssim \mu^{-1} \norm{ \cq^2}_{C_{T}} 	\norm{\p_t (a_{(k)}^2) }_{C_{T,x}}
	\norm{\psi_{(k_1)}}_{C_TL^{2p}_x }^2\norm{\phi_{(k)}}_{L^{2p}_x }^2\nonumber \\
	&\lesssim  \ell_q^{-15} (1+\j)^4\mu^{-1} (\rp\rs^2)^{\frac{1}{p}-1}\notag \\
 &\lesssim \ell_q^{-15}(1+\j)^4\lambda_{q+1}^{-\frac97+\ve}.
\end{align}
Regarding the mollification error $\mathring{R}_{osc.3}$, by \eqref{rl1b-s} we get
\begin{align}  \label{I3-esti}
	\norm{\mathring{R}_{osc.3}  }_{ L_x^1}
	&\lesssim \norm{ \cq^2}_{C_{T}}  \norm{\mathring{R}_{q}-\mathring{R}_{\ell_q}}_{L_x^1}\lesssim  \ell_q^{\frac12-\delta} ( \|\mathring{R}_{q} \|_{  C^{\frac12-\delta}_{[T_{q+2},T]}L^1_x }+	\| \mathring{R}_{q}  \|_{C_{[T_{q+2},T]}W^{1-\delta,1}_x} ).
\end{align}
Therefore, putting estimates \eqref{I1-esti}-\eqref{I3-esti} altogether, we come to
\begin{align}
	\label{oscillation estimate}
	\norm{\mathring{R}_{osc} }_{L^1_x}
	&\lesssim \ell_q^{-32}(1+\j)^7\lambda_{q+1}^{-\frac17+\ve} +\ell_q^{\frac12-\delta} (\|\mathring{R}_{q} \|_{  C^{\frac12-\delta}_{[T_{q+2},T]}L^1_x }+
	\| \mathring{R}_{q}  \|_{C_{[T_{q+2},T]}W^{1-\delta,1}_x} ).
\end{align}

\paragraph{\bf Corrector error.}
The corrector error can be bounded directly by the
H\"older inequality, Lemma \ref{totalest}, \eqref{Lp-wdp-1} and \eqref{e3.41}:
\begin{align}\label{corrector estimate}
	\norm{\mathring{R}_{cor} }_{ L^{1}_x}
	\lesssim& \norm{  w_{q+1}^{(c)+(t)}}_{C_T L^{2}_x} (\norm{  w^{(p)}_{q+1}  }_{C_TL^{2}_{x}} + \norm{w_{q+1} }_{C_TL^{2}_{x}}) \notag\\
	\lesssim& \ell_q^{-8} (1+\j)^{8} \( r_{\perp} r_{\parallel}^{-1}
	+  \mu^{-1} (\rp\rs^2)^{-\frac12} \)\(1+r_{\perp} r_{\parallel}^{-1}
	+  \mu^{-1} (\rp\rs^2)^{-\frac12}  \)  \notag\\
	\lesssim& \ell_q^{-8}(1+\j)^{8} \lambda_{q+1}^{-\frac17} .
\end{align}

\paragraph{\bf Commutator error.}
By \eqref{def-ziq}  we derive
\begin{align}\label{est-zq1-zq}
	 \| z_{q+1}-\zq \|_{L^{2}_x} \lesssim \laq^{-\frac{\ve}{8}(1-\delta)} \|\zq\|_{ H^{1-\delta}_x}.
\end{align}

Thus, using \eqref{est-zq1-zq} we have
\begin{align}\label{R-com-esti}
	\|\mathring{R} _{com}\|_{ L^1_x}	&\leq\|u_{q+1}\mathring\otimes (z_{q+1}-\zq)+(z_{q+1}-\zq) \otimes u_{q+1} +z_{q+1}\mathring\otimes z_{q+1}-\zq\mathring\otimes \zq\|_{L^1_x}\notag\\
	&\lesssim  (\|u_{q+1}\|_{ L^2_x}+ \|z_{q+1}\|_{ L^2_x}+ \| \zq\|_{ L^2_x})\|z_{q+1}-\zq\|_{ L^2_x}\notag\\
	&\lesssim \laq^{-\frac{\ve}{8}(1-\delta)} \|\zq\|_{H^{1-\delta}_x}  (\|u_q\|_{ L^2_x}+\|w_{q+1}\|_{ L^2_x}+ \|z\|_{ L^2_x} ).
\end{align}

Now, combining  \eqref{linear estimate},
\eqref{oscillation estimate},
\eqref{corrector estimate}
and \eqref{R-com-esti} altogether,  we arrive at
\begin{align} \label{rq1-1}
	&\quad\| \mathring{R}_{lin}  \|_{L^r_{\Omega}C_{I_1} L^p_x} +  \| \mathring{R}_{osc} \|_{L^r_{\Omega}C_{I_1}L^p_x}
	+  \|\mathring{R}_{cor}  \|_{L^r_{\Omega}C_{I_1}L^{ 1}_x} +\|\mathring{R}_{com} \|_{L^r_{\Omega}C_{I_1}L^1_x} \nonumber  \\
&\lesssim  \ell_q^{-2}\lambda_{q+1}^{-\frac17+\ve}(1+ \|\rr_q\|_{L^{5r}_{\Omega}C_TL^1_x}^5)+ \ell_q^{-4}\lambda_{q+1}^{-\frac87+\ve}(1+ \|\rr_q\|_{L^{8r}_{\Omega}C_TL^1_x}^4) ( \|u_q\|_{L^{2r}_{\Omega}C_{I_1}L^\infty_x} + \|z_{q}\|_{L^{2r}_{\Omega}C_{I_1}L^\infty_x})\notag\\
&\quad +\ell_q^{-32}\lambda_{q+1}^{-\frac17+\ve}(1+ \|\rr_q\|_{L^{7r}_{\Omega}C_TL^1_x}^7)+\ell_q^{\frac12-\delta} (\|\mathring{R}_{q} \|_{ L^{r}_{\Omega} C^{\frac12-\delta}_{[T_{q+2},T]}L^1_x }+
	\| \mathring{R}_{q}  \|_{L^{r}_{\Omega}C_{[T_{q+2},T]}W^{1-\delta,1}_x} )  \notag\\
 &\quad + \laq^{-\frac{\ve}{8}(1-\delta)} \|\zq\|_{L^{2r}_{\Omega}C_{I_1}H^{1-\delta}_x}  (\|u_q\|_{ L^{2r}_{\Omega}C_{I_1}L^2_x}+\|w_{q+1}\|_{L^{2r}_{\Omega}C_{I_1}L^2_x}+ \|z\|_{L^{2r}_{\Omega}C_{I_1} L^2_x} ) \notag\\
	&\leq  \frac14 c_*\delta_{q+3},
\end{align}
for $a$ large enough, where we used $\delta<{1}/{6}$ and $( 12^{q+1}\cdot 8 rL^2)^{ 12^{q+1} } \leq \lambda_{q}$ (when $a$ is large enough) in the last step.

\paragraph{\bf Cut-off error.}
In the subcase where $t\in K_1:=[T_{q+1},T]$, we have $\mathring{R}_{cut}(t)=0$.
Moreover, in the other subcase $t\in K_2:= [\frac{T_{q+2}+T_{q+1}}{2},T_{q+1}]$, using \eqref{rl1b}, \eqref{asp-0} and \eqref{asp-q}, for $a$ large enough we derive
\begin{align}  \label{R-cut-esti-1}
	\|\mathring{R}_{cut} \|_{L^r_{\Omega}C_{K_2}L^1_x}
	&\lesssim \|\mathcal{R} \(  \partial_t\chi_{q+1}(w^{(p)}_{q+ 1}+w^{(c)}_{q+ 1})+\partial_t(\chi_{q+1}^2)(w^{(t)}_{q+ 1})\)+ (1-\chi_{q+1}^2 )\mathring{R}_{\ell_q} \|_{L^r_{\Omega}C_{K_2}L^1_x} \notag \\
	&\lesssim \ell_q^{-5}(1+\| \mathring R_{q}\|_{L^{4r}_{\Omega}C_TL^{1}_{x}}^4) (\lambda_{q+1}^{-\frac87+\ve}+\lambda_{q+1}^{-\frac{10}{7}+\ve}+\lambda_{q+1}^{-\frac{9}{7}+\ve} )+\| \mathring R_{q}\|_{L^r_{\Omega}C_{[T_{q+2},T_{q+1}]}L^{1}_{x}}\notag\\
	&\lesssim \ell_q^{-5}( 12^q\cdot 8\cdot 4rL^2)^{4\cdot 12^q }\lambda_{q+1}^{-\frac87+\ve} + \| \mathring R_{q}\|_{L^r_{\Omega}C_{[T_{q+2},T_{q+1}]}L^{1}_{x}}\leq \frac14 c_*\delta_{q+ 3}.
\end{align}

Now,  combining  \eqref{rq1-1}-\eqref{R-cut-esti-1} together,  we
conclude that
\begin{align} \label{rq1b}
\|\mathring{R}_{q+1} \|_{L^r_{\Omega}C_{I_1}L^1_{x}}
	&\leq \| \mathring{R}_{lin}  \|_{L^r_{\Omega}C_{I_1}L^p_x} +  \| \mathring{R}_{osc}  \|_{L^r_{\Omega}C_{I_1}L^p_x}
	+  \|\mathring{R}_{cor}  \|_{L^r_{\Omega}C_{I_1}L^{ 1}_x}  +  \|\mathring{R}_{com}   \|_{L^r_{\Omega}C_{I_1}L^1_x} +  \|\mathring{R}_{cut} (t) \|_{L^r_{\Omega}C_{I_1}L^1_x}\nonumber  \\
	&\leq \frac12c_*\delta_{q+3}.
\end{align}

\subsubsection{The temporal regime 
near the initial time.}

Concerning the case where $t\in I_2:=[0,\frac{T_{q+2}+T_{q+1}}{2}]$,
we have $\mathring{R}_{lin}=\mathring{R}_{osc}  =\mathring{R}_{cor}=0$ and $\mathring{R}_{cut}=\mathring{R}_{q}$.
Thus, using \eqref{asp-0}, \eqref{asp-q}, \eqref{R-com-esti} and choosing $a$ sufficiently large,
we get
\begin{align}\label{rq1u-3}
	\|\mathring{R}_{q+1} \|_{L^r_{\Omega}C_{I_2} L^1_{x}}&\leq \|\mathring{R}_{com}   \|_{L^r_{\Omega}C_{I_2} L^1_x} + \|\mathring{R}_{cut} \|_{L^r_{\Omega} C_{I_2}L^1_x} \leq \frac12 c_*\delta_{q+ 3},
\end{align}
thereby verifying \eqref{rl1s}.

Following a similar manner as in the proof of \eqref{rq1-1}-\eqref{R-cut-esti-1}, for $m\geq 1$ we have
\begin{align} \label{rq1b-m}
	\|\mathring{R}_{q+1}   \|_{L^{m}_{\Omega}C_TL^1_{x}}\lesssim (12^{q+1}\cdot 8mL^2)^{12^{q+1}}.
\end{align}

\subsection{Verification of growth estimate}  \label{sucsec-est-r-b}

Below we verify the $L^m_\Omega C^{\frac12-\delta}_{[T_{q+3},T]}L^1_x$
and $L^m_\Omega C_{[T_{q+3},T]}W_x^{1-\delta,1}$ growth estimates in \eqref{rl1b-s} for the Reynolds stress.
Let us consider two cases
where $t\in I_1':=[T_{q+3},T_{q+2}]$
and $t\in I_2':=[T_{q+2},T]$, respectively,
in the following.

\subsubsection{The temporal regime $I_1':=[T_{q+3},T_{q+2}]$}

For the case where $t\in I_1':=[T_{q+3},T_{q+2}]$, we have $u_{q+1}(t)= u_{q}(t)$, $\mathring{R}_{lin}=\mathring{R}_{osc}  =\mathring{R}_{cor}=0$ and $\mathring{R}_{cut}=\mathring{R}_{q}=\mathring{R}_{0}$. Hence,
\begin{align}
    \rr_{q+1}= u_{q}\mathring\otimes (z_{q+1}-\zq)+(z_{q+1}-\zq) \otimes u_{q} +z_{q+1}\mathring\otimes z_{q+1}-\zq\mathring\otimes \zq+ \mathring{R}_{0}.
\end{align}
By \eqref{uc}, interpolation and Proposition~\ref{Prop-noise}, we get
\begin{align*}
	&\quad\| u_{q}\mathring\otimes (z_{q+1}-\zq)+(z_{q+1}-\zq) \mathring\otimes u_{q} +z_{q+1}\mathring\otimes z_{q+1}-\zq\mathring\otimes \zq \|_{L^{m}_{\Omega}C^{\frac12-\delta}_{I_1'}L_x^1} \notag\\
	&\lesssim   \|u_{q} \|_{L^{2m}_{\Omega}C^{\frac12-\delta}_{I_1'}L_x^2} \|z \|_{L^{2m}_{\Omega}C^{\frac12-\delta}_TL_x^2} + \|z \|_{L^{2m}_{\Omega}C^{\frac12-\delta}_TL_x^2}^2\notag\\
	&\lesssim \lambda_{q}^{4}(12^{q-1}\cdot 8\cdot 12mL^2)^{6\cdot 12^{q-1}}(2m-1)^{\frac12}L +(2m-1)L^2,
\end{align*}
which along with \eqref{est-r0-c-end} yields
\begin{align}\label{est-rq1-c-p1}
    \norm{\mathring{R}_{q+1} }_{L^m_{\Omega}C^{\frac12-\delta}_{I_1'}L_x^1}\lesssim \lambda_{q+1}^5 (12^{q+1}\cdot 8mL^2)^{12^{q+1}}.
\end{align}
Similarly, we have
\begin{align*}
	&\quad\| u_{q}\mathring\otimes (z_{q+1}-\zq)+(z_{q+1}-\zq) \mathring\otimes u_{q} +z_{q+1}\mathring\otimes z_{q+1}-\zq\mathring\otimes \zq \|_{L^m_{\Omega}C_{I_1'}W_x^{1-\delta,1}} \notag\\
	&\lesssim   \|u_{q} \|_{L^{2m}_{\Omega}C_{I_1'}H_x^{1-\delta}} \|z \|_{L^{2m}_{\Omega}C_TH_x^{1-\delta}} + \|z \|_{L^{2m}_{\Omega}C_TH_x^{1-\delta}}^2\notag\\
	&\lesssim \lambda_{q}^{4}(12^{q-1}\cdot 8\cdot 12mL^2)^{6\cdot 12^{q-1}}(2m-1)^{\frac12}L +(2m-1)L^2,
\end{align*}
which along with \eqref{est-r0-w-1} yields
\begin{align}\label{est-rq1-w-p1}
    \norm{\mathring{R}_{q+1} }_{L^m_{\Omega}C_{I_1'}W_x^{1-\delta,1}}\lesssim \lambda_{q+1}^5 (12^{q+1}\cdot 8mL^2)^{12^{q+1}}.
\end{align}

\subsubsection{The temporal regime
$t\in I_2':=[T_{q+2},T]$}
Now let us treat the more delicate case where $t\in [T_{q+2},T]$.
\paragraph{\bf Linear error.}
Concerning the linear error $\mathring{R}_{lin}$, following a similar manner as in \eqref{Rlin-t}, we have
\begin{align}\label{Rlin-t-c}
	\| \cq \mathcal{R}\partial_t( w_{q+1}^{(p) }+w_{q+1}^{(c)})\|_{C^{\frac12-\delta}_{I_2'}L_x^1}
	\lesssim&\  \sum_{k \in \Lambda }\| \cq\mathcal{R} \curl\partial_t( a_{(k)} W^c_{(k)}) \|_{C^{\frac12-\delta}_{I_2'}L_x^p}  \nonumber \\
	\lesssim&\ \sum_{k \in \Lambda }\|\cq\partial_t( a_{(k)}W^c_{(k)}) \|_{C_{I_2'}L_x^p}^{\frac12+\delta}\| \cq\partial_t( a_{(k)} W^c_{(k)}) \|_{C^1_{I_2'}L_x^p}^{\frac12-\delta} \notag\\
    \lesssim&\  (\ell_q^{-2}(1+\j)^3\lambda_{q+1}^{-\frac27+\ve})^{\frac12+\delta}(\ell_q^{-3}(1+\j)^4\lambda_{q+1}^{3})^{\frac12-\delta} \notag\\
	\lesssim&\   \ell_q^{-3}(1+\j)^4\lambda_{q+1}^{\frac32},
\end{align}
and
\begin{align}\label{Rlin-t-h}
	\| \cq \mathcal{R}\partial_t( w_{q+1}^{(p) }+w_{q+1}^{(c)})\|_{C_{I_2'}W_x^{1-\delta,1}} \lesssim 	\|  \mathcal{R}\partial_t( w_{q+1}^{(p) }+w_{q+1}^{(c)})\|_{C_{I_2'}W_x^{1 ,p}}
	\lesssim&\   \ell_q^{-2}(1+\j)^5\lambda_{q+1},
\end{align}
where the integrability exponent $p$ is given by \eqref{def-p}.

Regarding the viscosity term, an application of interpolation and Lemma \ref{buildingblockestlemma} gives
\begin{align}\label{viscosity-c}
	\norm{\nu\mathcal{R}\Delta w_{q+1}}_{C^{\frac12-\delta}_{I_2'}L_x^1}
& \lesssim	\norm{\nu\mathcal{R}\Delta w_{q+1}}_{C_{I_2'}L_x^p}^{\frac12+\delta} 	
\norm{\nu\mathcal{R}\Delta w_{q+1}}_{C^1_{I_2'}L_x^p}^{\frac12-\delta} \notag\\
 &\lesssim (\ell_q^{-2}(1+\j)^{5}\lambda_{q+1}^{-\frac17+\ve})^{\frac12+\delta}  (\ell_q^{-2}(1+\j)^{6}\lambda_{q+1}^{3})^{\frac12-\delta} \notag\\
 &\lesssim \ell_q^{-2}(1+\j)^{6}\lambda_{q+1}^{\frac32},
\end{align}
and
\begin{align}\label{viscosity-h}
	\norm{\nu\mathcal{R} \Delta w_{q+1}}_{C_{I_2'}W_x^{1-\delta,1}}\leq \norm{\nu\mathcal{R}\Delta w_{q+1}}_{C_{I_2'}W_x^{1,p}}  &\lesssim \ell_q^{-2}(1+\j)^{7}\lambda_{q+1}.
\end{align}

For the remaining nonlinear terms in \eqref{rbp}, we have
\begin{align} \label{linear estimate1-c}
	&\norm{  (u_{q}+\zq) \mathring\otimes  w_{q+1} +  w_{q + 1} \mathring\otimes (u_{q}+\zq)}_{C^{\frac12-\delta}_{I_2'}L_x^1}   \nonumber \\	
	\lesssim\,& \norm{u_{q}+\zq}_{C^{\frac12-\delta}_{I_2'}L_x^2}   \norm{  w_{q+1}}_{C^{1}_{T,x}}   \nonumber \\
	\lesssim\, & \lambda_{q+1}^4 (1+\j)^{6} (  \|u_q\|_{C^{\frac12-\delta}_{I_2'}L_x^2}  + \|z_{q}\|_{C^{\frac12-\delta}_TL_x^2} ),
\end{align}
and
\begin{align} \label{linear estimate1-h}
	&\norm{  (u_{q}+\zq) \otimes  w_{q+1} + w_{q + 1} \otimes (u_{q}+\zq)}_{C_{I_2'}W_x^{1-\delta,1}}   \nonumber \\	
	\lesssim\,& \norm{u_{q}+\zq}_{C_{I_2'}H_x^{1-\delta}}   \norm{ w_{q+1}}_{C^{1}_{T,x}}   \nonumber \\
	\lesssim\, & \lambda_{q+1}^4 (1+\j)^{6} (  \|u_q\|_{C_{I_2'}H_x^{1-\delta}}   + \|z_{q}\|_{C_TH_x^{1-\delta}}  ).
\end{align}

Thus, combining \eqref{Rlin-t-c}-\eqref{linear estimate1-h} together
 we arrive at
\begin{align}   \label{linear estimate-c}
	\norm{\mathring{R}_{lin} }_{C^{\frac12-\delta}_{I_2'}L_x^1}
     & \lesssim \ell_q^{-3}(1+\j)^{6}\lambda_{q+1}^{\frac32}+\lambda_{q+1}^4 (1+\j)^{6} (\|u_q\|_{C^{\frac12-\delta}_{I_2'}L_x^2} + \|z_{q}\|_{C^{\frac12-\delta}_TL_x^2} ),
\end{align}
and
\begin{align}   \label{linear estimate-h}
	\norm{\mathring{R}_{lin} }_{C_{I_2'}W_x^{1-\delta,1}}
	& \lesssim \ell_q^{-2}(1+\j)^{7}\lambda_{q+1}+\lambda_{q+1}^4 (1+\j)^{6} ( \|u_q\|_{C_{I_2'}H_x^{1-\delta}}   + \|z_{q}\|_{C_TH_x^{1-\delta}}).
\end{align}

\paragraph{\bf Oscillation error.}
In order to treat the oscillation error, applying Lemmas~\ref{buildingblockestlemma}, \ref{mae} and interpolation we get
\begin{align}  \label{ro-esti-c}
	\norm{\mathring{R}_{osc}  }_{C^{\frac12-\delta}_{I_2'}L_x^1}
	& \lesssim 	\|\cq^2\|_{C_T^1}	\norm{\mathring{R}_{q}  }_{C^{\frac12-\delta}_{I_2'}L_x^1}+ \sum_{ k \in \Lambda } \|\cq^2\|_{C_T^1} \norm{ \na (a^2_{(k)})}_{C^1_{T,x}}\norm{W_{(k)} }_{C^{\frac12-\delta}_TL^{2} }^2\notag\\
	&\quad  +  {\mu}^{-1} \|\cq^2\|_{C_T^1} \sum_{k\in\Lambda }
	\norm{\p_t (a_{(k)}^2) }_{C^1_{T,x}}
	\norm{\psi_{(k_1)}}_{C^{\frac12-\delta}_TL^{2} }^2\norm{\phi_{(k)}}_{L^{2} }^2 \nonumber  \\
	& \lesssim  	\ell_q^{-1} \norm{\mathring{R}_{q}  }_{C^{\frac12-\delta}_{I_2'}L_x^1}+ \ell_q^{-27}(1+\j)^8\lambda_{q+1}^{3},
\end{align}
and
\begin{align}  \label{ro-esti-h}
	\norm{\mathring{R}_{osc}  }_{C_{I_2'}W_x^{1-\delta,1}}
	& \lesssim 	\|\cq^2\|_{C_T}\norm{\mathring{R}_{q}  }_{C_{I_2'}W_x^{1-\delta,1}} +\|\cq^2\|_{C_T} \sum_{ k \in \Lambda } \norm{ \na (a^2_{(k)})}_{C^1_{T,x}}\norm{W_{(k)} }_{C_TH_x^{1}} ^2\notag\\
	&\quad  +  {\mu}^{-1}  \|\cq^2\|_{C_T}\sum_{k\in\Lambda }
	\norm{\p_t (a_{(k)}^2) }_{C^1_{T,x}}
	\norm{\psi_{(k_1)}^2\phi_{(k)}^2}_{C_TW_x^{1,1}} \nonumber  \\
	& \lesssim  	\norm{\mathring{R}_{q}  }_{C_{I_2'}W_x^{1-\delta,1}}+ \ell_q^{-26}(1+\j)^8\lambda_{q+1}^{3}.
\end{align}

\paragraph{\bf Corrector error.}
Using the interpolation inequality and
applying Lemma \ref{totalest} to \eqref{rbp2} we get
\begin{align}\label{corrector estimate-c}
	\norm{\mathring{R}_{cor}   }_{C^{\frac12-\delta}_{I_2'}L_x^1}
	\lesssim \norm{\wt  w_{q+1}^{(c)+(t)}}_{C^{\frac12-\delta}_TL_x^2}  (\norm{\wt  w^{(p)}_{q+1}  }_{C^{\frac12-\delta}_TL_x^2}  + \norm{  w_{q+1} }_{C^{\frac12-\delta}_TL_x^2} )
	 \lesssim \lambda_{q+1}^{4} (1+\j)^{6} ,
\end{align}
and
\begin{align}\label{corrector estimate-h}
	\norm{\mathring{R}_{cor}   }_{C_{I_2'}W_x^{1-\delta,1}}
	\lesssim \norm{ \wt  w_{q+1}^{(c)+(t)}}_{C_TH_x^{1}}  (\norm{\wt  w^{(p)}_{q+1}  }_{C_TH_x^{1}}  + \norm{  w_{q+1} }_{C_TH_x^{1}} )
	\lesssim& \lambda_{q+1}^{4}(1+\j)^{6}.
\end{align}
\paragraph{\bf Commutator error.}
By interpolation and Lemma~\ref{totalest},
\begin{align}\label{R-com-c}
	\|\mathring{R} _{com}\|_{C^{\frac12-\delta}_{I_2'}L_x^1} 	&\leq\| u_{q+1}\mathring\otimes (z_{q+1}-\zq)+(z_{q+1}-\zq) \mathring\otimes u_{q+1} +z_{q+1}\mathring\otimes z_{q+1}-\zq\mathring\otimes \zq \|_{C^{\frac12-\delta}_{I_2'}L_x^1} \notag\\
	&\lesssim   \|u_{q+1} \|_{C^{\frac12-\delta}_{I_2'}L_x^2} \|z \|_{C^{\frac12-\delta}_TL_x^2} + \|z \|_{C^{\frac12-\delta}_TL_x^2}^2\notag\\
	&\lesssim  (\lambda_{q+1}^{2}(1+\j)^5+\|u_{q} \|_{C^{\frac12-\delta}_{I_2'}L_x^2} )\|z \|_{C^{\frac12-\delta}_TL_x^2} + \|z \|_{C^{\frac12-\delta}_TL_x^2}^2,
\end{align}
and
\begin{align}\label{R-com-h}
	\|\mathring{R} _{com}\|_{C_{I_2'}W_x^{1-\delta,1}} &\lesssim  \|u_{q+1} \|_{C_{I_2'}H_x^{1}} \|z \|_{C_TH_x^{1-\delta}} + \|z \|_{C_TH_x^{1-\delta}}^2\notag\\
 &\lesssim (\lambda_{q+1}^{2}(1+\j)^5+\|u_{q} \|_{C_{I_2'}H_x^{1}} ) \|z \|_{C_TH_x^{1-\delta}} + \|z \|_{C_TH_x^{1-\delta}}^2.
\end{align}

\paragraph{\bf Cut-off error.}
In order to treat the cut-off error, using Lemma~\ref{totalest}, \eqref{est-r0-c-end}, \eqref{est-r0-w-1} and the interpolation inequality we get
\begin{align}\label{cut-est-c}
    \|\mathring{R} _{cut}\|_{C^{\frac12-\delta}_{I_2'}L_x^1} 	&\leq \| \mathcal{R} \(  \partial_t\chi_{q+1}(w^{(p)}_{q+ 1}+w^{(c)}_{q+ 1})+\partial_t(\chi_{q+1}^2)(w^{(t)}_{q+ 1})\)+ (1-\chi_{q+1}^2 )\mathring{R}_{q}\|_{C^{\frac12-\delta}_{I_2'}L_x^1}\notag\\
    &\lesssim \|\partial_t\chi_{q+1}\|_{C^{\frac12-\delta}_{I_2'}} \|w^{(p)}_{q+ 1}+w^{(c)}_{q+ 1}\|_{C^{\frac12-\delta}_{I_2'}L_x^p} + \|\partial_t(\chi_{q+1}^2)\|_{C^{\frac12-\delta}_{I_2'}} \|w^{(t)}_{q+ 1} \|_{C^{\frac12-\delta}_{I_2'}L_x^p} \notag\\
    &\quad +\|(1-\chi_{q+1}^2 )\|_{C^{\frac12-\delta}_{I_2'}} \|\mathring{R}_{q}\|_{C^{\frac12-\delta}_{I_2'}L_x^1}\notag\\
    &\lesssim \lambda_{q+1}^{4}(1+\j)^{6}+ \ell_q^{-1} \|\mathring{R}_{q}\|_{C^{\frac12-\delta}_{I_2'}L_x^1}
\end{align}
and 
\begin{align}\label{cut-est-h}
    \|\mathring{R} _{cut}\|_{C_{I_2'}W_x^{1-\delta,1}} &\lesssim   \|\partial_t\chi_{q+1}\|_{C_{I_2'}} \|w^{(p)}_{q+ 1}+w^{(c)}_{q+ 1}\|_{C_{I_2'}W_x^{1-\delta,p}} + \|\partial_t(\chi_{q+1}^2)\|_{C_{I_2'}} \|w^{(t)}_{q+ 1} \|_{C_{I_2'}W_x^{1-\delta,p}} \notag\\
    &\quad +\|(1-\chi_{q+1}^2 )\|_{C_{I_2'}} \|\mathring{R}_{q}\|_{C_{I_2'}W_x^{1-\delta,1}}\notag\\
     &\lesssim \lambda_{q+1}^{4}(1+\j)^{6}+  \|\mathring{R}_{q}\|_{C_{I_2'}W_x^{1-\delta,1}}.
\end{align}

Therefore,  combining  \eqref{linear estimate-c},
\eqref{ro-esti-c},
\eqref{corrector estimate-c}, \eqref{R-com-c}
and \eqref{cut-est-c} altogether, for $m\geq 1$ we obtain
\begin{align} \label{rq1b-c}
	\|\mathring{R}_{q+1}   \|_{L^{m}_{\Omega}C^{\frac12-\delta}_{I_2'}L_x^1}
	&\leq \| \mathring{R}_{lin}   \|_{L^{m}_{\Omega}C^{\frac12-\delta}_{I_2'}L_x^1} +  \| \mathring{R}_{osc}  \|_{ L^{m}_{\Omega}C^{\frac12-\delta}_{I_2'}L_x^1}
	+  \|\mathring{R}_{cor}   \|_{L^{m}_{\Omega}C^{\frac12-\delta}_{I_2'}L_x^1}  +  \|\mathring{R}_{com}\|_{L^{m}_{\Omega}C^{\frac12-\delta}_{I_2'}L_x^1} \nonumber  \\
 &\quad + \|\mathring{R}_{cut}   \|_{L^{m}_{\Omega}C^{\frac12-\delta}_{I_2'}L_x^1} \notag\\
	&\lesssim  \lambda_{q+1}^{4}( 12^q\cdot 8\cdot 12mL^2)^{6\cdot 12^q } (\lambda_{q}^4(12^{q-1}\cdot8\cdot12mL^2)^{6\cdot12^{q-1}} +(2m-1)^\frac12 L )\notag\\
	&\quad+\ell_q^{-3}\lambda_{q+1}^{\frac32}( 12^q\cdot 8\cdot6mL^2)^{6\cdot12^q }+\ell_q^{-1}\lambda_{q}^5 (12^q\cdot 8mL^2)^{12^q}\notag\\
    &\quad+ \ell_q^{-27}\lambda_{q+1}^{3}( 12^q\cdot 8\cdot 8mL^2)^{8\cdot12^q } + \lambda_{q+1}^{4}( 12^q\cdot 8\cdot6mL^2)^{6\cdot12^q }\notag\\
	 &\quad + (\lambda_{q+1}^{2}( 12^q\cdot 8\cdot10mL^2)^{5\cdot12^q }+ \lambda_{q}^4(12^{q-1}\cdot8\cdot12mL^2)^{6\cdot12^{q-1}} )(2m-1)^{\frac12} L  +( 2m-1)  L^2\notag\\
  &\quad + \lambda_{q+1}^{4}( 12^q\cdot 8\cdot6mL^2)^{6\cdot12^q } +\ell_q^{-1} \lambda_{q}^5 (12^q\cdot 8mL^2)^{12^q}\notag\\
	&\lesssim \lambda_{q+1} ^5 (12^{q+1}\cdot 8mL^2)^{12^{q+1}},
\end{align}
and, similarly,
\begin{align} \label{rq1b-h}
	\|\mathring{R}_{q+1}   \|_{L^m_{\Omega}C_{I_2'}W_x^{1-\delta,1}}
	\lesssim \lambda_{q+1} ^5 (12^{q+1}\cdot 8mL^2)^{12^{q+1}}.
\end{align}

Therefore, estimates \eqref{est-rq1-c-p1}, \eqref{est-rq1-w-p1}, \eqref{rq1b-c} and \eqref{rq1b-h} together verify the inductive estimates in \eqref{rl1b-s} at level $q+1$.

\section{Proof of main results} \label{Sec-prf-thm}
This section is devoted to the proof of the main iteration estimates in Proposition~\ref{Prop-Iterat} and the main results
concerning the global-in-time continuous energy solutions
with prescribed initial data in Theorem \ref{Thm-Non-SNS} and Corollary \ref{thm-Nonuniq-Stoch}.

\subsection{Main iteration}

Let us first prove the main iteration estimates in Proposition~\ref{Prop-Iterat} below. 

\medskip
{\bf Proof of Proposition~\ref{Prop-Iterat}.}  
Based on estimates in the previous two sections, 
we are left to verify the inductive estimates \eqref{ul2}-\eqref{energy-est} 
for $u_0$ and $\rr_0$ at the initial step $q=0$. 

Note that, 
by our choice, 
$u_0$ is exactly the solution $\wt u$ to the 
$\Lambda$-NSE \eqref{equa-nse-2} with $q=0$. 
Then, 
by the energy balance \eqref{eq-e-2}, 
$$\|u_0\|_{L^\infty_\Omega C_T L^2_x} \leq 
\|v_0\|_{L^\infty_\Omega L^2_x} \leq M,$$ 
which satisfies \eqref{ul2} 
by taking $\delta_0\geq M$. 
Estimate \eqref{uc} has been verified by \eqref{est-u0-c1}. Moreover, 
by the choice of $\delta_2$ in \eqref{R0-CtL1}, 
\eqref{rl1s} is satisfied. 

We also derive from 
\eqref{r0b} and Proposition~\ref{Prop-noise} that 
\begin{align}\label{est-r0-m}
\| \mathring{R}_{0} \|_{L^{m}_{\Omega}C_TL^1_x}
	&\lesssim \|\wt u \|_{L^{2m}_{\Omega}C_TL^{2}_x}^2+\|\wt u \|_{L^{2m}_{\Omega}C_TL^{2}_x} \| z \|_{L^{2m}_{\Omega}C_TL^2_x}+  \| z \|_{L^{2m}_{\Omega} C_TL^2_x}^2 \lesssim 8mL^2,
\end{align} 
which verifies \eqref{rl1b}. 

Concerning \eqref{rl1b-s}, 
for $\delta\in(0,1/6)$, 
we estimate 
\begin{align}\label{est-r0-c-1}
\| \mathring{R}_{0} \|_{L^{m}_{\Omega}C^{\frac12-\delta}_{[T_{q+2},T]}L^1_x} &\lesssim  \|\P_{\geq \Lambda}(\P_{< \Lambda}\wt u\mathring \otimes \P_{< \Lambda}\wt u) \|_{L^{m}_{\Omega}C^{\frac12-\delta}_{[T_{q+2},T]}L^1_x} +\|(\P_{< \Lambda}\wt u\mathring \otimes \P_{\geq \Lambda}\wt u ) \|_{L^{m}_{\Omega}C^{\frac12-\delta}_{[T_{q+2},T]}L^1_x} \notag\\
	&\quad + \|(\P_{\geq\Lambda}\wt u\mathring \otimes \wt u) \|_{L^{m}_{\Omega}C^{\frac12-\delta}_{[T_{q+2},T]}L^1_x} +  \|\wt u \|_{L^{2m}_{\Omega}C^{\frac12-\delta}_{[T_{q+2},T]}L^2_x} \| z_0 \|_{L^{2m}_{\Omega} C^{\frac12-\delta}_{[T_{q+2},T]}L^2_x}\notag\\
 &\quad+ \| z_0\|_{L^{2m}_{\Omega}C^{\frac12-\delta}_{[T_{q+2},T]}L^{2}_x}^{2}  \notag \\
	&\lesssim \|\wt u \|_{L^{2m}_{\Omega}C^{\frac12-\delta}_{[T_{q+2},T]}L^{2}_x}^2+\|\wt u \|_{L^{2m}_{\Omega}C^{\frac12-\delta}_{[T_{q+2},T]}L^{2}_x} \| z \|_{L^{2m}_{\Omega} C^{\frac12-\delta}_{[T_{q+2},T]}L^2_x}+  \| z \|_{L^{2m}_{\Omega} C^{\frac12-\delta}_{[T_{q+2},T]}L^2_x}^2.
\end{align}
By \eqref{spa-regu}, \eqref{tem-regu-1} and interpolation,  
\begin{align}\label{est-wtu-c}
    \|\wt u \|_{L^{2m}_{\Omega}C^{\frac12-\delta}_{[T_{q+2},T]}L^{2}_x}&\lesssim \|\wt u \|_{L^{2m}_{\Omega}C_{[T_{q+2},T]}L^{2}_x}^{\frac12+\delta}\|\wt u \|_{L^{2m}_{\Omega}C^{1}_{[T_{q+2},T]}L^{2}_x}^{\frac12-\delta}\lesssim  (1+T_{q+2}^{-1}) (\|v_0 \|_{L^{2m}_{\Omega}L^2_x }+\|v_0\|_{L^{4m}_{\Omega}L^2_x}^2).
\end{align}
Plugging \eqref{est-wtu-c} into \eqref{est-r0-c-1} and using Proposition~\ref{Prop-noise}, 
we lead to 
\begin{align}\label{est-r0-c-end}
    \| \mathring{R}_{0} \|_{L^{m}_{\Omega}C^{\frac12-\delta}_{[T_{q+2},T]}L^1_x} &\lesssim (1+T_{q+2}^{-2}) (\|v_0\|_{L^{2m}_{\Omega}L^2_x}+\|v_0\|_{L^{4m}_{\Omega}L^2_x}^2)^2 +(2m-1) L^2\notag\\
    &\quad + (1+T_{q+2}^{-1}) (2m-1)^{\frac12}L(\|v_0\|_{L^{2m}_{\Omega}L^2_x}+\|v_0\|_{L^{4m}_{\Omega}L^2_x}^2)\notag\\
    &\lesssim \lambda_{q}^4 (2m-1)L.
\end{align}
Similar arguments also yield
\begin{align}\label{est-r0-w-1}
\| \mathring{R}_{0} \|_{L^{m}_{\Omega}C_{[T_{q+2},T]}W^{1-\delta,1}_x} &\lesssim  \lambda_{q}^4 (2m-1)L.
\end{align}
Thus, estimate \eqref{rl1b-s} is justified 
for $q=0$. 

At last, 
by the choice of energy profile $e$ in \eqref{asp-0-energy} and \eqref{asp-e-q},
\eqref{energy-est} is satisfied at level $q=0$.

Now, for $q\geq 1$, \eqref{ul2}, \eqref{uc}, \eqref{energy-est} and \eqref{u-B-L2tx-conv} have been verified in Section~\ref{Sec-Interm-Flow},  and \eqref{rl1b} and \eqref{rl1b-s} have been verified in Section~\ref{Sec-Rey-stress}. 
Therefore, the proof of Proposition~\ref{Prop-Iterat} is complete. 
\hfill $\square$

\subsection{Global continuous energy solutions}
We prove the construction of continuous energy solutions
with prescribed initial data in Theorem~\ref{Thm-Non-SNS} 
below. 

\medskip
{\bf Proof of Theorem~\ref{Thm-Non-SNS}.} 
As explained in Section \ref{Sec-Strategy},  
for any initial datum $v_0\in L^2_\sigma$, we choose the adapted solution $\wt u$ to the $\Lambda$-NSE
as the initial solution
to the relaxed Navier-Stokes-Reynolds system \eqref{equa-nsr}
with the Reynolds stress given by \eqref{r0b}.
Note that the estimates \eqref{ul2}-\eqref{energy-est} hold at step $q=0$. 
Then, applying Proposition~\ref{Prop-Iterat} we obtain
a sequence of relaxed solutions $(u_{q},\rr_{q})$ to \eqref{equa-nsr}
which satisfy estimates \eqref{ul2}-\eqref{energy-est} for all $q\geq 0$.

Concerning the $\{\mathcal{F}_t\}$-adaptedness of relaxed solutions $(u_{q},\rr_{q})$, we start by considering the relaxed solution $(u_{0},\rr_{0})$ at the initial step.
Since $\wt u$ is adapted, by \eqref{r0b},
$(u_{0},\rr_{0})$ is adapted.
Moreover, as $\varphi_{\ell_0}$ is a one-sided temporal mollifier,
the mollified Reynolds stress
$\mathring{R}_{\ell_0}$ is adapted,
and so is $a_{(k)}$
due to \eqref{rhob}-\eqref{def-fq} and \eqref{akb}.
Then, in view of \eqref{velocity perturbation},
we infer that the velocity perturbation $w_1$ is adapted.
Thus, by \eqref{q+1 iterate} and \eqref{rucom},
$(u_{1},\rr_{1})$ is also adapted.
Arguing in an iterative way we deduce that $(u_{q},\rr_{q})$
is adapted for all $q\geq 0$.

\medskip 
Next we claim that $\{u_{q}\}$ is a Cauchy sequence in $L^{2 r} (\Omega ; C ([0, T] ; L_x^2 ) )$ and the limit solves \eqref{equ-sns} in the sense of Definition \ref{weaksolu}. 

For this purpose, using \eqref{u-B-L2tx-conv} we derive that
\begin{align}\label{interpo}
 \sum_{q \geq 0} \| u_{q+1} - u_q \|_{L^{2r}_{\Omega}C_T L^2_{ x}}
	\lesssim  &\,  \sum_{q \geq 0}   \delta_{q+1}^{\frac{1}{2}} <\infty,
\end{align}
which yields that $\{u_q\}_{q\geq 0}$ is a Cauchy sequence, and so, there exists $u \in L^{2 r} (\Omega ; C ([0, T] ; L_x^2 ) )$ such that
\begin{align}   \label{uqBq-uB-0}
	\lim_{q\rightarrow\infty} u_q = u \ \ \text{in}\ L^{2 r} (\Omega ; C ([0, T] ; L_x^2 ) ).
\end{align}
Since $\{u_q\}_{q\geq 0}$ is adapted,
so is the limit $u$. Moreover,
by the backward construction, $w_{q}(0)=0$ and $\u(0)=v_0$ for all $q\geq 0$, we then have $u(0)=\wt u(0)=v_0$.

In order to verify that $ u $ satisfies system \eqref{equ-sns}, we derive from \eqref{rl1s} that
 $\lim_{q \to \infty}  \mathring{R}_{q}  = 0 $
 in $L^r_{\Omega} C_TL^1_x$.
Moreover, by \eqref{def-ziq},
\begin{align}
\|z_{ q}-z \|_{L^2_x}=  \|\P_{>\laq^{\frac{\ve}{8}}}z  \|_{L^2_x}\leq \left\|\hat z (\xi)\right\|_{L^2(|\xi|>\laq^{\frac{\ve}{8}})}\rightarrow 0,\ \ \text{as}\ \  q\rightarrow \9,
\end{align}
hence
\begin{align}\label{con-zq}
 z_{ q} \to z \ \  \text{in} \ \  L^p_\Omega C_TL^2_x\ \  \text{for any}\ \ 1\leq p<\infty.
\end{align}
Taking into account \eqref{uqBq-uB-0}, we deduce that $(u_q, \rr_q, z_q)$ convergences in probability to $(u,0,z)$ in $
C_{[0,T]}L^2_x \times C_TL^1_x\times C_TL^2_x$.  Selecting a subsequence $\{(u_{q_k}, \rr_{q_k}, z_{q_k})\}$ if necessary we have 
\begin{align}
    (u_{q_k}, \rr_{q_k}, z_{q_k})\rightarrow (u,0,z)\quad \text{in}\quad C_{T}L^2_x\times C_TL^1_x\times C_TL^2_x,\quad \P-a.s.
\end{align}
This yields that 
\begin{align*}
&\quad\ \|(u_{q_k}+z_{q_k})\otimes (u_{q_k}+z_{q_k})- (u +z)\otimes (u+z) \|_{C_TL^1_x}\notag\\
&\leq   (  \|u_{q_k}+z_{q_k}\|_{C_TL^2_x}+ \|u+z\|_{C_TL^2_x} ) (\|u_{q_k}-u \|_{ C_TL^2_x}+\|z_{q_k}-z\|_{C_TL^2_x} )\rightarrow 0\ \  \text{as}\ q\rightarrow\infty. 
\end{align*} 
Thus,  passing to the limit $q_k\rightarrow \infty$, we infer that $u$ solves \eqref{equ-sns} in the distributional sense.

In particular, via \eqref{u-vz-shift},
$v = u + z$ solves \eqref{equa-NS} on $[0,T]$ with the initial datum $v_0$ in the sense of Definition~\ref{weaksolu}, 
as claimed. 

At last,
by \eqref{energy-est}, \eqref{uqBq-uB-0} and \eqref{con-zq}, 
we have for any $t\in [0,T]$, 
\begin{align}\label{ener-tb0}
	e(t)- \E \|v(t)\|_{L^2_x}^2   =\lim_{q \to +\infty}  e(t)-\E\|(u_q+z_{q})(t)\|_{L^2_x}^2 = 0.
\end{align}
which gives \eqref{energy}.
Therefore,
the proof of Theorem \ref{Thm-Non-SNS} is complete. 
\hfill $\square$

\medskip  
 Corollary~\ref{thm-Nonuniq-Stoch} now follows 
 essentially from Theorem \ref{Thm-Non-SNS}.

{\bf Proof of Corollary~\ref{thm-Nonuniq-Stoch}.} 
Given any initial datum $v_0\in L^2_\sigma$,   
$\mathbb{P}$-a.s., let 
$$\Omega_M:=\{ M\leq  \|v_0\|_{L^2_x}<M+1\}\in \mathcal{F}_0,
\ \ M\in \mathbb{N}. $$  
Then $\Omega=\bigcup_{M=0}^\infty \Omega_M$. 
By virtue of Theorem \ref{Thm-Non-SNS}, 
there exists an adapted solution $v^{M}$ to \eqref{equa-NS} 
corresponding to the initial datum 
$v_0 1_{\Omega_{M}}$. Let 
$$v:= \displaystyle\sum_{M=0}^\infty  v^{M}1_{\Omega_{M} }.$$
Then, $v$ is an adapted solution to \eqref{equa-NS} 
and  $v\in C([0,T];L^2_x)$, $\P$-a.s. 
Since $v(T)\in  L^2_\sigma$, $\P$-a.s., 
one can regard $v(T)$ as the new initial datum to extend the solution $v$ to the interval $[0, 2T]$. Iterating this process, 
one then obtain a global solution to \eqref{equa-NS} on $[0,\infty)$ with the initial datum $v_0$.

Next, we show the non-uniqueness of global weak solutions. 
For this purpose, 
on every $\Omega_{M}$, 
in view of Theorem \ref{Thm-Non-SNS}, 
there exist infinitely many distinct energy profiles 
$\{e_n^{M}\}_{n=1}^\infty$ 
and the corresponding solutions $\{v_n^M\}_{n=1}^\infty$ 
to \eqref{equa-NS} on $[0,T]$, 
such that $e_n^{M}\neq e_m^{M}$ for any $n\neq m$, 
$v_n^M(0) = v_0 1_{\Omega_M}$ 
and $e_n^{M}(t)=\E\|v_n^M(t)\|_{L^2_x}^2$ for all $n\geq 0$.
Then,  for every $n\geq 1$, 
$v_n= \sum_{M=0}^\infty v_n^M 1_{\Omega_M}$ 
can be extended to a global solution to \eqref{equa-NS} 
on $[0,\infty)$.  
Note that 
$\{v_n\}$ have the same initial datum $v_0$, 
but distinct energy profiles at least on $[0,T]$. 
Thus, 
we obtain infinitely many different global-in-time 
solutions to \eqref{equa-NS} 
and finish the proof.  
\hfill $\square$

\begin{remark} 
(Regular initial data: regular solutions)
\label{Rem-regular}
For regular $H^3_x$ initial data, 
the backward convex integration method 
can also construct 
solutions to \eqref{equa-hyperNSE-stocha} 
in 
a more regular class $L^{2r}_{\Omega}C_{T} H^{\beta'}_{ x} 
\cap L^{2r}_{\Omega}C^{\beta'}_{T} L^2_{ x}$
for some $\beta'>0$ 
as in the usual convex integrations. 

As a matter of fact, under the choice of parameters $\lbb_q$,  $\delta_{q}$ and $\ell_q$ given by \eqref{la}, 
following a similar manner as in 
Sections \ref{Sec-Interm-Flow}-\ref{Sec-Rey-stress} we have that the relaxed solutions $(u_{q},\rr_{q})$ satisfy \eqref{ul2}-\eqref{u-B-L2tx-conv} with \eqref{uc} replaced by 
\begin{align}\label{refine-uc1}
    \| u_q  \|_{ L^{m}_{\Omega}C_{T,x}^{1}} \lesssim \lambda_{q}^{4}(12^{q-1}\cdot 8\cdot 6mL^2)^{6\cdot 12^{q-1}}.
\end{align}
Note that, for $H^3_x$ initial data, the solution $\wt u$ to $\Lambda$-NES has the global regularity $C^1_{T,x}$, $\P$-a.s., 
and so, via a similar argument as in \eqref{verifyuc1}-\eqref{verifyuc1-m}, 
the refined estimate \eqref{refine-uc1} 
can be derived on the whole time regime. 

Hence, by the interpolation, 
for any $\beta'\in (0,\frac{\beta}{\beta+5})$,
\begin{align*} 
\sum_{q \geq 0} \| u_{q+1} - u_q \|_{L^{2r}_{\Omega}C_{T} H^{\beta'}_{ x}} &\lesssim  \sum_{q \geq 0} \| u_{q+1} - u_q \|_{L^{2r}_{\Omega}C_{T} L^2_{ x}}^{1-\beta'}\| u_{q+1} - u_q \|_{L^{2r}_{\Omega}C^1_{T,x}} ^{\beta'}\notag\\
&\lesssim  \sum_{q \geq 0}   \delta_{q+1}^{\frac{1-\beta'}{2}} \lambda_{q+1}^{5\beta'} \notag \\ 
&=\sum_{q =0}^2   \delta_{q+1}^{\frac{1-\beta'}{2}} \lambda_{q+1}^{5\beta'} + \sum_{q \geq 3}  \lambda_{q+1}^{- (1-\beta')\beta+ 5\beta' } <\infty, 
\end{align*}
where $a$ is chosen sufficiently large such that $\lambda_{q+1}\geq (12^{q}\cdot 96rL^2)^{6\cdot 12^{q}}$. Similar argument also yields
\begin{align*} 
 \sum_{q \geq 0} \| u_{q+1} - u_q \|_{L^{2r}_{\Omega}C^{\beta'}_{T} L^2_{ x}} &\lesssim  \sum_{q \geq 0} \| u_{q+1} - u_q \|_{L^{2r}_{\Omega}C_{T} L^2_{ x}}^{1-\beta'}\| u_{q+1} - u_q \|_{L^{2r}_{\Omega}C^1_{T,x}} ^{\beta'}  <\infty.
\end{align*} 
Thus, $\{u_q\}$ convergences in 
$L^{2r}_{\Omega}C_{T} H^{\beta'}_{ x} 
\cap L^{2r}_{\Omega}C^{\beta'}_{T} L^2_{ x}$, and so, 
as a limit, $u\in L^{2r}_{\Omega}C_{T} H^{\beta'}_{ x} 
\cap L^{2r}_{\Omega}C^{\beta'}_{T} L^2_{ x}$.
\end{remark}

\section{Appendix}
This Appendix collects some tools used in this paper.
Let us first recall the Aubin-Lions Lemma  from \cite{rrs16}
which is used to derive the global existence for the $\Lambda$-NSE.

\begin{lemma}[Aubin-Lions Lemma \cite{rrs16}]\label{ab-lem}
Let $\{u_n\}$ be a sequence of functions,
and
let $V:= L^2_{\sigma}\cap H^1_x$
and $V^*$ be a Banach space
and the corresponding dual space, respectively.
Assume that $p,q>1$ and
\begin{align*}
 \left\|u_n\right\|_{L^q(0, T ; V)}+\left\|\partial_t u_n\right\|_{L^p\left(0, T ; V^*\right)} <\infty \ \ \text { for all } n \in \mathbb{N}.
\end{align*}
Then there exists $u \in L^q(0, T ; L^2_\sigma)$ and a subsequence $\{u_{n_j}\}$ of $\{u_n\}$ such that
\begin{align*}
u_{n_j} \rightarrow u \quad \text { strongly in } \quad L^q(0, T ; L^2_\sigma) .
\end{align*}
\end{lemma}

Next, we present the following Geometric Lemma
that is crucial in the construction of velocity flows
in the convex integration scheme.

\begin{lemma} [Geometric Lemma, \cite{bbv20}, Lemma 4.2]
	\label{geometric lem 2}
	There exist a set $\lambda  \subset \mathbb{S}^2  \cap \mathbb{Q}^3$ that consists of vectors $k$ with associated orthonormal bases $(k, k_1, k_2)$,  $\varepsilon_u > 0$, and smooth positive functions $\gamma_{(k)}: B_{\varepsilon_u}(\Id) \to \mathbb{R}$, where $B_{\varepsilon_u}(\Id)$ is the ball of radius $\varepsilon_u$ centered at the identity in the space of $3 \times 3$ symmetric matrices,  such that for  $S \in B_{\varepsilon_u}(\Id)$ we have the following identity:
	\begin{equation}
		\label{sym}
		S = \sum_{k \in \Lambda  } \gamma_{(k)}^2(S) k_1 \otimes k_1 .
	\end{equation}
\end{lemma}

Note that there exists $N_{\Lambda} \in \mathbb{N}$ such that
\begin{equation} \label{NLambda}
	\{ N_{\Lambda} k,N_{\Lambda}k_1 , N_{\Lambda}k_2 \} \subset N_{\Lambda} \mathbb{S}^2 \cap \mathbb{Z}^3.
\end{equation}
Let $M_*$ denote the geometric constant such that
\begin{align}	\label{M bound}
	\sum_{k \in \Lambda_{u}} \norm{\gamma_{(k)}}_{C^4(B_{\varepsilon_u}(\Id))} \leq M_*.
\end{align}

We recall the definition of inverse-divergence operator $\mathcal{R} $, defined by
\begin{align}
	& (\mathcal{R}  v)^{kl} := \partial_k \Delta^{-1} v^l + \partial_l \Delta^{-1} v^k - \frac{1}{2}(\delta_{kl} + \partial_k \partial_l \Delta^{-1})\div \Delta^{-1} v, \label{operaru}
\end{align}
where $\int_{\mathbb{T}^3} v dx =0$
and $\varepsilon_{ijk}$ is the Levi-Civita tensor, $i,j,k,l \in \{1,2,3\}$.
The operator $\mathcal{R}$
is used in the construction of Reynolds stress.
It returns symmetric and trace-free matrices,
and satisfies the algebraic identities
\begin{align*}
	\div \mathcal{R} (v) = v.
\end{align*}
Moreover, $|\nabla|\mathcal{R} $ is a Calderon-Zygmund operator
and is bounded in the spaces $L^p$, $1<p<+\infty$.
See \cite{dls13} for more details.

The following stationary phase lemma is a useful tool to handle the errors of Reynolds stress.
\begin{lemma}[\cite{lt20}, Lemma 6; see also \cite{bv19b}, Lemma B.1] \label{commutator estimate1}
	Let $a \in C^{2}\left(\mathbb{T}^{3}\right)$. For all $1<p<+\infty$ we have
	$$
	\left\||\nabla|^{-1} \P_{\neq 0}\left(a \P_{\geq k} f\right)\right\|_{L^{p}\left(\mathbb{T}^{3}\right)} \lesssim k^{-1}\left\|\nabla^{2} a\right\|_{L^{\infty}\left(\mathbb{T}^{3}\right)}\|f\|_{L^{p}\left(\mathbb{T}^{3}\right)},
	$$
	holds for any smooth function $f \in L^{p}\left(\mathbb{T}^{3}\right)$.
\end{lemma}

Lemma \ref{Decorrelation1} below permits to
control the $L^2$-norm of the velocity perturbations.
\begin{lemma}[\cite{cl21}, Lemma 2.4; see also \cite{bv19b},  Lemma 3.7]   \label{Decorrelation1}
	Let $\sigma\in \mathbb{N}$ and $f,g:\mathbb{T}^3\rightarrow \R$ be smooth functions. Then for every $p\in[1,+\9]$,
	\begin{equation}\label{lpdecor}
		\big|\|fg(\sigma\cdot)\|_{L^p(\T^3)}-\|f\|_{L^p(\T^3)}\|g\|_{L^p(\T^3)} \big|\lesssim \sigma^{-\frac{1}{p}}\|f\|_{C^1(\T^3)}\|g\|_{L^p(\T^3)}.
	\end{equation}
\end{lemma}

\begin{lemma}[\cite{bms21}, Proposition 2]\label{lem-mean}
	Let $a \in C^{\infty}\left(\mathbb{T}^3 ; \mathbb{R}\right), v \in C_0^{\infty}\left(\mathbb{T}^3 ; \mathbb{R}\right)$. Then for any $r \in[1, \infty]$ and $\lambda\in \mathbb{N}$,
	$$
	\left|\int_{\mathbb{T}^3} a v(\lambda x)\d x\right| \leq \lambda^{-1} C_r\|\nabla a\|_{L^r_x}\|v\|_{L^{r'}_x}.
	$$
\end{lemma}

\medskip
\noindent{\bf Acknowledgment.} 
D. Zhang  is partially supported by 
the NSFC grants (No. 12271352, 12322108) 
and Shanghai Frontiers Science Center of Modern Analysis.

\end{document}